\documentclass{tac}
\pdfoutput=1
\usepackage[utf8]{inputenc}
\usepackage[T1]{fontenc}
\usepackage[shortlabels]{enumitem}

\usepackage{etoolbox,xparse}
\usepackage{amsmath,amssymb,mathtools}
\usepackage{dsfont}

\newcommand{\define}[1]{\textbf{#1}}
\newcommand{\pair}[1]{\langle{#1}\rangle}
\renewcommand{\vec}[1]{\underline{#1}}
\newcommand{\vecvec}[1]{\underline{\underline{#1}}}

\newcommand{\from}{\leftarrow}

\newcommand{\xto}[1]{\xrightarrow{#1}}
\newcommand{\xfrom}[1]{\xleftarrow{#1}}

\newcommand{\To}{\Rightarrow}
\newcommand{\op}{\mathrm{op}}
\newcommand{\co}{\mathrm{co}}
\DeclareMathOperator{\id}{id}
\DeclareMathOperator{\El}{El}

\newcommand{\cat}[1]{\mathsf{#1}}
\newcommand{\catV}{\mathcal{V}}
\newcommand{\Set}{\cat{Set}}
\newcommand{\FinSet}{\cat{FinSet}}
\newcommand{\CatOne}{\cat{Cat}}
\newcommand{\Fam}{\mathsf{Fam}}
\newcommand{\FamOp}{\mathsf{Fam}^*}
\newcommand{\FinFam}{\mathsf{FinFam}}

\makeatletter %
\newcommand{\xRrightarrow}[2][]{\ext@arrow 0359\Rrightarrowfill@{#1}{#2}}
\newcommand{\Rrightarrowfill@}{\arrowfill@\equiv\equiv\Rrightarrow}
\makeatother

\newcommand{\bicat}[1]{\mathbf{#1}}
\newcommand{\Cat}{\bicat{Cat}}
\newcommand{\TwoCat}{\bicat{2Cat}}
\newcommand{\Dbl}{\bicat{Dbl}}
\newcommand{\DblLax}{\bicat{Dbl}_{\mathrm{lax}}}
\newcommand{\DblColax}{\bicat{Dbl}_{\mathrm{clx}}}

\def\p#1{\mathrel{\ooalign{\hfil$\mapstochar\mkern 5mu$\hfil\cr$#1$}}}
\newcommand{\proto}{\p\rightarrow}
\newcommand{\xproto}[1]{\overset{#1}{\p\rightarrow}}
\newcommand{\proTo}{\p\Rightarrow}

\DeclareMathOperator{\res}{res}
\DeclareMathOperator{\ext}{ext}

\newcommand{\dbl}[1]{\ifstrequal{#1}{1}{\mathds{1}}{\mathbb{#1}}}
\NewDocumentCommand{\Span}{g}{\mathbb{S}\mathsf{pan}\IfNoValueTF{#1}{}{(#1)}}
\NewDocumentCommand{\Mat}{g}{\IfNoValueTF{#1}{}{{#1}\text{-}}\mathbb{M}\mathsf{at}}
\NewDocumentCommand{\Cospan}{g}{\mathbb{C}\mathsf{ospan}\IfNoValueTF{#1}{}{(#1)}}
\newcommand{\FinSpan}{\mathbb{F}\mathsf{inSpan}}
\newcommand{\FinCospan}{\mathbb{F}\mathsf{inCospan}}

\newcommand{\Sq}{\mathbb{S}\mathsf{q}}
\newcommand{\DblFam}{\mathbb{F}\mathsf{am}}
\newcommand{\DblFamOp}{\mathbb{F}\mathsf{am}^*}
\newcommand{\DblFinFam}{\mathbb{F}\mathsf{inFam}}
\newcommand{\Th}[1]{\mathbb{T}_{#1}}

\newcommand{\Eta}{\mathrm{H}}
\newcommand{\Kappa}{\mathrm{K}}
\newcommand{\Mu}{\mathrm{M}}
\newcommand{\Nu}{\mathrm{N}}

\usepackage{quiver}
\newcommand{\inlineCell}[1][small]{%
  \def\tempa{#1}%
  \inlineCellContinued
}
\newcommand{\inlineCellContinued}[9]{%
\begin{tikzcd}[ampersand replacement=\&, cramped, row sep=small, column sep=\tempa]
  {#1} \& {#2} \\
  {#3} \& {#4}
  \arrow[""{name=0, anchor=center, inner sep=0}, "{#5}", "\shortmid"{marking}, from=1-1, to=1-2]
  \arrow["{#7}"', from=1-1, to=2-1]
  \arrow["{#8}", from=1-2, to=2-2]
  \arrow[""{name=1, anchor=center, inner sep=0}, "{#6}"', "\shortmid"{marking}, from=2-1, to=2-2]
  \arrow["{#9}"{description}, draw=none, from=0, to=1]
\end{tikzcd}}

\hypersetup{bookmarksnumbered}

\newtheoremrm{construction}{Construction}

\usepackage[square]{natbib}
\title{Products in double categories, revisited}
\author{Evan Patterson}
\date{}

\begin{document}
\maketitle

\begin{abstract}
  Products in double categories, as found in cartesian double categories, are an
  elegant concept with numerous applications, yet also have a few puzzling
  aspects. In this paper, we revisit double-categorical products from an
  unbiased perspective, following up an original idea by Paré to employ a
  double-categorical analogue of the family construction, or free product
  completion. Defined in this way, double categories with finite products are
  strictly more expressive than cartesian double categories, while being
  governed by a single universal property that is no more difficult to work
  with. We develop the basic theory and examples of such products and, by
  duality, of coproducts in double categories. As an application, we introduce
  finite-product double theories, a categorification of finite-product theories
  that extends recent work by Lambert and the author on cartesian double
  theories, and we construct the virtual double category of models of a
  finite-product double theory.
\end{abstract}

\tableofcontents

\section{Introduction}

A cartesian double category is a double category with binary and nullary
products \citep{aleiferi2018}, as defined by the general theory of limits in
double categories \citep{grandis1999,grandis2019}. This means that a double
category $\dbl{D}$ is \define{cartesian} when there exist right adjoints to the
diagonal double functor $\Delta: \dbl{D} \to \dbl{D} \times \dbl{D}$ and the unique double
functor $!: \dbl{D} \to \dbl{1}$. Equivalently, a cartesian double category is a
cartesian object in the 2-category of double categories.

Cartesian double categories aim to capture through a simple universal property
what in a bicategory can only be described through far more elaborate
structures, such as cartesian bicategories \citep{carboni1987,carboni2008}.
Recent works in pure and applied category theory have sought to exploit this
feature of double categories. For example, Lambert has given axioms for ``a
double category of relations'' \citep{lambert2022} simplifying those of ``a
bicategory of relations'' \citep{carboni1987}. The author has used the dual notion
of a cocartesian double category to streamline the theory of structured cospans
\citep{patterson2023}, a framework for modeling open systems
\citep{fiadeiro2007,baez2020}. Finally, cartesian double categories are the point
of departure for ``cartesian double theories,'' a categorification of
finite-product theories whose models are categorical doctrines
\citep{lambert2024}.

Yet, despite their elegance and success in applications, there are good reasons
to doubt that cartesian double categories are the right notion of a ``double
category with finite products.'' In ordinary category theory, binary and nullary
products famously generate all finite products, hence it is standard to identify
a ``category with binary and nullary products'' with a ``category with finite
products.'' Surprisingly, this identification, rarely given a second thought,
breaks down in certain ways for double categories.

We begin with a striking disanalogy between cocartesian categories and
cocartesian double categories. The category of sets and functions plays a
distinguished role in category theory through the Yoneda lemma and the
representability arguments that the lemma enables. In developing the Yoneda
theory of double categories \citep{pare2011}, Paré has convincingly argued that
the role played in category theory by $\Set$ is played in double category theory
by $\Span$, the double category of sets, functions, and spans. For this reason,
Paré and others even call this double category $\mathbb{S}\mathsf{et}$. The
category $\Set$ is also characterized as the free cocompletion of the terminal
category; similarly, $\FinSet$ is the free cocartesian category (category with
finite coproducts) on the terminal category. Yet the free cocartesian double
category on the terminal double category is not $\FinSpan$, as one would expect
under this analogy, but the degenerate fragment of it whose proarrows are
identity spans.

There is also disharmony between cartesian double categories and cartesian
bicategories. Although Aleiferi motivates her introduction of cartesian double
categories through cartesian bicategories \citep{aleiferi2018}, hence the
parallel terminology, no formal comparison between the two concepts has been
made. One might hope that a cartesian double category has an underlying
cartesian bicategory, but that seems unlikely since a generic cartesian double
category has no means of transferring the cartesian structure from arrows to
proarrows. A cartesian \emph{equipment} does have such a means. So, a cartesian
equipment is expected to have an underlying cartesian bicategory, and, at least
in the locally posetal case, Lambert has proved this
\citep[\mbox{Proposition 3.1}]{lambert2022}. But, even supposing that this is
true in general, a puzzle remains: a cartesian equipment is defined by two
independent universal properties, and the universal property of being an
equipment has, on the face of it, nothing to do with products, except insofar as
it involves a ``mapping into'' universal property.\footnote{Even this comparison
  is arguably misleading since equipments have a dual characterization in terms
  of a ``mapping out'' universal property \citep[\mbox{Theorem 4.1}]{shulman2008},
  so are, in any analogy with limits or colimits, equally close to either.} Is
this really necessary to obtain a cartesian bicategory?

So it appears that a cartesian double category possess too little structure to
fulfill the role of a ``double category with finite products,'' whereas a
cartesian equipment possesses too much. In this paper, we explore an
intermediate conception of products in a double category, based on a
double-categorical version of the family construction. This idea was first
proposed by Robert Paré in a talk at the Category Theory conference
\citep{pare2009} but has never been systematically developed in print. It is the
aim of this paper to do so, while exhibiting its virtues and drawing connections
with the related ideas surveyed above.

The essence of the idea is easily explained. As reviewed in Section
\ref{sec:family-construction}, the free coproduct completion of a category
$\cat{C}$ is a category, often called $\Fam(\cat{C})$, whose objects are indexed
families of objects in $\cat{C}$. This logic can be turned around. Forgetting
the definition of coproducts and starting from the family construction, one can
define a category $\cat{C}$ to have coproducts just when the embedding
$\Delta: \cat{C} \to \Fam(\cat{C})$ that sends objects in $\cat{C}$ to singleton
families has a left adjoint, $\Sigma: \Fam(\cat{C}) \to \cat{C}$. Products can then be
defined by dualization. We carry out an analogous procedure for double
categories, constructing a double category of families, $\DblFam(\dbl{D})$, on a
double category $\dbl{D}$ and then defining $\dbl{D}$ to have (colax) coproducts
just when the embedding $\Delta: \dbl{D} \to \DblFam(\dbl{D})$ has a (colax) left
adjoint $\Sigma: \DblFam(\dbl{D}) \to \dbl{D}$, in the sense of Grandis and Paré's
theory of adjunctions between double categories \citep{grandis2004,grandis2019}.

In a hypothetical double category of families, objects should clearly be indexed
by sets, as in the one-dimensional case, but proarrows admit at least two
possible indexing schemes. One could insist that an indexed family of proarrows
go between indexed families of objects with the same indexing set. That choice
leads to cocartesian double categories and their infinitary analogues. However,
having adopted the unbiased perspective, enforcing equality of indexing sets
seems artificial, even uncategorical insofar as it asks for equality between
objects. A more flexible approach takes a family of proarrows to be indexed by
an arbitrary \emph{span} between the object indexing sets. Thus, given a double
category $\dbl{D}$, a proarrow in $\DblFam(\dbl{D})$ from an indexed family of
objects $(x_i)_{i \in I}$ to another $(y_j)_{j \in J}$ consists of a span
$I \xfrom{\ell} A \xto{r} J$ of sets together with a family of proarrows in
$\dbl{D}$ of the form
\begin{equation*}
  m_a: x_{\ell(a)} \proto y_{r(a)}, \qquad a \in A.
\end{equation*}
The double family construction is developed along these lines in
Section \ref{sec:dbl-family-construction}. In Section \ref{sec:coproducts}, we define
\define{colax} or \define{strong coproducts} in a double category $\dbl{D}$ to
be a colax or pseudo left adjoint to the embedding
$\Delta: \dbl{D} \to \DblFam(\dbl{D})$, and we verify that two fundamental
classes of double categories, those of spans and matrices, have strong
coproducts under appropriate assumptions.

The new definition's most immediate advantage is in giving a more flexible
notion of coproduct. Coproducts of proarrows indexed by an identity span
$I \xfrom{1} I \xto{1} I$ are the \define{parallel coproducts} familiar from
cocartesian double categories. Coproducts of proarrows with common source and
target, indexed by a span $1 \xfrom{!} I \xto{!} 1$, are \define{local
  coproducts}. These are, in particular, local coproducts in the underlying
bicategory but have a stronger universal property.

Once the notion of coproduct-preserving double functor has been established in
Section \ref{sec:preserving-coproducts}, we formulate and prove the expected
result that $\DblFam(\dbl{D})$ is the free coproduct completion of the double
category $\dbl{D}$. In particular, $\Span$ is the free coproduct completion of
the terminal double category, since we have an isomorphism of double categories
$\DblFam(\dbl{1}) \cong \Span$ essentially by construction. Similarly,
$\DblFinFam(\dbl{1}) \cong \FinSpan$ is the free completion of the terminal
double category with finite coproducts. Thus the first conceptual problem with
cocartesian double categories is resolved.

In Section \ref{sec:products}, we turn to products in double categories. Of
course, products in a double category $\dbl{D}$ can be defined as coproducts in
the opposite double category $\dbl{D}^\op$, where
$(\dbl{D}^\op)_i = (\dbl{D}_i)^\op$ for each $i=0,1$. Yet there is an
interesting asymmetry in the main examples that has no counterpart for
one-dimensional categories. Key examples of double categories, such as those of
spans and of matrices, have \emph{strong} coproducts but \emph{lax} products, a
situation first noticed by Paré in the double category of relations
\citep{pare2009}. Products in these double categories are not entirely lax:
products commute with external composition when the adjacent legs of the
indexing spans are bijections. We isolate this condition as \emph{iso-strong
  products} in Section \ref{sec:iso-strong-products} and check that double
categories of spans and of matrices have iso-strong products under the expected
hypotheses.

A less immediate, but very important, consequence of the defining universal
property is that double categories with products automatically possess certain
\emph{restriction} cells (also known as \emph{cartesian} cells). Namely, they
have restrictions along any structure arrow between products, such as
projections and diagonals. Restrictions, companions, and conjoints in double
categories with products are investigated in Section \ref{sec:restrictions},
culminating in the characterization of double categories with iso-strong finite
products as cartesian double categories that have restrictions along structure
arrows between products. We thus obtain inclusions
\begin{align*}
  \text{cartesian equipments}
    &\subset \text{double categories with iso-strong finite products} \\
    &\subset \text{cartesian double categories}
\end{align*}
and the inclusions are strict. The same sequence of inclusions holds for
precartesian equipments, double categories with lax finite products, and
precartesian double categories, where a \emph{precartesian} double category is
like a cartesian double category except that the pseudo double functors
comprising the right adjoints are replaced by lax double functors
\citep[\S{4.1}]{aleiferi2018}.

Double categories with products possess \emph{just enough} restrictions to have
underlying cartesian bicategories, allowing the structure arrows between
products to be transferred to proarrows but in general nothing more. To be more
precise, the underlying bicategory of a double category with iso-strong finite
products is cartesian. We defer the proof of this theorem to another work
\citep[\mbox{Theorem 5.7}]{patterson2024}, since it involves techniques for
transposing structure in double categories of a rather different spirit than the
present paper. Nevertheless, this result resolves the second conceptual problem
raised above. As a corollary, it follows that cartesian equipments have
underlying cartesian bicategories.

The last two sections of the paper concern a different application,
double-categorical logic, which was actually the author's original motivation to
revisit products in double categories.\footnote{In fact, double-categorical
  logic was also Paré's original motivation \citep{pare2009}, although he
  considers both a different type of a double theory and a different type of
  model than we do.} In recent work \citep{lambert2024}, Michael Lambert and the
author proposed cartesian double categories as a double-categorical approach to
doctrines, defining a \define{cartesian double theory} to be a small, strict
cartesian double category and a \define{model} of a cartesian double theory to
be a cartesian lax functor out of it. Cartesian equipments were not considered
viable as double theories since lax functors seem not to interact well with all
of the abundant structure present in an equipment. Double categories with finite
products, in the sense developed here, are a useful compromise, more expressive
than a cartesian double category yet fully controlled by a single universal
property that lax functors respect.

In Section \ref{sec:theories}, we define a \define{finite-product double theory} to be
a small, strict double category with strong finite products and a \define{model}
of a finite-product double theory to be a product-preserving lax functor out of
it. We present several examples of finite-product double theories illustrating
their extra flexibility compared to cartesian double theories. Finally, in
Section \ref{sec:models}, we construct a unital virtual double category of models of a
finite-product double theory and prove a series of simple but necessary lemmas
showing that morphisms and higher morphisms of models behave properly.

Our thesis is, in short, that the canonical notion of products in double
categories is the one investigated here, at least for double categories with
``object-like'' proarrows, such as those of spans, matrices, relations, and
profunctors. A natural question for the future is whether the same unbiased
perspective can be usefully applied to general limits and colimits in double
categories.

\paragraph{Organization} This paper naturally divides into two parts, concerning
coproducts and products. The two parts are largely, though not completely,
independent of each other. In Sections
\ref{sec:family-construction}--\ref{sec:preserving-coproducts}, we review the
family construction for ordinary categories, extend it to a double-categorical
family construction, formulate coproducts in a double category $\dbl{D}$ as a
left adjoint to the inclusion double functor
$\Delta: \dbl{D} \to \DblFam(\dbl{D})$, and then explicitly construct coproducts
in several important classes of double categories. In Sections
\ref{sec:products}--\ref{sec:models}, we dualize to study products in double
categories; distinguish between lax, iso-strong, and strong products in key
examples; and finally apply these notions to formulate finite-product double
theories and their virtual double categories of models. The reader primarily
interested in a direct description of products, or in their applications to
formal category theory, can begin at Section \ref{sec:products}, referring
occasionally to earlier sections for technical results.

\paragraph{Conventions} Double categories and double functors are assumed to be
pseudo unless otherwise stated. Every double category $\dbl{D}$ has two
\define{underlying categories}, the category $\dbl{D}_0$ of objects and arrows
and the category $\dbl{D}_1$ of proarrows and cells. A double category also has
an \define{underlying 2-category} of objects, arrows, and cells bounded by
identity proarrows, and an \define{underlying bicategory} of objects, proarrows,
and \define{globular} cells, which are bounded by identity arrows. We always
write applications of the external composition
$\odot: \dbl{D}_1 \times_{\dbl{D}_0} \dbl{D}_1 \to \dbl{D}_1$ in a double category in
diagrammatic order, and we denote the external identity by
$\id: \dbl{D}_0 \to \dbl{D}_1$.

\paragraph{Acknowledgments} I thank Michael Lambert for a series of helpful
conversations over the course of this research. I thank Brandon Shapiro for
first implanting in my mind the belief that $\Span$ should be the free coproduct
completion of the terminal double category. Finally, I am grateful to Nathanael
Arkor and the anonymous reviewer for careful and detailed feedback on the
manuscript, which has notably improved as a result.

\section{Review of the family construction}
\label{sec:family-construction}

In this expository section, we review the family construction, or free coproduct
completion, of a category, as we will develop the family construction for double
categories in analogy with it. Much is known about the category of families in a
given category, which enjoys surprisingly many properties besides simply being
the free coproduct completion. Notable references for these properties include
the papers \citep{carboni1995,hu1995,adamek2020}. The category of families
apparently originates with Bénabou's study of the foundations of naive category
theory \citep[\S{3.3}]{benabou1985}.

The family construction can be defined as the Grothendieck construction of a
restricted representable functor. Given the simplicity of the category of
families, this may seem a circuitous route but it has the advantage of
immediately exhibiting the fibered structure of families. More to the point, it
will motivate an analogous approach to families in double categories.

\begin{construction}[Families] \label{def:fam}
  The \define{category of families} in a category $\cat{C}$, denoted
  $\Fam(\cat{C})$, is the Grothendieck construction of the functor
  $\Cat(-, \cat{C})|_{\Set}: \Set^\op \to \Cat$:
  \begin{equation*}
    \Fam(\cat{C}) \coloneqq \int^{I \in \Set} \Cat(I, \cat{C}).
  \end{equation*}
\end{construction}

In the definition, the sets $I$ are regarded as discrete categories by a common
abuse of notation. Thus, the category $\Fam(\cat{C})$ has
\begin{itemize}
  \item as objects $(I, \vec{x})$, an \define{indexing set} $I$ together with an
    $I$-indexed family $\vec{x}: I \to \cat{C}$ of objects in $\cat{C}$, whose
    elements are denoted $x_i \coloneqq \vec{x}(i)$;
  \item as morphisms $(f_0, f): (I, \vec{x}) \to (J, \vec{y})$, a function
    $f_0: I \to J$ between indexing sets together with an $I$-indexed family
    $f: \vec{x} \To \vec{y} \circ f_0$ of morphisms in $\cat{C}$, having
    components of the form $f_i: x_i \to y_{f_0(i)}$.
\end{itemize}
By the fundamental relationship between indexed categories and fibrations, the
canonical projection $\Fam(\cat{C}) \to \Set$ sending a family of objects to its
indexing set is a fibration.

There is also a \define{category of finite families} in $\cat{C}$, denoted
$\FinFam(\cat{C})$. It is the full subcategory of $\Fam(\cat{C})$ spanned by
families of objects indexed by \emph{finite} sets; alternatively, it is the
Grothendieck construction of the functor
$\Cat(-,\cat{C})|_{\FinSet}: \FinSet^\op \to \Cat$. Everything that we will say
about the family construction and coproducts remains true when ``family'' is
replaced by ``finite family'' and ``coproduct'' by ``finite coproduct.'' Often
we will not bother to say so explicitly.

It would be conventional at this point to observe that $\Fam(\cat{C})$ is the
free completion of $\cat{C}$ to a category with coproducts, but this presupposes
the definition of coproducts, whereas we intend to use a double family
construction to \emph{define} coproducts in a double category. Mirroring this
logic, we recall how to recover the usual definition of a coproduct in a
category from the family construction. Denote by
$\Delta \coloneqq \Delta_{\cat{C}}: \cat{C} \to \Fam({\cat{C}})$ the functor
that sends each object $x$ in $\cat{C}$ to the singleton family $(1,x)$, indexed by
the terminal set.

\begin{proposition}[Coproducts as left adjoints]
  \label{prop:coproducts-as-adjoint}
  A category $\cat{C}$ has all (small) coproducts if and only if the functor
  $\Delta: \cat{C} \to \Fam(\cat{C})$ has a left adjoint:
  \begin{equation*}
    \begin{tikzcd}
      {\cat{C}} & {\Fam(\cat{C})}
      \arrow[""{name=0, anchor=center, inner sep=0}, "\Sigma"', curve={height=18pt}, from=1-2, to=1-1]
      \arrow[""{name=1, anchor=center, inner sep=0}, "\Delta"', curve={height=18pt}, from=1-1, to=1-2]
      \arrow["\dashv"{anchor=center, rotate=-90}, draw=none, from=0, to=1]
    \end{tikzcd}.
  \end{equation*}
\end{proposition}
\begin{proof}
  The most expeditious proof uses the equivalence between adjunctions and
  universal arrows; see, for example, \citep[\S{IV.1}]{maclane1998} or
  \citep[\S{1.5.1}]{grandis2019}. Thus, the existence of a functor
  $\Sigma: \Fam(\cat{C}) \to \cat{C}$ along with an adjunction
  $\Sigma \dashv \Delta$ is equivalent to the existence merely of, for every
  family of objects $(I, \vec{x})$ in $\cat{C}$, a universal arrow from
  $(I,\vec{x})$ to the functor $\Delta$. Such a universal arrow consists of an
  object $\Sigma \vec{x} \coloneqq \Sigma(I, \vec{x})$ in $\cat{C}$ along with a
  morphism $(\iota_0, \iota): (I,\vec{x}) \to \Delta \Sigma \vec{x}$ in
  $\Fam(\cat{C})$. It satisfies the universal property that, for any object
  $y \in \cat{C}$ and morphism $(f_0, f): (I,\vec{x}) \to \Delta y$ in
  $\Fam(\cat{C})$, there exists a unique morphism $h: \Sigma\vec{x} \to y$ in
  $\cat{C}$ making the triangle commute:
  \begin{equation*}
    \begin{tikzcd}
      {(I, \vec{x})} & {\Delta\Sigma\vec{x}} \\
      & {\Delta y}
      \arrow["{\Delta h}", from=1-2, to=2-2]
      \arrow["{(\iota_0, \iota)}", from=1-1, to=1-2]
      \arrow["{(f_0, f)}"', from=1-1, to=2-2]
    \end{tikzcd}.
  \end{equation*}
  Now, since $f_0$ and $\iota_0$ must both be the unique map $!: I \to 1$, we
  recover the usual definition of a coproduct: there is an object
  $\Sigma \vec{x}$ in $\cat{C}$, conventionally denoted $\sum_{i \in I} x_i$,
  along with \define{coprojection} morphisms $\iota_i: x_i \to \Sigma\vec{x}$
  for each $i \in I$. These satisfy the universal property that for any object
  $y \in \cat{C}$ and family of morphisms $(f_i: x_i \to y)_{i \in I}$, there
  exists a unique morphism $h: \Sigma\vec{x} \to y$ making all the triangles
  commute:
  \begin{equation*}
    \begin{tikzcd}
      {x_i} & {\Sigma \vec{x}} \\
      & y
      \arrow["{\iota_i}", from=1-1, to=1-2]
      \arrow["{f_i}"', from=1-1, to=2-2]
      \arrow["h", from=1-2, to=2-2]
    \end{tikzcd},
    \qquad i \in I.
  \end{equation*}
\end{proof}

Functors that preserve coproducts can also be characterized using the family
construction. Notice that the assignment $\Fam: \Cat \to \Cat$ is a functor
(even a 2-functor, although we won't need that). Specifically, given a functor
$F: \cat{C} \to \cat{D}$, the functor $\Fam(F): \Fam(\cat{C}) \to \Fam(\cat{D})$
results by applying the Grothendieck construction, which is functorial
\citep[\S{6.2}]{peschke2020}, to the natural transformation
\begin{equation*}
  \Cat(-,F)|_\Set: \Cat(-,\cat{C})|_\Set \To \Cat(-,\cat{D})|_\Set: \Set^\op \to \Cat.
\end{equation*}
In concrete terms, $\Fam(F)$ acts on indexed families by post-composition: for a
family of objects $(I, \vec{x})$ in $\cat{C}$, we have
$\Fam(F)(I, \vec{x}) = (I, F\vec{x})$, so that $(F \vec{x})_i = F(x_i)$ for each
$i \in I$, and similarly for families of morphisms. To avoid notational clutter,
we will often simply write $F \vec{x}$ for this action.

\begin{proposition}[Coproduct-preserving functors]
  \label{prop:preserve-coproducts}
  A functor $F: \cat{C} \to \cat{D}$ between categories with coproducts
  preserves all coproducts if and only if the morphism of adjunctions
  \begin{equation*}
    \begin{tikzcd}
      {\Fam(\cat{C})} & {\cat{C}} \\
      {\Fam(\cat{D})} & {\cat{D}}
      \arrow[""{name=0, anchor=center, inner sep=0}, "{\Sigma \dashv \Delta}", "\shortmid"{marking}, from=1-1, to=1-2]
      \arrow["F", from=1-2, to=2-2]
      \arrow["{\Fam(F)}"', from=1-1, to=2-1]
      \arrow[""{name=1, anchor=center, inner sep=0}, "{\Sigma \dashv \Delta}"', "\shortmid"{marking}, from=2-1, to=2-2]
      \arrow["{(\Phi,1)}"{description}, draw=none, from=0, to=1]
    \end{tikzcd}
  \end{equation*}
  defined by the pair of functors $(\Fam(F), F)$ is strong.
\end{proposition}
\begin{proof}
  A \emph{morphism of adjunctions} is a cell in the double category of
  adjunctions \citep[\S{3.1.6}]{grandis2019}. Less standardly, the cell being
  \emph{strong} means that the identity transformation on the left
  \begin{equation*}
    \begin{tikzcd}
      {\Fam(\cat{C})} & {\cat{C}} \\
      {\Fam(\cat{D})} & {\cat{D}}
      \arrow["F", from=1-2, to=2-2]
      \arrow["{\Delta_{\cat{C}}}"', from=1-2, to=1-1]
      \arrow["{\Fam(F)}"', from=1-1, to=2-1]
      \arrow["{\Delta_{\cat{D}}}", from=2-2, to=2-1]
      \arrow["1", shorten <=5pt, shorten >=8pt, Rightarrow, from=1-1, to=2-2]
    \end{tikzcd}
    \qquad\leftrightsquigarrow\qquad
    \begin{tikzcd}
      {\Fam(\cat{C})} & {\cat{C}} \\
      {\Fam(\cat{D})} & {\cat{D}}
      \arrow["F", from=1-2, to=2-2]
      \arrow["{\Sigma_{\cat{C}}}", from=1-1, to=1-2]
      \arrow["{\Fam(F)}"', from=1-1, to=2-1]
      \arrow["{\Sigma_{\cat{D}}}"', from=2-1, to=2-2]
      \arrow["\Phi", shorten <=5pt, shorten >=5pt, Rightarrow, from=2-1, to=1-2]
    \end{tikzcd}
  \end{equation*}
  has a mate cell $\Phi$ as shown on the right that is a natural isomorphism.
  Unwinding the definitions, the component of $\Phi$ at a family $(I,\vec{x})$
  is the canonical comparison
  \begin{equation*} \textstyle
    \sum_{i \in I} Fx_i = \Sigma F\vec{x} \xto{\Phi_{\vec{x}}}
      F \Sigma\vec{x} = F(\sum_{i \in I} x_i)
  \end{equation*}
  provided by the universal property of the coproduct in $\cat{D}$. By
  definition, the functor $F$ preserves the coproduct of a family $\vec{x}$ if
  and only if the component $\Phi_{\vec{x}}$ is an isomorphism.
\end{proof}

We now recall the central property of the family construction.

\begin{theorem}[Free coproduct completion] \label{thm:free-coproduct-completion}
  For any category $\cat{C}$, the category $\Fam(\cat{C})$ is the free coproduct
  completion of $\cat{C}$.
\end{theorem}
\begin{proof}
  Given a family of objects $(I, (\vec{U}, \vecvec{x}))$ in $\Fam(\cat{C})$,
  i.e., an object of $\Fam(\Fam(\cat{C}))$, a coproduct of it in $\Fam(\cat{C})$
  is
  \begin{equation*}
    \Sigma(\vec{U},\vecvec{x}) \coloneqq (\sqcup \vec{U}, \vec{x}),
    \qquad\text{where}\qquad
    \sqcup \vec{U} \coloneqq \bigsqcup_{i \in I} U_i,
    \qquad
    \vec{x} \coloneqq [\vec{x}_i]_{i \in I}: \sqcup \vec{U} \to \cat{C}.
  \end{equation*}
  In other words, $\sqcup \vec{U}$ is the disjoint union or coproduct of the
  indexing sets and $\vec{x}$ is the copairing of the indexed families given by
  the universal property of the coproduct in $\Set$. The coprojections in
  $\Fam(\cat{C})$ are the morphisms
  \begin{equation*}
    (\iota_i, 1_{\vec{x}_i}): (U_i, \vec{x}_i) \to (\sqcup U, \vec{x}), \qquad i \in I,
  \end{equation*}
  defined by the inclusions $\iota_i: U_i \to \sqcup \vec{U}$ in $\Set$ and the
  family of identity morphisms $(1_{\vec{x}_i})_u: x_{i,u} \to x_{\iota_i(u)}$
  in $\cat{C}$, for each $u \in U_i$. The universal property of the coproduct
  in $\Fam(\cat{C})$ follows directly from that of the coproduct in $\Set$.

  So the functor $\Delta: \cat{C} \to \Fam(\cat{C})$ embeds $\cat{C}$ into a
  category with coproducts. That $\Fam(\cat{C})$ is the \emph{free} coproduct
  completion of $\cat{C}$ is the universal property that for any other functor
  $F: \cat{C} \to \cat{D}$ into a category $\cat{D}$ with coproducts, there
  exists a coproduct-preserving functor $\hat F: \Fam(\cat{C}) \to \cat{D}$
  making the triangle
  \begin{equation*}
    \begin{tikzcd}
      & {\Fam(\cat{C})} \\
      {\cat{C}} & {\cat{D}}
      \arrow["{\hat{F}}", dashed, from=1-2, to=2-2]
      \arrow["\Delta", from=2-1, to=1-2]
      \arrow["F"', from=2-1, to=2-2]
    \end{tikzcd}
  \end{equation*}
  commute, and $\hat F$ is unique up to natural isomorphism. Indeed, a functor
  $\hat F: \Fam(\cat{C}) \to \cat{D}$ making the triangle commute is uniquely
  determined on singleton families and maps between them by $F$. But an
  arbitrary family of objects in $\cat{C}$ is a coproduct of singleton families,
  so if $\hat{F}$ is to preserve coproducts, it must send the family to a
  coproduct of fixed objects in $\cat{D}$. Since $\cat{D}$ is assumed to have
  coproducts, we can make such a choice of $\hat{F}$, and since coproducts are
  unique up to unique isomorphism commuting with coprojections, any such choice
  of $\hat{F}$ is unique up to unique natural isomorphism preserving the
  triangle formed with $\Delta$ and $F$.
\end{proof}

\section{The family construction for double categories}
\label{sec:dbl-family-construction}

As reviewed in the previous section, families in a category can be defined as
the Grothendieck construction of the restriction of a representable functor. We
will now define families in a double category analogously using the double
Grothendieck construction, a tool recently introduced by Cruttwell, Lambert,
Pronk, and Szyld in their study of double fibrations \citep{cruttwell2022}. The
impatient reader may wish to skip to the concrete description of the double
family construction following Construction \ref{def:dbl-fam}.

Besides the double Grothendieck construction, we will need a bit of the
technology of \emph{double 2-categories} developed by Cruttwell et al. A
\define{double 2-category} is a pseudocategory in $\TwoCat$, the 2-category of
2-categories, 2-functors, and 2-natural transformations \citep[Definition
3.2]{cruttwell2022}. So, a double 2-category $\dbl{E}$ is underlied by a
2-category $\dbl{E}_0$ of objects, arrows, and 2-cells between arrows as well as
a 2-category $\dbl{E}_1$ of proarrows, cells, and 2-cells between cells. The
fundamental example of a double 2-category is nothing other than the receiving
object for the lax double pseudofunctors that are the input to the double
Grothendieck construction.

\begin{example}[Spans of categories] \label{ex:span-cat}
  As a particular case of \citep[Construction 3.6]{cruttwell2022}, the double
  category of spans of categories upgrades to a double 2-category $\Span{\Cat}$.
  Its 2-category of objects is simply $\Span{\Cat}_0 = \Cat$, the 2-category of
  categories. More interestingly, its 2-category of morphisms $\Span{\Cat}_1$
  has spans of categories as objects; has maps of spans as morphisms; and has as
  2-cells
  \begin{equation*}
    \begin{tikzcd}
      {\cat{X}} & {\cat{S}} & {\cat{Y}} \\
      {\cat{W}} & {\cat{T}} & {\cat{Z}}
      \arrow["L"', from=1-2, to=1-1]
      \arrow["R", from=1-2, to=1-3]
      \arrow["L", from=2-2, to=2-1]
      \arrow["R"', from=2-2, to=2-3]
      \arrow["H"', from=1-2, to=2-2]
      \arrow["F"', from=1-1, to=2-1]
      \arrow["G", from=1-3, to=2-3]
    \end{tikzcd}
    \quad\To\quad
    \begin{tikzcd}
      {\cat{X}} & {\cat{S}} & {\cat{Y}} \\
      {\cat{W}} & {\cat{T}} & {\cat{Z}}
      \arrow["L"', from=1-2, to=1-1]
      \arrow["R", from=1-2, to=1-3]
      \arrow["L", from=2-2, to=2-1]
      \arrow["R"', from=2-2, to=2-3]
      \arrow["{H'}"', from=1-2, to=2-2]
      \arrow["{F'}"', from=1-1, to=2-1]
      \arrow["{G'}", from=1-3, to=2-3]
    \end{tikzcd},
  \end{equation*}
  triples of natural transformations $\alpha: F \To F'$, $\beta: G \To G'$, and
  $\sigma: H \To H'$ satisfying the two ``cylinder conditions''
  \begin{equation*}
    \begin{tikzcd}
      {\cat{X}} & {\cat{S}} \\
      {\cat{W}} & {\cat{T}}
      \arrow["L"', from=1-2, to=1-1]
      \arrow["L", from=2-2, to=2-1]
      \arrow["{H'}", curve={height=-12pt}, from=1-2, to=2-2]
      \arrow[""{name=0, anchor=center, inner sep=0}, "F"', curve={height=12pt}, from=1-1, to=2-1]
      \arrow[""{name=1, anchor=center, inner sep=0}, "{F'}", curve={height=-12pt}, from=1-1, to=2-1]
      \arrow["\alpha", shorten <=5pt, shorten >=5pt, Rightarrow, from=0, to=1]
    \end{tikzcd}
    =
    \begin{tikzcd}
      {\cat{X}} & {\cat{S}} \\
      {\cat{W}} & {\cat{T}}
      \arrow["L"', from=1-2, to=1-1]
      \arrow["L", from=2-2, to=2-1]
      \arrow[""{name=0, anchor=center, inner sep=0}, "H"', curve={height=12pt}, from=1-2, to=2-2]
      \arrow["F"', curve={height=12pt}, from=1-1, to=2-1]
      \arrow[""{name=1, anchor=center, inner sep=0}, "{H'}", curve={height=-12pt}, from=1-2, to=2-2]
      \arrow["\sigma", shorten <=5pt, shorten >=5pt, Rightarrow, from=0, to=1]
    \end{tikzcd}
    \quad\text{and}\quad
    \begin{tikzcd}
      {\cat{S}} & {\cat{Y}} \\
      {\cat{T}} & {\cat{Z}}
      \arrow["R", from=1-1, to=1-2]
      \arrow["R"', from=2-1, to=2-2]
      \arrow[""{name=0, anchor=center, inner sep=0}, "H"', curve={height=12pt}, from=1-1, to=2-1]
      \arrow["{G'}", curve={height=-12pt}, from=1-2, to=2-2]
      \arrow[""{name=1, anchor=center, inner sep=0}, "{H'}", curve={height=-12pt}, from=1-1, to=2-1]
      \arrow["\sigma", shorten <=5pt, shorten >=5pt, Rightarrow, from=0, to=1]
    \end{tikzcd}
    =
    \begin{tikzcd}
      {\cat{S}} & {\cat{Y}} \\
      {\cat{T}} & {\cat{Z}}
      \arrow["R", from=1-1, to=1-2]
      \arrow["H"', curve={height=12pt}, from=1-1, to=2-1]
      \arrow[""{name=0, anchor=center, inner sep=0}, "G"', curve={height=12pt}, from=1-2, to=2-2]
      \arrow[""{name=1, anchor=center, inner sep=0}, "{G'}", curve={height=-12pt}, from=1-2, to=2-2]
      \arrow["R"', from=2-1, to=2-2]
      \arrow["\beta", shorten <=5pt, shorten >=5pt, Rightarrow, from=0, to=1]
    \end{tikzcd}.
  \end{equation*}
  The source and target 2-functors $\Span{\Cat}_1 \rightrightarrows \Cat$ are
  the obvious ones that extract the left and right feet, legs, and cells between
  legs. External composition is by 2-pullback in $\Cat$.
\end{example}

Any double category $\dbl{D}$ has an underlying span of categories
$\dbl{D}_0 \xfrom{s} \dbl{D}_1 \xto{t} \dbl{D}_0$, an object of $\Span{\Cat}$.
Any span of sets can be regarded as a span of discrete categories, hence also as
an object of $\Span{\Cat}$. These observations motivate the following
construction of a doubly indexed category. We will also need some special
notation for elements of spans of sets, inspired by
\mbox{\citep[\S{3.7}]{pare2011}}. For any span $I \xfrom{\ell} A \xto{r} J$ in
$\Set$, we will generally label the left and right legs as $\ell$ and $r$ and
then conflate the apex of the span with the span itself, writing
$A: I \proto J$. We then write $a: i \proto j$ for an element $a \in A$ such
that $\ell(a) = i$ and $r(a) = j$.

\begin{construction}[Double indexing of a double category]
  \label{def:dbl-indexing}
  Given a double category $\dbl{D}$, we construct a lax double functor
  suggestively denoted
  \begin{equation*}
    F = \Span{\Cat}(-, \dbl{D})|_{\Span}: \Span^\op \to \Span{\Cat}.
  \end{equation*}
  The underlying functor
  $F_0 \coloneqq \Cat(-,\dbl{D}_0)|_\Set: \Set^\op \to \Cat$ is that used
  previously in Construction \ref{def:fam}. It sends a set $I$ to the functor category
  $\Cat(I,\dbl{D}_0)$ having as objects, $I$-indexed families
  $\vec{x}: I \to \dbl{D}_0$ of objects in $\dbl{D}$ and as morphisms
  $\vec{x} \to \vec{x}'$, transformations $f: \vec{x} \To \vec{x}'$. The second
  underlying functor is
  \begin{equation*}
    F_1 \coloneqq \Span{\Cat}_1(-, \dbl{D})|_{\Span_1}: \Span_1^\op \to \Span{\Cat}_1.
  \end{equation*}
  It sends a span of sets $I \xfrom{\ell} A \xto{r} J$ to the span of categories
  whose apex is the category
  \begin{equation*}
    \Span{\Cat}_1(I \xfrom{\ell} A \xto{r} J,\
      \dbl{D}_0 \xfrom{s} \dbl{D}_1 \xto{t} \dbl{D}_0)
  \end{equation*}
  from Example \ref{ex:span-cat}. Thus, the objects of this category are indexed
  families of objects and proarrows in $\dbl{D}$ of the form
  \begin{equation*}
    \begin{tikzcd}
      I & A & J \\
      {\dbl{D}_0} & {\dbl{D}_1} & {\dbl{D}_0}
      \arrow["{\vec{x}}"', from=1-1, to=2-1]
      \arrow["s", from=2-2, to=2-1]
      \arrow["\ell"', from=1-2, to=1-1]
      \arrow["t"', from=2-2, to=2-3]
      \arrow["r", from=1-2, to=1-3]
      \arrow["{\vec{y}}", from=1-3, to=2-3]
      \arrow["{\vec{m}}"', from=1-2, to=2-2]
    \end{tikzcd},
  \end{equation*}
  which we abbreviate as $\vec{m}: \vec{x} \proto \vec{y}$; and its morphisms
  from $\vec{m}: \vec{x} \proto \vec{y}$ to $\vec{m}': \vec{x}' \proto \vec{y}'$
  are families of arrows $f: \vec{x} \To \vec{x}'$ and $g: \vec{y} \To \vec{y}'$
  together with a family of cells $\alpha: \vec{m} \To \vec{m}'$ of form
  \begin{equation*}
    \begin{tikzcd}
      {x_i} & {y_j} \\
      {x'_{f_0(i)}} & {y'_{g_0(j)}}
      \arrow[""{name=0, anchor=center, inner sep=0}, "{m_a}"{inner sep=.8ex}, "\shortmid"{marking}, from=1-1, to=1-2]
      \arrow["{f_i}"', from=1-1, to=2-1]
      \arrow["{g_j}", from=1-2, to=2-2]
      \arrow[""{name=1, anchor=center, inner sep=0}, "{m'_{\alpha_0(a)}}"'{inner sep=.8ex}, "\shortmid"{marking}, from=2-1, to=2-2]
      \arrow["{\alpha_a}"{description}, draw=none, from=0, to=1]
    \end{tikzcd},
    \qquad (i \xproto{a} j): (I \xproto{A} J).
  \end{equation*}
  The left and right legs of this span of categories are the obvious ones.

  To complete the construction, the laxators and unitors of the lax double
  functor $F$ are furnished by the external composition and identities in
  $\dbl{D}$. Given spans of sets $A: I \proto J$ and $B: J \proto K$, the
  laxator $F_{A,B}$ is defined by the functor
  $FA \times_{FJ} FB \to F(A \times_J B)$ that sends a pair of proarrow families
  $\vec{m}: \vec{x} \proto \vec{y}$ and $\vec{n}: \vec{y} \proto \vec{z}$ to the
  family $\vec{m} \odot \vec{n}: \vec{x} \proto \vec{z}$ with components
  \begin{equation}
    \label{eq:proarrow-family-composition}
    (\vec{m} \odot \vec{n})(a, b) \coloneqq m_a \odot n_b: x_i \proto z_k,
    \qquad (i \xproto{a} j \xproto{b} k): (I \xproto{A} J \xproto{B} K).
  \end{equation}
  and similarly for families of cells. Given a set $I$, the unitor
  $F_I: \id_{FI} \to F(\id_I)$ is defined by the functor sending each
  $I$-indexed object family $\vec{x}$ to the proarrow family
  $\id_{\vec{x}}: \vec{x} \to \vec{x}$ with components
  $\id_{x_i}: x_i \proto x_i$ for $i \in I$, and similarly for families of
  arrows.

  When $\dbl{D}$ is a \emph{strict} double category, we obtain a genuine lax
  double functor, where the associativity and unitality axioms for the laxators
  and unitors are immediate from the corresponding axioms for the external
  composition and identity in $\dbl{D}$. In general, though, what we have is a
  lax double \emph{pseudo}functor \citep[Definition 3.12]{cruttwell2022}, whose
  associativity and unitality modifications are given by the associators and
  unitors of $\dbl{D}$. The coherence axioms for these modifications correspond
  to the coherence axioms for the pseudo double category. In either case, the
  naturality axioms for the laxators and unitors amount to the statements that
  reindexing commutes with composing indexed families of proarrows.
\end{construction}

\begin{construction}[Double families] \label{def:dbl-fam}
  The \define{double category of families} in a double category $\dbl{D}$,
  denoted $\DblFam(\dbl{D})$, is the double Grothendieck construction of the lax
  double pseudofunctor from Construction \ref{def:dbl-indexing}:
  \begin{equation*}
    \DblFam(\dbl{D}) \coloneqq \int^{S \in \Span} \Span{\Cat}(S, \dbl{D}).
  \end{equation*}
\end{construction}

The \define{double category of finite families} in $\dbl{D}$, denoted
$\DblFinFam(\dbl{D})$, is defined similarly by replacing $\Span$ with
$\FinSpan$, the double category of spans of finite sets.

By construction, the canonical projection $\DblFam(\dbl{D}) \to \Span$, a strict
double functor, is a double fibration. In fact, if $\Sq(\cat{C})$ denotes the
double category of commutative squares in a category $\cat{C}$, then
$\DblFam(\Sq(\cat{C}))$ is the double category denoted $\DblFam(\cat{C})$ and
called the ``double family fibration'' by Cruttwell et al \citep[Examples 2.29
and 3.19]{cruttwell2022}. Any double category of families is also an example of
a ``double category of decorated spans'' in the generalized sense of
\citep[\S{3}]{patterson2023}.

It is indispensable to have a fully concrete description of the double category
of families. Given a double category $\dbl{D}$, its double category of families,
$\DblFam(\dbl{D})$, has
\begin{itemize}
  \item as objects, an \define{indexing set} $I$ together with an $I$-indexed
    family $\vec{x}: I \to \dbl{D}_0$, consisting of objects $x_i$ in $\dbl{D}$
    for $i \in I$;
  \item as arrows $(I, \vec{x}) \to (J, \vec{y})$, a function $f_0: I \to J$ together
    with an $I$-indexed family $f$ of arrows in $\dbl{D}$ of the form
    $f_i: x_i \to y_{f_0(i)}$ for $i \in I$;
  \item as proarrows $(I, \vec{x}) \proto (J, \vec{y})$, an \define{indexing
    span} $I \xfrom{\ell} A \xto{r} J$ together with an $A$-indexed family
    $\vec{m}: A \to \dbl{D}_1$, consisting of proarrows in $\dbl{D}$ of the form
    $m_a: x_i \proto y_j$ for each $a: i \proto j$;
  \item as cells
    $\inlineCell{(I,\vec{x})}{(J,\vec{y})}{(K,\vec{w})}{(L,\vec{z})}{(A,\vec{m})}{(B,\vec{n})}{(f_0,f)}{(g_0,g)}{}$,
    a map of spans $(f_0,\alpha_0,g_0)$ as shown on the left below together with
    an $A$-indexed family $\alpha$ of cells in $\dbl{D}$ of the form on the
    right:
    \begin{equation*}
      \begin{tikzcd}
        I & A & J \\
        K & B & L
        \arrow["{f_0}"', from=1-1, to=2-1]
        \arrow["{g_0}", from=1-3, to=2-3]
        \arrow["\ell"', from=1-2, to=1-1]
        \arrow["\ell", from=2-2, to=2-1]
        \arrow["r"', from=2-2, to=2-3]
        \arrow["{\alpha_0}"', from=1-2, to=2-2]
        \arrow["r", from=1-2, to=1-3]
      \end{tikzcd},
      \hspace{1in}
      \begin{tikzcd}
        {x_i} & {y_j} \\
        {w_{f_0(i)}} & {z_{g_0(j)}}
        \arrow["{f_i}"', from=1-1, to=2-1]
        \arrow["{g_j}", from=1-2, to=2-2]
        \arrow[""{name=0, anchor=center, inner sep=0}, "{m_a}", "\shortmid"{marking}, from=1-1, to=1-2]
        \arrow[""{name=1, anchor=center, inner sep=0}, "{n_{\alpha_0(a)}}"', "\shortmid"{marking}, from=2-1, to=2-2]
        \arrow["{\alpha_a}"{description}, draw=none, from=0, to=1]
      \end{tikzcd},
      \quad (i \xproto{a} j) : (I \xproto{A} J).
    \end{equation*}
\end{itemize}
To compose proarrows in $\DblFam(\dbl{D})$, first compose the indexing spans by
pullback as usual, then compose the families of proarrows componentwise in
$\dbl{D}$ as defined in Equation \eqref{eq:proarrow-family-composition}.

The definition of the double category of families is lent further support by the
several senses in which the 1-categorical and double-categorical family
constructions commute. First of all, we have the expected compatibility
$\DblFam(\dbl{D})_0 = \Fam(\dbl{D}_0)$ for any double category $\dbl{D}$. Less
immediately, we also have:

\begin{proposition}[Families of spans]
  Let $\cat{S}$ be a category with pullbacks. Then there is an isomorphism of
  double categories
  \begin{equation*}
    \DblFam(\Span{\cat{S}}) \cong \Span{\Fam(\cat{S})}.
  \end{equation*}
\end{proposition}
\begin{proof}
  As observed in \citep[Remark 2.3]{adamek2020}, when the category $\cat{S}$ has
  pullbacks, so does $\Fam(\cat{S})$. The pullback of a cospan
  $(I,\vec{x}) \xto{(f_0,f)} (K,\vec{z}) \xfrom{(g_0,g)} (J,\vec{y})$ in
  $\Fam(\cat{S})$ is computed by, first, taking the pullback of the indexing
  sets
  \begin{equation*}
    \begin{tikzcd}[column sep=small]
      & {I \times_K J} \\
      I && J \\
      & K
      \arrow["{f_0}"', from=2-1, to=3-2]
      \arrow["{g_0}", from=2-3, to=3-2]
      \arrow["{p_0}"', from=1-2, to=2-1]
      \arrow["{q_0}", from=1-2, to=2-3]
      \arrow["\lrcorner"{anchor=center, pos=0.125, rotate=-45}, draw=none, from=1-2, to=3-2]
    \end{tikzcd}
  \end{equation*}
  and then, for each pair $(i,j) \in I \times_K J$, taking the pullback in
  $\cat{S}$
  \begin{equation*}
    \begin{tikzcd}[column sep=0]
      & {w_{(i,j)}} \\
      {x_i} && {y_j} \\
      & {z_{f_0(i)}=z_{g_0(j)}}
      \arrow["{f_i}"', from=2-1, to=3-2]
      \arrow["{g_j}", from=2-3, to=3-2]
      \arrow["{p_{(i,j)}}"', from=1-2, to=2-1]
      \arrow["{q_{(i,j)}}", from=1-2, to=2-3]
      \arrow["\lrcorner"{anchor=center, pos=0.125, rotate=-45}, draw=none, from=1-2, to=3-2]
    \end{tikzcd}.
  \end{equation*}
  The span
  $(I,\vec{x}) \xfrom{(p_0,p)} (I \times_K J, \vec{w}) \xto{(q_0,q)} (J, \vec{y})$
  is then a pullback in $\Fam(\cat{S})$. This proves that the double category
  $\Span{\Fam(\cat{S})}$ is well defined. Comparing with the external
  composition in Construction \ref{def:dbl-fam}, it also shows that not only are the
  proarrows in $\DblFam(\Span{\cat{S}})$ and $\Span{\Fam(\cat{S})}$ the same,
  they have the same composition law. The remaining identifications are
  straightforward.
\end{proof}

As a corollary, we obtain an isomorphism that is also easy to see directly.

\begin{corollary} \label{cor:span-via-fam}
  Denoting the terminal double category by $\dbl{1}$, there is an isomorphism
  \begin{equation*}
    \DblFam(\dbl{1}) \cong \Span.
  \end{equation*}
\end{corollary}
\begin{proof}
  Taking $\cat{S} = \cat{1}$ to be the terminal category, we have
  $\Span{\cat{1}} \cong \dbl{1}$ and $\Fam(1) \cong \Set$.
\end{proof}

It will be important later to know that $\DblFam(\dbl{D})$ inherits extension
cells, including companions and conjoints, whenever they exist in $\dbl{D}$. The
general theory of companions, conjoints, and equipments is established in
\citep{grandis2004,shulman2008,shulman2010}. In this paper, we use the terms
\define{restriction} and \define{extension cell} as more evocative synonyms for
cartesian and opcartesian cells in a double category, respectively.

\begin{proposition}[Extensions of families]
  \label{prop:extensions-fam}
  Let $\dbl{D}$ be a double category. Suppose given a co-niche in
  $\DblFam(\dbl{D})$ of the form
  \begin{equation*}
    \begin{tikzcd}[row sep=scriptsize]
      {(I,\vec{x})} & {(J,\vec{y})} \\
      {(K,\vec{w})} & {(L,\vec{z})}
      \arrow["{(A,\vec{m})}", "\shortmid"{marking}, from=1-1, to=1-2]
      \arrow["{(f_0,f)}"', from=1-1, to=2-1]
      \arrow["{(g_0,g)}", from=1-2, to=2-2]
    \end{tikzcd}
  \end{equation*}
  such that each constituent co-niche in $\dbl{D}$
  \begin{equation*}
    \begin{tikzcd}[row sep=scriptsize]
      {x_i} & {y_j} \\
      {w_{f_0(i)}} & {z_{g_0(j)}}
      \arrow["{m_a}", "\shortmid"{marking}, from=1-1, to=1-2]
      \arrow["{f_i}"', from=1-1, to=2-1]
      \arrow["{g_j}", from=1-2, to=2-2]
    \end{tikzcd},
    \qquad (i \xproto{a} j) : (I \xproto{A} J),
  \end{equation*}
  can be filled by an extension cell. Then the co-niche in $\DblFam(\dbl{D})$ is
  also fillable by an extension cell.
\end{proposition}
\begin{proof}
  From the extension in $\Span$
  \begin{equation*}
    \begin{tikzcd}[row sep=scriptsize]
      I & A & J \\
      K & A & L
      \arrow["\ell"', from=1-2, to=1-1]
      \arrow["r", from=1-2, to=1-3]
      \arrow["{f_0}"', from=1-1, to=2-1]
      \arrow["{g_0}", from=1-3, to=2-3]
      \arrow[Rightarrow, no head, from=1-2, to=2-2]
      \arrow["{f_0 \circ \ell}", from=2-2, to=2-1]
      \arrow["{g_0 \circ r}"', from=2-2, to=2-3]
    \end{tikzcd}
  \end{equation*}
  together with the family of extensions in $\dbl{D}$
  \begin{equation*}
    \begin{tikzcd}
      {x_i} & {y_j} \\
      {w_{f_0(i)}} & {z_{g_0(j)}}
      \arrow[""{name=0, anchor=center, inner sep=0}, "{m_a}", "\shortmid"{marking}, from=1-1, to=1-2]
      \arrow["{f_i}"', from=1-1, to=2-1]
      \arrow["{g_j}", from=1-2, to=2-2]
      \arrow[""{name=1, anchor=center, inner sep=0}, "{n_a}"', "\shortmid"{marking}, from=2-1, to=2-2]
      \arrow["{\ext_a}"{description}, draw=none, from=0, to=1]
    \end{tikzcd},
    \qquad (i \xproto{a} j) : (I \xproto{A} J),
  \end{equation*}
  we obtain a family of cells in $\dbl{D}$:
  \begin{equation*}
    \begin{tikzcd}
      {(I,\vec{x})} & {(J,\vec{y})} \\
      {(K,\vec{w})} & {(L,\vec{z})}
      \arrow[""{name=0, anchor=center, inner sep=0}, "{(A,\vec{m})}", "\shortmid"{marking}, from=1-1, to=1-2]
      \arrow["{(f_0,f)}"', from=1-1, to=2-1]
      \arrow["{(g_0,g)}", from=1-2, to=2-2]
      \arrow[""{name=1, anchor=center, inner sep=0}, "{(A,\vec{n})}"', "\shortmid"{marking}, from=2-1, to=2-2]
      \arrow["{(1_A, \ext)}"{description}, draw=none, from=0, to=1]
    \end{tikzcd}.
  \end{equation*}
  We verify that this family satisfies the universal property of an extension
  cell in $\DblFam(\dbl{D})$. Given a cell in $\DblFam(\dbl{D})$ as on the left
  \begin{equation*}
    \begin{tikzcd}
      {(I,\vec{x})} & {(J,\vec{y})} \\
      {(K,\vec{w})} & {(L,\vec{z})} \\
      {(K',\vec{w}')} & {(L', \vec{z}')}
      \arrow[""{name=0, anchor=center, inner sep=0}, "{(A,\vec{m})}", "\shortmid"{marking}, from=1-1, to=1-2]
      \arrow["{(f_0,f)}"', from=1-1, to=2-1]
      \arrow["{(g_0,g)}", from=1-2, to=2-2]
      \arrow["{(h_0,h)}"', from=2-1, to=3-1]
      \arrow["{(k_0, k)}", from=2-2, to=3-2]
      \arrow[""{name=1, anchor=center, inner sep=0}, "{(B,\vec{p})}"', "\shortmid"{marking}, from=3-1, to=3-2]
      \arrow["{(\alpha_0,\alpha)}"{description}, draw=none, from=0, to=1]
    \end{tikzcd}
    \quad=\quad
    \begin{tikzcd}
      {(I,\vec{x})} & {(J,\vec{y})} \\
      {(K,\vec{w})} & {(L,\vec{z})} \\
      {(K',\vec{w}')} & {(L', \vec{z}')}
      \arrow[""{name=0, anchor=center, inner sep=0}, "{(A,\vec{m})}", "\shortmid"{marking}, from=1-1, to=1-2]
      \arrow["{(f_0,f)}"', from=1-1, to=2-1]
      \arrow["{(g_0,g)}", from=1-2, to=2-2]
      \arrow["{(h_0,h)}"', from=2-1, to=3-1]
      \arrow["{(k_0, k)}", from=2-2, to=3-2]
      \arrow[""{name=1, anchor=center, inner sep=0}, "{(B,\vec{p})}"', "\shortmid"{marking}, from=3-1, to=3-2]
      \arrow[""{name=2, anchor=center, inner sep=0}, "{(A,\vec{n})}"', "\shortmid"{marking}, from=2-1, to=2-2]
      \arrow["{(1_A, \ext)}"{description}, draw=none, from=0, to=2]
      \arrow["{\exists !}"{description, pos=0.6}, draw=none, from=2, to=1]
    \end{tikzcd},
  \end{equation*}
  we must show that there is a unique factorization as on the right. But this
  follow directly by applying the universal properties of each of the
  constituent extension cells:
  \begin{equation*}
    \begin{tikzcd}
      {x_i} & {y_j} \\
      {w_{f_0(i)}} & {z_{g_0(j)}} \\
      {w'_{h_0(f_0(i))}} & {z'_{k_0(g_0(j))}}
      \arrow[""{name=0, anchor=center, inner sep=0}, "{m_a}", "\shortmid"{marking}, from=1-1, to=1-2]
      \arrow["{f_i}"', from=1-1, to=2-1]
      \arrow["{g_j}", from=1-2, to=2-2]
      \arrow["{h_{f_0(i)}}"', from=2-1, to=3-1]
      \arrow["{k_{g_0(j)}}", from=2-2, to=3-2]
      \arrow[""{name=1, anchor=center, inner sep=0}, "{p_{\alpha_0(a)}}"', "\shortmid"{marking}, from=3-1, to=3-2]
      \arrow["{\alpha_a}"{description}, draw=none, from=0, to=1]
    \end{tikzcd}
    =
    \begin{tikzcd}
      {x_i} & {y_j} \\
      {w_{f_0(i)}} & {z_{g_0(j)}} \\
      {w'_{h_0(f_0(i))}} & {z'_{k_0(g_0(j))}}
      \arrow[""{name=0, anchor=center, inner sep=0}, "{m_a}", "\shortmid"{marking}, from=1-1, to=1-2]
      \arrow["{f_i}"', from=1-1, to=2-1]
      \arrow["{g_j}", from=1-2, to=2-2]
      \arrow["{h_{f_0(i)}}"', from=2-1, to=3-1]
      \arrow["{k_{g_0(j)}}", from=2-2, to=3-2]
      \arrow[""{name=1, anchor=center, inner sep=0}, "{p_{\alpha_0(a)}}"', "\shortmid"{marking}, from=3-1, to=3-2]
      \arrow[""{name=2, anchor=center, inner sep=0}, "{n_a}"', "\shortmid"{marking}, from=2-1, to=2-2]
      \arrow["{\ext_a}"{description}, draw=none, from=0, to=2]
      \arrow["{\exists !}"{description, pos=0.6}, draw=none, from=2, to=1]
    \end{tikzcd},
    \qquad (i \xproto{a} j) : (I \xproto{A} J).
  \end{equation*}
\end{proof}

\begin{corollary} \label{cor:equipment-fam}
  If a double category $\dbl{D}$ is an equipment, then $\DblFam(\dbl{D})$ is
  also an equipment.
\end{corollary}
\begin{proof}
  This follows immediately from Proposition \ref{prop:extensions-fam} using the
  characterization of an equipment as a double category whose source-target
  pairing is an opfibration \citep[\mbox{Theorem 4.1}]{shulman2008}.
\end{proof}

\begin{proposition}[Companions and conjoints of families]
  \label{prop:companions-fam}
  Suppose $(f_0,f): (I,\vec{x}) \to (J,\vec{y})$ is a family of arrows in a
  double category $\dbl{D}$ such that each component $f_i: x_i \to y_{f_0(i)}$
  has a companion $(f_i)_!: x_i \proto y_{f_0(i)}$ in $\dbl{D}$. Then the family
  of arrows $(f_0,f)$ has a companion in $\DblFam(\dbl{D})$, namely the family
  of proarrows $((f_0)_!, f_!): (I,\vec{x}) \proto (J,\vec{y})$, where
  $(f_0)_! = (I = I \xto{f_0} J)$ is the companion of $f_0$ in $\Span$ and
  $(f_!)_i \coloneqq (f_i)_!$ for each $i \in I$.

  Dually, if each component $f_i$ has a conjoint $f_i^*: y_{f_0(i)} \proto x_i$
  in $\dbl{D}$, then the family of arrows $(f_0,f)$ has a conjoint in
  $\DblFam(\dbl{D})$, namely the family of proarrows $(f_0^*, f^*)$ indexed by
  the conjoint span $f_0^* = (J \xfrom{f_0} I = I)$.
\end{proposition}
\begin{proof}
  Since companions and conjoints are special kinds of extension cells, this
  result is a direct consequence of Proposition \ref{prop:extensions-fam}. We sketch a
  direct proof anyway since it is useful to have formulas for the binding cells
  of companions and conjoints in $\DblFam(\dbl{D})$. The unit and counit cells
  for a companion pair in $\DblFam(\dbl{D})$,
  \begin{equation*}
    \begin{tikzcd}
      {(I,\vec{x})} & {(I,\vec{x})} \\
      {(I,\vec{x})} & {(J,\vec{y})}
      \arrow["{(f_0,f)}", from=1-2, to=2-2]
      \arrow[""{name=0, anchor=center, inner sep=0}, "{((f_0)_!, f_!)}"', "\shortmid"{marking}, from=2-1, to=2-2]
      \arrow[Rightarrow, no head, from=1-1, to=2-1]
      \arrow[""{name=1, anchor=center, inner sep=0}, "{\id_{(I,\vec{x})}}", "\shortmid"{marking}, from=1-1, to=1-2]
      \arrow["{(1_I,\eta)}"{description}, draw=none, from=1, to=0]
    \end{tikzcd}
    \qquad\text{and}\qquad
    \begin{tikzcd}
      {(I,\vec{x})} & {(J,\vec{y})} \\
      {(J,\vec{y})} & {(J,\vec{y})}
      \arrow["{(f_0,f)}"', from=1-1, to=2-1]
      \arrow[""{name=0, anchor=center, inner sep=0}, "{\id_{(J,\vec{y})}}"', "\shortmid"{marking}, from=2-1, to=2-2]
      \arrow[Rightarrow, no head, from=1-2, to=2-2]
      \arrow[""{name=1, anchor=center, inner sep=0}, "{((f_0)_!, f_!)}", "\shortmid"{marking}, from=1-1, to=1-2]
      \arrow["{(f_0,\varepsilon)}"{description}, draw=none, from=1, to=0]
    \end{tikzcd},
  \end{equation*}
  are the families of units and counits in $\dbl{D}$
  \begin{equation*}
    \begin{tikzcd}
      {x_i} & {x_i} \\
      {x_i} & {y_{f_0(i)}}
      \arrow["{f_i}", from=1-2, to=2-2]
      \arrow[""{name=0, anchor=center, inner sep=0}, "{\id_{x_i}}", "\shortmid"{marking}, from=1-1, to=1-2]
      \arrow[Rightarrow, no head, from=1-1, to=2-1]
      \arrow[""{name=1, anchor=center, inner sep=0}, "{(f_i)_!}"', "\shortmid"{marking}, from=2-1, to=2-2]
      \arrow["{\eta_i}"{description}, draw=none, from=0, to=1]
    \end{tikzcd}
    \qquad\text{and}\qquad
    \begin{tikzcd}
      {x_i} & {y_{f_0(i)}} \\
      {y_{f_0(i)}} & {y_{f_0(i)}}
      \arrow["{f_i}"', from=1-1, to=2-1]
      \arrow[""{name=0, anchor=center, inner sep=0}, "{\id_{y_{f_0(i)}}}"', "\shortmid"{marking}, from=2-1, to=2-2]
      \arrow[Rightarrow, no head, from=1-2, to=2-2]
      \arrow[""{name=1, anchor=center, inner sep=0}, "{(f_i)_!}", "\shortmid"{marking}, from=1-1, to=1-2]
      \arrow["{\varepsilon_i}"{description}, draw=none, from=1, to=0]
    \end{tikzcd},
    \qquad i \in I.
  \end{equation*}
  The equations for the companion pair in $\DblFam(\dbl{D})$ follow immediately
  by applying the companion equations componentwise in $\dbl{D}$. Dually, the
  unit and counit for a conjoint pair in $\DblFam(\dbl{D})$,
  \begin{equation*}
    \begin{tikzcd}
      {(I,\vec{x})} & {(I,\vec{x})} \\
      {(J,\vec{y})} & {(I,\vec{x})}
      \arrow[""{name=0, anchor=center, inner sep=0}, "{(f_0^*, f^*)}"', "\shortmid"{marking}, from=2-1, to=2-2]
      \arrow["{(f_0,f)}"', from=1-1, to=2-1]
      \arrow[""{name=1, anchor=center, inner sep=0}, "{\id_{(I,\vec{x})}}", "\shortmid"{marking}, from=1-1, to=1-2]
      \arrow[Rightarrow, no head, from=1-2, to=2-2]
      \arrow["{(1_I, \eta)}"{description}, draw=none, from=1, to=0]
    \end{tikzcd}
    \qquad\text{and}\qquad
    \begin{tikzcd}
      {(J,\vec{y})} & {(I,\vec{x})} \\
      {(J,\vec{y})} & {(J,\vec{y})}
      \arrow[""{name=0, anchor=center, inner sep=0}, "{(f_0^*, f^*)}", "\shortmid"{marking}, from=1-1, to=1-2]
      \arrow["{(f_0,f)}", from=1-2, to=2-2]
      \arrow[Rightarrow, no head, from=1-1, to=2-1]
      \arrow[""{name=1, anchor=center, inner sep=0}, "{\id_{(J,\vec{y})}}"', "\shortmid"{marking}, from=2-1, to=2-2]
      \arrow["{(f_0, \varepsilon)}"{description}, draw=none, from=0, to=1]
    \end{tikzcd},
  \end{equation*}
  are the families of units and counits in $\dbl{D}$
  \begin{equation*}
    \begin{tikzcd}
      {x_i} & {x_i} \\
      {y_{f_0(i)}} & {x_i}
      \arrow[""{name=0, anchor=center, inner sep=0}, "{\id_{x_i}}", "\shortmid"{marking}, from=1-1, to=1-2]
      \arrow["{f_i}"', from=1-1, to=2-1]
      \arrow[""{name=1, anchor=center, inner sep=0}, "{f_i^*}"', "\shortmid"{marking}, from=2-1, to=2-2]
      \arrow[Rightarrow, no head, from=1-2, to=2-2]
      \arrow["{\eta_i}"{description}, draw=none, from=0, to=1]
    \end{tikzcd}
    \qquad\text{and}\qquad
    \begin{tikzcd}
      {y_{f_0(i)}} & {x_i} \\
      {y_{f_0(i)}} & {y_{f_0(i)}}
      \arrow[""{name=0, anchor=center, inner sep=0}, "{f_i^*}", "\shortmid"{marking}, from=1-1, to=1-2]
      \arrow["{f_i}", from=1-2, to=2-2]
      \arrow[""{name=1, anchor=center, inner sep=0}, "{\id_{y_{f_0(i)}}}"', "\shortmid"{marking}, from=2-1, to=2-2]
      \arrow[Rightarrow, no head, from=1-1, to=2-1]
      \arrow["{\varepsilon_i}"{description}, draw=none, from=0, to=1]
    \end{tikzcd},
    \qquad i \in I.
  \end{equation*}
\end{proof}

\section{Coproducts in double categories}
\label{sec:coproducts}

Following the analogy now established between families in categories and in
double categories, we define coproducts in a double category. The
characterization of coproducts in a category $\cat{C}$ as left adjoint to the
embedding $\Delta: \cat{C} \to \Fam(\cat{C})$
(Proposition \ref{prop:coproducts-as-adjoint}) becomes the \emph{definition} of coproducts
in a double category. Given a double category $\dbl{D}$, denote by
$\Delta \coloneqq \Delta_{\dbl{D}}: \dbl{D} \to \DblFam(\dbl{D})$ the double
functor sending each object $x \in \dbl{D}$ to the singleton family $(1,x)$ and
each proarrow $m: x \proto y$ in $\dbl{D}$ to the singleton family
$(1 \xfrom{!} 1 \xto{!} 1, m): (1,x) \proto (1,y)$.

\begin{definition}[Double coproducts] \label{def:dbl-coproducts}
  A double category $\dbl{D}$ \define{has colax coproducts} if the double
  functor $\Delta: \dbl{D} \to \DblFam(\dbl{D})$ has a colax left adjoint:
  \begin{equation*}
    \begin{tikzcd}
      {\dbl{D}} & {\DblFam(\dbl{D})}
      \arrow[""{name=0, anchor=center, inner sep=0}, "\Sigma"', curve={height=18pt}, from=1-2, to=1-1]
      \arrow[""{name=1, anchor=center, inner sep=0}, "\Delta"', curve={height=18pt}, from=1-1, to=1-2]
      \arrow["\dashv"{anchor=center, rotate=-90}, draw=none, from=0, to=1]
    \end{tikzcd}.
  \end{equation*}
  The double category $\dbl{D}$ \define{has strong coproducts} when the left
  adjoint $\Sigma: \DblFam(\dbl{D}) \to \dbl{D}$ is a \emph{pseudo} double
  functor. Colax and strong \define{finite coproducts} in $\dbl{D}$ are defined
  analogously, replacing the double category $\DblFam(\dbl{D})$ with
  $\DblFinFam(\dbl{D})$.
\end{definition}

Recall that a double adjunction is in general between a colax double functor on
the left and a lax double functor on the right \citep{grandis2004},
\citep[\S{4.3}]{grandis2019}. In this case, the right adjoint is pseudo, so we
can see the adjunction as being in the 2-category $\DblColax$ of double
categories, colax double functors, and natural transformations. On the other
hand, nothing requires the left adjoint to be pseudo and we will see examples
where it is not.

Another description of double coproducts involves less data than the definition
above and hence is more convenient to check in examples.

\begin{proposition}[Double coproducts via universal arrows]
  \label{prop:dbl-coproducts-universal}
  A double category $\dbl{D}$ has colax coproducts if and only if
  \begin{enumerate}[(i)]
    \item for every object family $(I,\vec{x})$ in $\dbl{D}$, there is a choice
      of object $\Sigma \vec{x} \coloneq \Sigma(I,\vec{x})$ in $\dbl{D}$ and a
      universal arrow
      \begin{equation*}
        (!, \iota): (I, \vec{x}) \to \Delta \Sigma \vec{x}
      \end{equation*}
      from $(I,\vec{x})$ to the functor
      $\Delta_0: \dbl{D}_0 \to \Fam(\dbl{D}_0)$;
    \item for every proarrow family
      $(A,\vec{m}): (I,\vec{x}) \proto (J,\vec{y})$ in $\dbl{D}$, there is a
      choice of proarrow
      $\Sigma \vec{m} \coloneqq \Sigma(A,\vec{m}): \Sigma\vec{x} \proto \Sigma\vec{y}$
      in $\dbl{D}$, compatible with the choices on objects, and a universal
      arrow
      \begin{equation*}
        \begin{tikzcd}
          {(I, \vec{x})} & {(J, \vec{y})} \\
          {\Delta \Sigma \vec{x}} & {\Delta \Sigma \vec{y}}
          \arrow[""{name=0, anchor=center, inner sep=0}, "{\Delta \Sigma \vec{m}}"', "\shortmid"{marking}, from=2-1, to=2-2]
          \arrow["{(!,\iota)}"', from=1-1, to=2-1]
          \arrow["{(!,\iota)}", from=1-2, to=2-2]
          \arrow[""{name=1, anchor=center, inner sep=0}, "{(A, \vec{m})}", "\shortmid"{marking}, from=1-1, to=1-2]
          \arrow["{(!,\iota)}"{description}, draw=none, from=1, to=0]
        \end{tikzcd}
      \end{equation*}
      from $(A,\vec{m})$ to the functor
      $\Delta_1: \dbl{D}_1 \to \DblFam(\dbl{D})_1$.
  \end{enumerate}
  Moreover, in this case, the coproducts in $\dbl{D}$ are strong if and only if,
  for every pair of composable families
  $(I,\vec{x}) \xproto{(A,\vec{m})} (J,\vec{y}) \xproto{(B,\vec{n})} (K,\vec{z})$
  and for every family $(I,\vec{x})$, the canonical comparison cells
  \begin{equation*}
    \begin{tikzcd}
      {\Sigma\vec{x}} && {\Sigma\vec{z}} \\
      {\Sigma\vec{x}} & {\Sigma\vec{y}} & {\Sigma\vec{z}}
      \arrow[""{name=0, anchor=center, inner sep=0}, "{\Sigma(\vec{m} \odot \vec{n})}", "\shortmid"{marking}, from=1-1, to=1-3]
      \arrow[Rightarrow, no head, from=1-1, to=2-1]
      \arrow[Rightarrow, no head, from=1-3, to=2-3]
      \arrow["{\Sigma\vec{m}}"', "\shortmid"{marking}, from=2-1, to=2-2]
      \arrow["{\Sigma\vec{n}}"', "\shortmid"{marking}, from=2-2, to=2-3]
      \arrow["{\Sigma_{\vec{m},\vec{n}}}"{description, pos=0.6}, draw=none, from=0, to=2-2]
    \end{tikzcd}
    \qquad\text{and}\qquad
    \begin{tikzcd}
      {\Sigma\vec{x}} & {\Sigma\vec{x}} \\
      {\Sigma\vec{x}} & {\Sigma\vec{x}}
      \arrow[""{name=0, anchor=center, inner sep=0}, "{\id_{\Sigma\vec{x}}}"', "\shortmid"{marking}, from=2-1, to=2-2]
      \arrow[Rightarrow, no head, from=1-1, to=2-1]
      \arrow[Rightarrow, no head, from=1-2, to=2-2]
      \arrow[""{name=1, anchor=center, inner sep=0}, "{\Sigma(\id_{\vec{x}})}", "\shortmid"{marking}, from=1-1, to=1-2]
      \arrow["{\Sigma_{\vec{x}}}"{description}, draw=none, from=1, to=0]
    \end{tikzcd}
  \end{equation*}
  are isomorphisms in $\dbl{D}_1$, where the composite family
  $(A \times_J B, \vec{m} \odot \vec{n})$ has been defined in Equation
  \eqref{eq:proarrow-family-composition}.
\end{proposition}
\begin{proof}
  This is a direct application of the equivalence between adjunctions and
  universal arrows for double categories, stated for right adjoints in
  \citep[Theorem 3.6]{grandis2004} and \citep[Theorem 4.3.6]{grandis2019}.
\end{proof}

Note that condition (i) in Proposition \ref{prop:dbl-coproducts-universal} says
nothing more than that the underlying category $\dbl{D}_0$ has coproducts (cf.\
Proposition \ref{prop:coproducts-as-adjoint}). Condition (ii) is more elaborate.
It implies, but is stronger than, the statement that $\dbl{D}_1$ has coproducts.
First, it must be possible to choose coproducts in $\dbl{D}_1$ compatibly with
coproducts in $\dbl{D}_0$. Moreover, depending on the indexing span, condition
(ii) encompasses notions of coproduct going beyond coproducts in $\dbl{D}_1$,
such as local coproducts. Stated explicitly, the universal property in (ii) says
that for any family of cells indexed by the span $A: I \proto J$ and having the
form on the left,
\begin{equation*}
  \begin{tikzcd}
	{x_i} & {y_j} \\
	w & z
	\arrow[""{name=0, anchor=center, inner sep=0}, "{m_a}"{inner sep=.8ex}, "\shortmid"{marking}, from=1-1, to=1-2]
	\arrow["{f_i}"', from=1-1, to=2-1]
	\arrow["{g_j}", from=1-2, to=2-2]
	\arrow[""{name=1, anchor=center, inner sep=0}, "n"'{inner sep=.8ex}, "\shortmid"{marking}, from=2-1, to=2-2]
	\arrow["{\alpha_a}"{description}, draw=none, from=0, to=1]
  \end{tikzcd}
  \quad=\quad
  \begin{tikzcd}
	{x_i} & {y_j} \\
	{\Sigma\vec{x}} & {\Sigma\vec{y}} \\
	w & z
	\arrow[""{name=0, anchor=center, inner sep=0}, "{m_a}"{inner sep=.8ex}, "\shortmid"{marking}, from=1-1, to=1-2]
	\arrow["{\iota_i}"', from=1-1, to=2-1]
	\arrow["{\iota_j}", from=1-2, to=2-2]
	\arrow[""{name=1, anchor=center, inner sep=0}, "{\Sigma\vec{m}}"{inner sep=.8ex}, "\shortmid"{marking}, from=2-1, to=2-2]
	\arrow[dashed, from=2-1, to=3-1]
	\arrow[dashed, from=2-2, to=3-2]
	\arrow[""{name=2, anchor=center, inner sep=0}, "n"'{inner sep=.8ex}, "\shortmid"{marking}, from=3-1, to=3-2]
	\arrow["{\iota_a}"{description, pos=0.4}, draw=none, from=0, to=1]
	\arrow[between={0.3}{0.7}, Rightarrow, dashed, from=1, to=2]
  \end{tikzcd},
  \qquad (i \xproto{a} j) : (I \xproto{A} J),
\end{equation*}
there exists a unique cell factoring each cell $\alpha_a$ through the
corresponding coprojection $\iota_a$, as shown on the right.

The canonical comparison cells used in the final statement of Proposition
\ref{prop:dbl-coproducts-universal} are the unique solutions, given by the
universal properties of universal arrows, to the equations
\begin{equation*}
  \begin{tikzcd}
    {x_i} && {z_k} \\
    {\Sigma\vec{x}} && {\Sigma\vec{z}} \\
    {\Sigma\vec{x}} & {\Sigma\vec{y}} & {\Sigma\vec{z}}
    \arrow[""{name=0, anchor=center, inner sep=0}, "{\Sigma(\vec{m} \odot \vec{n})}", "\shortmid"{marking}, from=2-1, to=2-3]
    \arrow[Rightarrow, no head, from=2-1, to=3-1]
    \arrow[Rightarrow, no head, from=2-3, to=3-3]
    \arrow["{\Sigma\vec{m}}"', "\shortmid"{marking}, from=3-1, to=3-2]
    \arrow["{\Sigma\vec{n}}"', "\shortmid"{marking}, from=3-2, to=3-3]
    \arrow["{\iota_i}"', from=1-1, to=2-1]
    \arrow["{\iota_k}", from=1-3, to=2-3]
    \arrow[""{name=1, anchor=center, inner sep=0}, "{m_a \odot m_b}", "\shortmid"{marking}, from=1-1, to=1-3]
    \arrow["{\Sigma_{\vec{m},\vec{n}}}"{description, pos=0.6}, draw=none, from=0, to=3-2]
    \arrow["{\iota_{(a,b)}}"{description, pos=0.4}, draw=none, from=1, to=0]
  \end{tikzcd}
  =
  \begin{tikzcd}
    {x_i} & {y_j} & {z_k} \\
    {\Sigma\vec{x}} & {\Sigma\vec{y}} & {\Sigma\vec{z}}
    \arrow[""{name=0, anchor=center, inner sep=0}, "{m_a}", "\shortmid"{marking}, from=1-1, to=1-2]
    \arrow[""{name=1, anchor=center, inner sep=0}, "{m_b}", "\shortmid"{marking}, from=1-2, to=1-3]
    \arrow["{\iota_i}"', from=1-1, to=2-1]
    \arrow["{\iota_j}"{description}, from=1-2, to=2-2]
    \arrow[""{name=2, anchor=center, inner sep=0}, "{\Sigma{\vec{m}}}"', "\shortmid"{marking}, from=2-1, to=2-2]
    \arrow["{\iota_k}", from=1-3, to=2-3]
    \arrow[""{name=3, anchor=center, inner sep=0}, "{\Sigma\vec{n}}"', "\shortmid"{marking}, from=2-2, to=2-3]
    \arrow["{\iota_a}"{description}, draw=none, from=0, to=2]
    \arrow["{\iota_b}"{description}, draw=none, from=1, to=3]
  \end{tikzcd},
  \qquad (i \xproto{a} j \xproto{b} k) : (I \xproto{A} J \xproto{B} K),
\end{equation*}
where the cells $\iota_{(a,b)}$ are the hypothesized coprojections into
$\Sigma(A \times_J B, \vec{m} \odot \vec{n})$, along with the equations
\begin{equation*}
  \begin{tikzcd}
    {x_i} & {x_i} \\
    {\Sigma\vec{x}} & {\Sigma\vec{x}} \\
    {\Sigma\vec{x}} & {\Sigma\vec{x}}
    \arrow[""{name=0, anchor=center, inner sep=0}, "{\id_{\Sigma\vec{x}}}"', "\shortmid"{marking}, from=3-1, to=3-2]
    \arrow[Rightarrow, no head, from=2-1, to=3-1]
    \arrow[Rightarrow, no head, from=2-2, to=3-2]
    \arrow[""{name=1, anchor=center, inner sep=0}, "{\Sigma(\id_{\vec{x}})}", "\shortmid"{marking}, from=2-1, to=2-2]
    \arrow["{\iota_i}"', from=1-1, to=2-1]
    \arrow["{\iota_i}", from=1-2, to=2-2]
    \arrow[""{name=2, anchor=center, inner sep=0}, "{\id_{x_i}}", "\shortmid"{marking}, from=1-1, to=1-2]
    \arrow["{\Sigma_{\vec{x}}}"{description}, draw=none, from=1, to=0]
    \arrow["{\iota_i}"{description, pos=0.4}, draw=none, from=2, to=1]
  \end{tikzcd}
  =
  \begin{tikzcd}
    {x_i} & {x_i} \\
    {\Sigma\vec{x}} & {\Sigma\vec{x}}
    \arrow[""{name=0, anchor=center, inner sep=0}, "{\id_{x_i}}", "\shortmid"{marking}, from=1-1, to=1-2]
    \arrow["{\iota_i}"', from=1-1, to=2-1]
    \arrow["{\iota_i}", from=1-2, to=2-2]
    \arrow[""{name=1, anchor=center, inner sep=0}, "{\id_{\Sigma\vec{x}}}"', "\shortmid"{marking}, from=2-1, to=2-2]
    \arrow["{\id_{\iota_i}}"{description}, draw=none, from=0, to=1]
  \end{tikzcd},
  \qquad i \in I.
\end{equation*}
When the assignments in Proposition \ref{prop:dbl-coproducts-universal} are
extended to form a colax double functor $\Sigma: \DblFam(\dbl{D}) \to \dbl{D}$, the
cells above become the composition and identity comparison cells for the colax
functor.

Under mild conditions, double categories of spans have coproducts.

\begin{theorem}[Coproducts of spans] \label{thm:coproducts-span}
  Let $\cat{S}$ be a category with pullbacks and (finite) coproducts. Then the
  double category $\Span{\cat{S}}$ has colax (finite) coproducts.

  Moreover, if $\cat{S}$ is an (finitary or infinitary) extensive category, then
  the (finite or arbitrary) coproducts in $\Span{\cat{S}}$ are strong.
\end{theorem}
\begin{proof}
  We first show that $\Span{\cat{S}}$ has colax coproducts. By assumption, the
  category $\Span{\cat{S}}_0 = \cat{S}$ has coproducts. Suppose that
  $(A,\vec{m}): (I,\vec{x}) \proto (J,\vec{y})$ is a family of spans in
  $\cat{S}$, say of the form $m_a = (x_i \xfrom{\ell_a} s_a \xto{r_a} y_j)$ for
  each $a: i \proto j$. We take the coproduct
  $\Sigma\vec{s} = \sum_{a \in A} s_a$ in $\cat{S}$ of the apexes and then form
  the span $\Sigma\vec{x} \from \Sigma\vec{s} \to \Sigma\vec{y}$ by taking the
  copairing in $\cat{S}$ of the morphisms
  \begin{equation*}
    \Sigma\vec{x} \xfrom{\iota_i} x_i \xfrom{\ell_a} s_a
      \xto{r_a} y_j \xto{\iota_j} \Sigma\vec{y},
    \qquad (i \xproto{a} j): (I \xproto{A} J).
  \end{equation*}
  The accompanying family of coprojections is
  \begin{equation*}
    \begin{tikzcd}
      {(I,\vec{x})} & {(J,\vec{y})} \\
      {\Delta \Sigma \vec{x}} & {\Delta \Sigma \vec{y}}
      \arrow[""{name=0, anchor=center, inner sep=0}, "{\Delta \Sigma \vec{m}}"', "\shortmid"{marking}, from=2-1, to=2-2]
      \arrow[""{name=1, anchor=center, inner sep=0}, "{(A,\vec{m})}", "\shortmid"{marking}, from=1-1, to=1-2]
      \arrow["\iota"', from=1-1, to=2-1]
      \arrow["\iota", from=1-2, to=2-2]
      \arrow["\iota"{description}, draw=none, from=1, to=0]
    \end{tikzcd}
    \qquad\leftrightsquigarrow\qquad
    \begin{tikzcd}
      {x_i} & {s_a} & {y_j} \\
      {\Sigma\vec{x}} & {\Sigma\vec{s}} & {\Sigma\vec{y}}
      \arrow[from=2-2, to=2-1]
      \arrow[from=2-2, to=2-3]
      \arrow["{\iota_i}"', from=1-1, to=2-1]
      \arrow["{\ell_a}"', from=1-2, to=1-1]
      \arrow["{\iota_a}"', from=1-2, to=2-2]
      \arrow["{r_a}", from=1-2, to=1-3]
      \arrow["{\iota_j}", from=1-3, to=2-3]
    \end{tikzcd},
    \qquad (i \xproto{a} j): (I \xproto{A} J).
  \end{equation*}

  As for the universal property, suppose that $n = (w \xfrom{\ell} t \xto{r} z)$
  is another span in $\cat{S}$ along with a family of span maps
  \begin{equation*}
    \begin{tikzcd}
      {(I,\vec{x})} & {(J,\vec{y})} \\
      {\Delta w} & {\Delta z}
      \arrow[""{name=0, anchor=center, inner sep=0}, "{\Delta n}"', "\shortmid"{marking}, from=2-1, to=2-2]
      \arrow[""{name=1, anchor=center, inner sep=0}, "{(A,\vec{m})}", "\shortmid"{marking}, from=1-1, to=1-2]
      \arrow["f"', from=1-1, to=2-1]
      \arrow["g", from=1-2, to=2-2]
      \arrow["h"{description}, draw=none, from=1, to=0]
    \end{tikzcd}
    \qquad\leftrightsquigarrow\qquad
    \begin{tikzcd}
      {x_i} & {s_a} & {y_j} \\
      w & t & z
      \arrow["{\ell_a}"', from=1-2, to=1-1]
      \arrow["{r_a}", from=1-2, to=1-3]
      \arrow["{h_a}"', from=1-2, to=2-2]
      \arrow["{f_i}"', from=1-1, to=2-1]
      \arrow["{g_i}", from=1-3, to=2-3]
      \arrow["\ell", from=2-2, to=2-1]
      \arrow["r"', from=2-2, to=2-3]
    \end{tikzcd},
    \qquad (i \xproto{a} j): (I \xproto{A} J).
  \end{equation*}
  Taking the copairings $f \coloneqq [f_i]_{i \in I}$,
  $g \coloneqq [g_j]_{j \in J}$, and $h \coloneqq [h_a]_{a \in A}$, we get a
  unique map of spans
  \begin{equation*}
    \begin{tikzcd}
      {\Sigma\vec{x}} & {\Sigma\vec{s}} & {\Sigma\vec{y}} \\
      w & t & z
      \arrow["f"', from=1-1, to=2-1]
      \arrow["\ell", from=2-2, to=2-1]
      \arrow["r"', from=2-2, to=2-3]
      \arrow[from=1-2, to=1-1]
      \arrow["h"', from=1-2, to=2-2]
      \arrow["g", from=1-3, to=2-3]
      \arrow[from=1-2, to=1-3]
    \end{tikzcd}
  \end{equation*}
  that factors through the coprojections. Therefore, by
  Proposition \ref{prop:dbl-coproducts-universal}, the double category $\Span{\cat{S}}$ has
  colax coproducts.

  Now suppose that $\cat{S}$ is an extensive category. Under the assumption that
  $\cat{S}$ has coproducts and pullbacks, extensivity amounts to the properties
  that coproducts in $\cat{S}$ are disjoint and stable under pullback
  \citep[Proposition 17]{lack2006}, cf.\ \citep[Proposition 2.14]{carboni1993}. We
  need to see that $\Span{\cat{S}}$ has \emph{strong} coproducts. Given
  composable families of spans
  $(I,\vec{x}) \xproto{(A,\vec{m})} (J,\vec{y}) \xproto{(B,\vec{n})} (K,\vec{z})$,
  with families of apexes $\vec{s}$ and $\vec{t}$, we must show that the
  canonical comparison map in $\cat{S}$
  \begin{equation*}
    \begin{tikzcd}
      {\Sigma(\vec{s} \times_{\vec{y}} \vec{t})} \\
      {\Sigma\vec{s} \times_{\Sigma\vec{y}} \Sigma\vec{t}}
      \arrow[from=1-1, to=2-1]
    \end{tikzcd}
    \qquad\leftrightsquigarrow\qquad
    \begin{tikzcd}
      & {s_a \times_{y_j} t_b} \\
      {s_a} & {\Sigma\vec{s} \times_{\Sigma\vec{y}} \Sigma\vec{t}} & {t_b} \\
      {\Sigma\vec{s}} & {y_j} & {\Sigma\vec{t}} \\
      & {\Sigma\vec{y}}
      \arrow["{\iota_a}"', from=2-1, to=3-1]
      \arrow["{\pi_1}"', from=1-2, to=2-1]
      \arrow["{\pi_2}", from=1-2, to=2-3]
      \arrow["{\iota_b}", from=2-3, to=3-3]
      \arrow[from=2-2, to=3-1]
      \arrow[from=2-2, to=3-3]
      \arrow[from=3-1, to=4-2]
      \arrow[from=3-3, to=4-2]
      \arrow["\lrcorner"{anchor=center, pos=0.125, rotate=-45}, draw=none, from=2-2, to=4-2]
      \arrow["{\sigma_{a,b}}"', dashed, from=1-2, to=2-2]
      \arrow["{r_a}"'{pos=0.7}, from=2-1, to=3-2, crossing over]
      \arrow["{\iota_j}"', from=3-2, to=4-2]
      \arrow["{\ell_b}"{pos=0.7}, from=2-3, to=3-2, crossing over]
    \end{tikzcd},
    \quad (i \xproto{a} j \xproto{b} k) : (I \xproto{A} J \xproto{B} K).
  \end{equation*}
  on the left, which is determined by the family of comparisons $\sigma_{a,b}$
  on the right, is an isomorphism. This we do by decomposing it into several
  parts
  \begin{equation*}
    \begin{tikzcd}
      {\displaystyle\sum_{\mathclap{i \xproto{a} j \xproto{b} k}} s_a \times_{y_j} t_b} && {\displaystyle\sum_{\mathclap{i \xproto{a} j \xproto{b} k}} s_a \times_{\Sigma\vec{y}} t_b} \\
      {\Sigma\vec{s} \times_{\Sigma\vec{y}} \Sigma\vec{t}} && {\displaystyle\sum_{\substack{a: i \proto j \\ b: j' \proto k}} s_a \times_{\Sigma\vec{y}} t_b}
      \arrow["\cong", from=2-3, to=2-1]
      \arrow[dashed, from=1-1, to=2-1]
      \arrow["\cong", from=1-1, to=1-3]
      \arrow["\cong", from=1-3, to=2-3]
    \end{tikzcd}
    \qquad\leftrightsquigarrow\qquad
    \begin{tikzcd}
      {s_a \times_{y_j} t_b} & {s_a \times_{\Sigma\vec{y}} t_b} \\
      {\Sigma\vec{s} \times_{\Sigma\vec{y}} \Sigma\vec{t}} & {s_a \times_{\Sigma\vec{y}} t_b}
      \arrow["{\phi_{a,b}}", from=1-1, to=1-2]
      \arrow["{\sigma_{a,b}}"', dashed, from=1-1, to=2-1]
      \arrow["{\psi_{a,b}}", from=2-2, to=2-1]
      \arrow[Rightarrow, no head, from=1-2, to=2-2]
    \end{tikzcd}
  \end{equation*}
  that we individually verify to be isomorphisms. Starting from the end, the
  final comparison morphism in this sequence
  \begin{equation*}
    \begin{tikzcd}
      {\displaystyle\sum_{\substack{a: i \proto j \\ b: j' \proto k}} s_a \times_{\Sigma\vec{y}} t_b} \\
      {\Sigma\vec{s} \times_{\Sigma\vec{y}} \Sigma\vec{t}}
      \arrow["\cong"', from=1-1, to=2-1]
    \end{tikzcd}
    \qquad\leftrightsquigarrow\qquad
    \begin{tikzcd}[sep=scriptsize]
      & {s_a} & {s_a \times_{\Sigma\vec{y}} t_b} & {t_b} \\
      {y_j} & {\Sigma\vec{s}} & {\Sigma\vec{s} \times_{\Sigma\vec{y}} \Sigma\vec{t}} & {\Sigma\vec{t}} & {y_{j'}} \\
      & {\Sigma\vec{y}} && {\Sigma\vec{y}}
      \arrow["{\iota_a}"', from=1-2, to=2-2]
      \arrow["{\pi_1}"', from=1-3, to=1-2]
      \arrow["{\pi_2}", from=1-3, to=1-4]
      \arrow["{\iota_b}", from=1-4, to=2-4]
      \arrow["{\pi_1}"', from=2-3, to=2-2]
      \arrow[from=2-2, to=3-2]
      \arrow["{\psi_{a,b}}"', dashed, from=1-3, to=2-3]
      \arrow["{r_a}"', from=1-2, to=2-1]
      \arrow["{\iota_j}"', from=2-1, to=3-2]
      \arrow[""{name=0, anchor=center, inner sep=0}, Rightarrow, no head, from=3-2, to=3-4]
      \arrow["{\pi_2}", from=2-3, to=2-4]
      \arrow[from=2-4, to=3-4]
      \arrow["{\ell_b}", from=1-4, to=2-5]
      \arrow["{\iota_{j'}}", from=2-5, to=3-4]
      \arrow["\lrcorner"{anchor=center, pos=0.125, rotate=-45}, draw=none, from=2-3, to=0]
    \end{tikzcd},
    \quad \begin{aligned}
      &a: i \proto j \\
      &b: j' \proto k,
    \end{aligned}
  \end{equation*}
  is an isomorphism because coproducts in $\cat{S}$ are stable under pullback.
  Next, we analyze the pullbacks $s_a \times_{\Sigma\vec{y}} t_b$ involved by
  cases. In the case that $j = j'$, i.e., $(a,b) \in A \times_J B$, we form the
  pasting of pullback squares
  \begin{equation*}
    \begin{tikzcd}
      {s_a \times_{y_j} t_b} & {t_b} & {t_b} \\
      {s_a} & {y_j} & {y_j} \\
      {s_a} & {y_j} & {\Sigma\vec{y}}
      \arrow["{r_a}"', from=2-1, to=2-2]
      \arrow[from=1-1, to=2-1]
      \arrow[from=1-1, to=1-2]
      \arrow["{\ell_b}", from=1-2, to=2-2]
      \arrow["\lrcorner"{anchor=center, pos=0.125}, draw=none, from=1-1, to=2-2]
      \arrow[Rightarrow, no head, from=2-1, to=3-1]
      \arrow["{r_a}"', from=3-1, to=3-2]
      \arrow["{\iota_j}"', from=3-2, to=3-3]
      \arrow["{\iota_j}", from=2-3, to=3-3]
      \arrow[Rightarrow, no head, from=2-2, to=3-2]
      \arrow[Rightarrow, no head, from=2-2, to=2-3]
      \arrow["\lrcorner"{anchor=center, pos=0.125}, draw=none, from=2-2, to=3-3]
      \arrow[Rightarrow, no head, from=1-2, to=1-3]
      \arrow["{\ell_b}", from=1-3, to=2-3]
    \end{tikzcd},
  \end{equation*}
  where the bottom right square is a pullback, i.e., the coprojection
  $\iota_j: y_j \to \Sigma\vec{y}$ is monic, because coproducts in $\cat{S}$ are
  disjoint. Thus, the overall square is also a pullback and we obtain a
  canonical isomorphism
  $\phi_{a,b}: s_a \times_{y_j} t_b \xto{\cong} s_a \times_{\Sigma\vec{y}} t_b$.
  Otherwise, $j \neq j'$ and the pullback of
  $y_j \xto{\iota_j} \Sigma\vec{y} \xfrom{\iota_{j'}} y_{j'}$ is the initial
  object, again because coproducts in $\cat{S}$ are disjoint. This calculation
  verifies the first two isomorphisms in the decomposition and completes the
  proof.
\end{proof}

\begin{corollary}
  The double category $\Span$ has strong coproducts, and $\FinSpan$ has strong
  finite coproducts.
\end{corollary}
\begin{proof}
  The category $\cat{S} = \Set$ is infinitary extensive, and $\FinSet$ is
  finitary extensive.
\end{proof}

Double categories of matrices also have strong coproducts. This is true without
further assumptions as distributivity is needed to construct a double category
in the first place. By a \emph{distributive monoidal category}, we will mean a
symmetric monoidal category with arbitrary coproducts over which the tensor
distributes separately in each variable.

\begin{proposition}[Coproducts of matrices]
  For any distributive monoidal category $\catV$, the double category of
  $\catV$-matrices, $\Mat{\catV}$, has strong coproducts.
\end{proposition}
\begin{proof}
  Clearly, the underlying category $\Mat{\catV}_0 = \Set$ has coproducts. Given
  a family of $\catV$-matrices $(A,\vec{M}): (I,\vec{X}) \proto (J,\vec{Y})$,
  which have the form $M_a: X_i \times Y_j \to \catV$ for each element
  $a: i \proto j$, the coproduct
  $\Sigma\vec{M}: \Sigma\vec{X} \proto \Sigma\vec{Y}$ is the $\catV$-matrix
  $\Sigma\vec{M}: \Sigma\vec{X} \times \Sigma\vec{Y} \to \catV$ defined
  pointwise by the coproducts in $\catV$
  \begin{equation*}
    \Sigma\vec{M}(x_i, y_j) \coloneqq \sum_{a: i \proto j} M_a(x_i,y_j),
    \qquad \begin{aligned}
      & i \in I,\ j \in J, \\
      & x_i \in X_i,\ y_j \in Y_j.
    \end{aligned}
  \end{equation*}
  Here, for \emph{fixed} $i$ and $j$, the sum ranges over all $a \in A$ such
  that $\ell(a) = i$ and $r(a) = j$. The accompanying coprojections
  \begin{equation*}
    \begin{tikzcd}
      {X_i} & {Y_j} \\
      {\Sigma\vec{X}} & {\Sigma\vec{Y}}
      \arrow["{\iota_i}"', from=1-1, to=2-1]
      \arrow[""{name=0, anchor=center, inner sep=0}, "{M_a}", "\shortmid"{marking}, from=1-1, to=1-2]
      \arrow["{\iota_j}", from=1-2, to=2-2]
      \arrow[""{name=1, anchor=center, inner sep=0}, "{\Sigma\vec{M}}"', "\shortmid"{marking}, from=2-1, to=2-2]
      \arrow["{\iota_a}"{description}, draw=none, from=0, to=1]
    \end{tikzcd},
    \qquad (i \xproto{a} j) : (I \xproto{A} J),
  \end{equation*}
  are the maps of matrices with components given by coprojections in $\catV$
  \begin{equation*}
    M_a(x_i, y_j) \xto{\iota_a}
      \sum_{a': i \proto j} M_{a'}(x_i, y_j) = \Sigma\vec{M}(x_i, y_j),
    \qquad \begin{aligned}
      & a: i \proto j, \\
      & x_i \in X_i,\ y_j \in Y_j.
    \end{aligned}
  \end{equation*}
  The universal property of the coproduct in $\Mat{\catV}$ is easily verified
  using the universal properties of coproducts in $\Set$ and in $\catV$.

  So $\Mat{\catV}$ has colax coproducts. To see that it has strong coproducts,
  take composable families of $\catV$-matrices
  $(I,\vec{X}) \xproto{(A,\vec{M})} (J,\vec{Y}) \xproto{(B,\vec{N})} (K,\vec{Z})$.
  The domain of the comparison cell
  $\Sigma(\vec{M} \odot \vec{N}) \to \Sigma\vec{M} \odot \Sigma\vec{N}$ has
  elements
  \begin{equation*}
    \Sigma(\vec{M} \odot \vec{N})(x_i, z_k)
    = \sum_{i \xproto{a} j \xproto{b} k}
      (M_a \odot N_b)(x_i, z_k)
    = \sum_{i \xproto{a} j \xproto{b} k} \sum_{y_j \in Y_j}
      M_a(x_i, y_j) \otimes N_b(y_j, z_k)
  \end{equation*}
  for each choice of $i \in I$, $k \in K$, $x_i \in X_i$, and $z_k \in Z_k$,
  where the outer sum ranges over all $(a,b) \in A \times_J B$ such that
  $\ell(a) = i$ and $r(b) = k$. Rearranging the sums, the right-hand side is
  isomorphic in $\catV$ to
  \begin{equation*}
    \sum_{j \in J} \sum_{y_j \in Y_j} \sum_{a: i \proto j} \sum_{b: j \proto k}
      M_a(x_i, y_j) \otimes M_b(y_j, z_k).
  \end{equation*}
  Using the distributivity of tensors over coproducts in $\catV$, this is in
  turn isomorphic to
  \begin{equation*}
    \sum_{j \in J} \sum_{y_j \in Y_j}
      \Sigma\vec{M}(x_i, y_j) \otimes \Sigma\vec{N}(y_j, z_k)
    \cong (\Sigma\vec{M} \odot \Sigma\vec{N})(x_i, z_k).
  \end{equation*}
  The composite of these isomorphisms is one component of the comparison cell
  between matrices. Since all the components are isomorphisms, the comparison
  cell is itself an isomorphism.
\end{proof}

From the equivalence $\Mat \simeq \Span$, we obtain an independent proof of the
previous corollary.

\begin{corollary}
  The double category $\Mat$ has strong coproducts.
\end{corollary}
\begin{proof}
  The cartesian monoidal category $\catV = \Set$ is distributive.
\end{proof}

\section{Functors between double categories with coproducts}
\label{sec:preserving-coproducts}

We have yet to characterize $\DblFam(\dbl{D})$ as the free coproduct completion
of the double category $\dbl{D}$, as it must be. To even state such a result, we
need a notion not just of coproducts but of double functors that preserve
coproducts. Unlike in the one-dimensional case, choices must be made to define
an ambient 2-category for the universal property of the free coproduct
completion: between colax and strong coproducts, between colax and pseudo double
functors. We begin by saying what it means for a colax functor to preserve
coproducts, turning the characterization in Proposition \ref{prop:preserve-coproducts} of
preserving ordinary coproducts into the \emph{definition} of preserving double
coproducts.

\begin{construction}[Functorality of double families]
  \label{def:dbl-fam-functor}
  Given a colax double functor $F: \dbl{D} \to \dbl{E}$, there is another colax
  double functor
  \begin{equation*}
    \DblFam(F): \DblFam(\dbl{D}) \to \DblFam(\dbl{E})
  \end{equation*}
  that acts on families of objects and proarrows in $\dbl{D}$ by applying $F$
  elementwise to objects and proarrows, while preserving indexing sets and
  spans. Specifically, $\DblFam(F)$ sends
  \begin{itemize}
    \item an object family $\vec{x}: I \to \dbl{D}_0$ in $\dbl{D}$ to the object
      family $F_0 \circ \vec{x}: I \to \dbl{E}_0$ in $\dbl{E}$;
    \item a proarrow family $(A,\vec{m}): (I,\vec{x}) \proto (J,\vec{y})$ in
    $\dbl{D}$ to the proarrow family in $\dbl{E}$:
    \begin{equation*}
      \begin{tikzcd}
        I & A & J \\
        {\dbl{D}_0} & {\dbl{D}_1} & {\dbl{D}_0} \\
        {\dbl{E}_0} & {\dbl{E}_1} & {\dbl{E}_0}
        \arrow["{F_0}"', from=2-1, to=3-1]
        \arrow["\ell"', from=1-2, to=1-1]
        \arrow["{\vec{x}}"', from=1-1, to=2-1]
        \arrow["{\vec{m}}"', from=1-2, to=2-2]
        \arrow["s"', from=2-2, to=2-1]
        \arrow["{F_1}"', from=2-2, to=3-2]
        \arrow["s", from=3-2, to=3-1]
        \arrow["t", from=2-2, to=2-3]
        \arrow["{F_0}", from=2-3, to=3-3]
        \arrow["t"', from=3-2, to=3-3]
        \arrow["{\vec{y}}", from=1-3, to=2-3]
        \arrow["r", from=1-2, to=1-3]
      \end{tikzcd}.
    \end{equation*}
  \end{itemize}
  Similarly, $\DblFam(F)$ acts on families of arrows and cells by applying $F$
  elementwise to arrows and cells. The composition and identity comparisons for
  $\DblFam(F)$ are families of composition and identity comparisons for $F$, to
  be described in greater detail in Section \ref{sec:lax-functors} when they are
  needed. In particular, $\DblFam(F)$ is pseudo whenever $F$ is.
\end{construction}

Given that the double functor $\DblFam(F)$ acts by post-composition with $F$, it
is only a mild abuse of notation to write
$F\vec{x} \coloneqq \DblFam(F)(I,\vec{x})$ and
$F\vec{m} \coloneqq \DblFam(F)(A,\vec{m})$ for this ``elementwise'' application,
and we will often do so.

\begin{definition}[Preservation of double coproducts]
  \label{def:preserve-dbl-coproducts}
  Let $\dbl{D}$ and $\dbl{E}$ be double categories with colax (possibly strong)
  coproducts. A colax functor $F: \dbl{D} \to \dbl{E}$ \define{preserves
    coproducts} if the induced morphism of adjunctions
  \begin{equation*}
    \begin{tikzcd}
      {\DblFam(\dbl{D})} & {\dbl{D}} \\
      {\DblFam(\dbl{E})} & {\dbl{E}}
      \arrow[""{name=0, anchor=center, inner sep=0}, "{\Sigma \dashv \Delta}", "\shortmid"{marking}, from=1-1, to=1-2]
      \arrow["F", from=1-2, to=2-2]
      \arrow["{\DblFam(F)}"', from=1-1, to=2-1]
      \arrow[""{name=1, anchor=center, inner sep=0}, "{\Sigma \dashv \Delta}"', "\shortmid"{marking}, from=2-1, to=2-2]
      \arrow["{(\Phi,1)}"{description}, draw=none, from=0, to=1]
    \end{tikzcd}
  \end{equation*}
  is strong.
\end{definition}

Here the morphism of adjunctions $(\Phi, 1)$ is a cell in the double category of
adjunctions \citep[\S{3.1.5}]{grandis2019} in $\DblColax$, the 2-category of
double categories, colax double functors, and natural transformations. It
consists of the mate pair
\begin{equation*}
  \begin{tikzcd}
    {\DblFam(\dbl{D})} & {\dbl{D}} \\
    {\DblFam(\dbl{E})} & {\dbl{E}}
    \arrow["F", from=1-2, to=2-2]
    \arrow["{\Sigma_{\dbl{D}}}", from=1-1, to=1-2]
    \arrow["{\DblFam(F)}"', from=1-1, to=2-1]
    \arrow["{\Sigma_{\dbl{E}}}"', from=2-1, to=2-2]
    \arrow["\Phi", shorten <=6pt, shorten >=6pt, Rightarrow, from=2-1, to=1-2]
  \end{tikzcd}
  \qquad\leftrightsquigarrow\qquad
  \begin{tikzcd}
    {\DblFam(\dbl{D})} & {\dbl{D}} \\
    {\DblFam(\dbl{E})} & {\dbl{E}}
    \arrow["F", from=1-2, to=2-2]
    \arrow["{\Delta_{\dbl{D}}}"', from=1-2, to=1-1]
    \arrow["{\DblFam(F)}"', from=1-1, to=2-1]
    \arrow["{\Delta_{\dbl{E}}}", from=2-2, to=2-1]
    \arrow["1", shorten <=9pt, shorten >=9pt, Rightarrow, from=1-1, to=2-2]
  \end{tikzcd},
\end{equation*}
where the transformation on the right is the identity. The components of $\Phi$
are, at an object family $(I,\vec{x})$ in $\dbl{D}$, the usual coproduct
comparison $\Phi_{\vec{x}}: \Sigma F \vec{x} \to F \Sigma \vec{x}$ and, at a proarrow family
$(A,\vec{m}): (I,\vec{x}) \proto (J,\vec{y})$ in $\dbl{D}$, the cell on the left
\begin{equation*}
  \begin{tikzcd}
    {\Sigma F \vec{x}} & {\Sigma F \vec{y}} \\
    {F \Sigma \vec{x}} & {F \Sigma \vec{y}}
    \arrow["{\Phi_{\vec{x}}}"', from=1-1, to=2-1]
    \arrow[""{name=0, anchor=center, inner sep=0}, "{\Sigma F \vec{m}}", "\shortmid"{marking}, from=1-1, to=1-2]
    \arrow["{\Phi_{\vec{y}}}", from=1-2, to=2-2]
    \arrow[""{name=1, anchor=center, inner sep=0}, "{F \Sigma \vec{m}}"', "\shortmid"{marking}, from=2-1, to=2-2]
    \arrow["{\Phi_{\vec{m}}}"{description}, draw=none, from=0, to=1]
  \end{tikzcd}
  \qquad\leftrightsquigarrow\qquad
  \begin{tikzcd}
    {F x_i} & {F y_j} \\
    {F\Sigma\vec{x}} & {F\Sigma\vec{y}}
    \arrow["{F \iota_i}"', from=1-1, to=2-1]
    \arrow[""{name=0, anchor=center, inner sep=0}, "{F m_a}", "\shortmid"{marking}, from=1-1, to=1-2]
    \arrow["{F \iota_j}", from=1-2, to=2-2]
    \arrow[""{name=1, anchor=center, inner sep=0}, "{F\Sigma\vec{m}}"', "\shortmid"{marking}, from=2-1, to=2-2]
    \arrow["{F \iota_a}"{description}, draw=none, from=0, to=1]
  \end{tikzcd},
  \quad (i \xproto{a} j) : (I \xproto{A} J).
\end{equation*}
that is uniquely determined by the family of cells on the right, using the
universal property of the coproduct of $F\vec{m}$ in $\dbl{E}$. By definition,
the morphism of adjunctions $(\Phi, 1)$ is \define{strong} just when $\Phi$ is a
natural isomorphism. Thus, a colax functor $F$ preserves coproducts just when
the canonical comparisons $\Phi_{\vec{x}}$ and $\Phi_{\vec{m}}$ are all isomorphisms
in $\dbl{E}_0$ and $\dbl{E}_1$, respectively.

\begin{example}[Preserving coproducts of spans]
  \label{ex:preserve-coproducts-span}
  As observed in \citep[\S{C3.11}]{grandis2019}, for any functor
  $F: \cat{C} \to \cat{D}$ between categories with pullbacks, there is a unitary
  colax functor
  \begin{equation*}
    \Span{F}: \Span{\cat{C}} \to \Span{\cat{D}}
  \end{equation*}
  that applies $F$ pointwise to spans and maps between them, and $\Span{F}$ is
  pseudo if and only if $F$ preserves pullbacks.

  Suppose that $\cat{C}$ and $\cat{D}$ have coproducts, so that $\Span{\cat{C}}$
  and $\Span{\cat{D}}$ have colax coproducts (Theorem \ref{thm:coproducts-span}). Then
  the colax functor $\Span{F}$ preserves coproducts if and only if $F$ preserves
  coproducts in the usual sense.
\end{example}

To express that $\DblFam(\dbl{D})$ is the free coproduct cocompletion of
$\dbl{D}$, we work in the 2-category of double categories with colax coproducts,
colax functors that preserve coproducts, and natural transformations. However,
$\DblFam(\dbl{D})$ itself has strong coproducts, which turns out to be important
in the proof.

\begin{theorem}[Free double coproduct completion]
  \label{thm:free-dbl-coproduct-completion}
  For any double category $\dbl{D}$, the double category $\DblFam(\dbl{D})$ has
  strong coproducts. Moreover, $\DblFam(\dbl{D})$ is the free coproduct
  completion of $\dbl{D}$ in the sense that, for any colax functor
  $F: \dbl{D} \to \dbl{E}$ into a double category $\dbl{E}$ with colax
  coproducts, there exists a coproduct-preserving colax functor
  $\hat{F}: \DblFam(\dbl{D}) \to \dbl{E}$ making the triangle
  \begin{equation} \label{eq:free-dbl-coproduct-completion}
    \begin{tikzcd}
      & {\DblFam(\dbl{D})} \\
      {\dbl{D}} & {\dbl{E}}
      \arrow["{\hat{F}}", dashed, from=1-2, to=2-2]
      \arrow["\Delta", from=2-1, to=1-2]
      \arrow["F"', from=2-1, to=2-2]
    \end{tikzcd}
  \end{equation}
  commute, and $\hat{F}$ is unique up to natural isomorphism.
\end{theorem}
\begin{proof}
  We first show that $\DblFam(\dbl{D})$ has strong coproducts, which are
  reminiscent of coproducts in $\Span{\cat{S}}$ (Theorem \ref{thm:coproducts-span}). As
  reviewed in Theorem \ref{thm:free-coproduct-completion}, the underlying category
  $\DblFam(\dbl{D})_0 = \Fam(\dbl{D}_0)$ has coproducts and is in fact the free
  coproduct completion of $\dbl{D}_0$. So suppose that
  \begin{equation*}
    (A,(\vec{S},\vecvec{m})): (I,(\vec{U},\vecvec{x})) \proto (J,(\vec{V},\vecvec{y}))
  \end{equation*}
  is a family of proarrows in $\DblFam(\dbl{D})$, i.e., a proarrow in
  $\DblFam(\DblFam(\dbl{D}))$. This data comprises a span of sets
  $I \xfrom{\ell} A \xto{r} J$ together with a family of proarrows
  \begin{equation*}
    (S_a, \vec{m}_a): (U_i, \vec{x}_i) \proto (V_j, \vec{y}_j),
    \qquad (i \xproto{a} j) : (I \xproto{A} J),
  \end{equation*}
  in $\DblFam(\dbl{D})$. Each of the latter consists, in turn, of a span of sets
  $U_i \xfrom{\ell_a} S_a \xto{r_a} V_j$ together with a family of proarrows
  \begin{equation*}
    m_{a,s}: x_{i,u} \proto y_{j,v},
    \qquad (u \xproto{s} v) : (U_i \xproto{S_a} V_j),
  \end{equation*}
  in $\dbl{D}$. Now, a coproduct
  $\Sigma(\vec{S},\vecvec{m}): \Sigma(\vec{U},\vecvec{x}) \proto \Sigma(\vec{V},\vecvec{y})$
  of this family of families is constructed by taking the span
  $\sqcup\vec{U} \from \sqcup\vec{S} \to \sqcup\vec{V}$ defined by the
  copairings of the functions
  \begin{equation*}
    \sqcup\vec{U} \xfrom{\iota_i} U_i \xfrom{\ell_a} S_a
      \xto{r_a} V_j \xto{\iota_j} \sqcup\vec{V},
    \qquad (i \xproto{a} j) : (I \xproto{A} J),
  \end{equation*}
  together with the family of proarrows in $\dbl{D}$
  \begin{equation*}
    \begin{tikzcd}
      {\sqcup\vec{U}} & {\sqcup\vec{S}} & {\sqcup\vec{V}} \\
      {\dbl{D}_0} & {\dbl{D}_1} & {\dbl{D}_0}
      \arrow["{\vec{x}}"', from=1-1, to=2-1]
      \arrow[from=1-2, to=1-1]
      \arrow[from=1-2, to=1-3]
      \arrow["{\vec{m}}"', from=1-2, to=2-2]
      \arrow["{\vec{y}}", from=1-3, to=2-3]
      \arrow["s", from=2-2, to=2-1]
      \arrow["t"', from=2-2, to=2-3]
    \end{tikzcd},
  \end{equation*}
  where again the copairings $\vec{x} \coloneqq [\vec{x}_i]_{i \in I}$,
  $\vec{y} \coloneqq [\vec{y}_j]_{j \in J}$, and
  $\vec{m} \coloneqq [\vec{m}_a]_{a \in A}$ are defined using the universal
  property of coproducts in $\Set$. For each element $a: i \proto j$, the
  coprojection
  \begin{equation*}
    \begin{tikzcd}
      {(U_i, \vec{x}_i)} & {(V_j, \vec{y}_j)} \\
      {(\sqcup \vec{U}, \vec{x})} & {(\sqcup \vec{V}, \vec{y})}
      \arrow["{(\iota_i, 1_{\vec{x}_i})}"', from=1-1, to=2-1]
      \arrow[""{name=0, anchor=center, inner sep=0}, "{(S_a, \vec{m}_a)}", "\shortmid"{marking}, from=1-1, to=1-2]
      \arrow[""{name=1, anchor=center, inner sep=0}, "{(\sqcup \vec{S}, \vec{m})}"', "\shortmid"{marking}, from=2-1, to=2-2]
      \arrow["{(\iota_j, 1_{\vec{y}_j})}", from=1-2, to=2-2]
      \arrow["{(\iota_a, 1_{\vec{m}_a})}"{description}, draw=none, from=0, to=1]
    \end{tikzcd}
  \end{equation*}
  in $\DblFam(\dbl{D})$ consists of the map of spans
  \begin{equation*}
    \begin{tikzcd}
      {U_i} & {S_a} & {V_j} \\
      {\sqcup\vec{U}} & {\sqcup\vec{S}} & {\sqcup\vec{V}}
      \arrow["{\iota_i}"', from=1-1, to=2-1]
      \arrow["{\ell_a}"', from=1-2, to=1-1]
      \arrow["{r_a}", from=1-2, to=1-3]
      \arrow["{\iota_a}"', from=1-2, to=2-2]
      \arrow["{\iota_j}", from=1-3, to=2-3]
      \arrow[from=2-2, to=2-1]
      \arrow[from=2-2, to=2-3]
    \end{tikzcd},
    \hspace{1in}
    \begin{tikzcd}
      {x_{i,u}} & {y_{j,v}} \\
      {x_{\iota_i(u)}} & {y_{\iota_j(v)}}
      \arrow[Rightarrow, no head, from=1-1, to=2-1]
      \arrow[""{name=0, anchor=center, inner sep=0}, "{m_{a,s}}", "\shortmid"{marking}, from=1-1, to=1-2]
      \arrow[Rightarrow, no head, from=1-2, to=2-2]
      \arrow[""{name=1, anchor=center, inner sep=0}, "{m_{\iota_a(s)}}"', "\shortmid"{marking}, from=2-1, to=2-2]
      \arrow["1"{description}, draw=none, from=0, to=1]
    \end{tikzcd},
    \quad (u \xproto{s} v) : (U_i \xproto{S_a} V_j),
  \end{equation*}
  on the left and along with the family of identity cells in $\dbl{D}$ on the
  right. The universal property of the coproduct in $\DblFam(\dbl{D})$ follows
  easily from that of coproducts in $\Set$.

  We only sketch the proof that coproducts in $\DblFam(\dbl{D})$ are strong,
  being similar to the proof that coproducts in $\Span{\cat{S}}$ are strong
  (Theorem \ref{thm:coproducts-span}). For any composable pair of proarrows
  \begin{equation*}
    \begin{tikzcd}
      {(I,(\vec{U},\vecvec{x}))} & {(J,(\vec{V},\vecvec{y}))} & {(K,(\vec{W},\vecvec{z}))}
      \arrow["{(A,(\vec{S},\vecvec{m}))}", "\shortmid"{marking}, from=1-1, to=1-2]
      \arrow["{(B,(\vec{T},\vecvec{n}))}", "\shortmid"{marking}, from=1-2, to=1-3]
    \end{tikzcd}
  \end{equation*}
  in $\DblFam(\DblFam(\dbl{D}))$, we have canonical bijections
  \begin{equation*}
    \begin{tikzcd}
      {\displaystyle \bigsqcup_{(a,b) \in A \times_J B} S_a \times_{V_j} T_b} & {\displaystyle \bigsqcup_{(a,b) \in A \times B} S_a \times_{\sqcup \vec{V}} T_b} & {\sqcup\vec{S} \times_{\sqcup\vec{V}} \sqcup\vec{T}}
      \arrow["\cong", from=1-1, to=1-2]
      \arrow["\cong", from=1-2, to=1-3]
    \end{tikzcd}
  \end{equation*}
  due to the extensivity of $\Set$. The map between indexing spans in the
  canonical comparison cell is thus a bijection. Moreover, the components of
  that comparison cell are just the identities on the proarrow composites
  $m_{a,s} \odot n_{b,t}$.

  It remains to prove the universal property of the free coproduct completion.
  Any colax functor $\hat{F}: \DblFam(\dbl{D}) \to \dbl{E}$ making the triangle
  \eqref{eq:free-dbl-coproduct-completion} commute is uniquely determined on
  singleton families of objects and proarrows, and on maps between them, by
  $F: \dbl{D} \to \dbl{E}$. But arbitrary families of objects and proarrows in
  $\dbl{D}$ are coproducts in $\DblFam(\dbl{D})$ of singleton families, so if
  $\hat{F}$ is to preserve coproducts, it must send these to coproducts of fixed
  objects and proarrows in $\dbl{E}$. By functorality and the universal property
  of coproducts, $\hat{F}$ is then also determined on arrows and cells in
  $\DblFam(\dbl{D})$. Finally, the composition and identity comparisons of
  $\hat{F}$ are uniquely determined by those of $F$, using the dual of
  Lemma \ref{lem:laxator-unitor-product}, proved in detail later, along with the fact
  that coproducts in $\DblFam(\dbl{D})$ are strong. So, since $\dbl{E}$ is
  assumed to have colax coproducts, we can choose coproducts in $\dbl{E}$ to
  define $\hat{F}$, and such $\hat{F}$ is unique up to the choice of coproducts,
  hence up to natural isomorphism commuting with the triangle
  \eqref{eq:free-dbl-coproduct-completion}.
\end{proof}

\begin{corollary}
  The double category $\Span$ is the free coproduct completion of the terminal
  double category. Similarly, $\FinSpan$ is the free finite coproduct completion
  of the terminal double category.
\end{corollary}
\begin{proof}
  Immediate from the preceding theorem and Corollary \ref{cor:span-via-fam}.
\end{proof}

\section{Lax products in double categories}
\label{sec:products}

We turn now from coproducts to products in double categories. Of course,
products are dual to coproducts, so for certain abstract purposes nothing
further need be said. For other purposes that is not enough. Certainly, to
understand any examples, a concrete description of products in double categories
is needed, and we will give one. But there is a deeper divergence between double
products and coproducts. Famously, the symmetry arising from duality is partly
broken in category theory due to the special role played by the category $\Set$,
which is far from being self-dual. The same is true in double category theory
due to the special role played by the double category $\Span$, but to an even
greater degree. We will see that the prototypical double categories of spans and
of matrices have strong coproducts, as shown in Section \ref{sec:coproducts},
but have \emph{lax} products.

Let us recall the principle of duality in double category theory. The
\define{oppositization} 2-functor $\op: \Cat^\co \to \Cat$ sends a category to
its opposite category, preserving the orientation of functors and reversing that
of natural transformations. It also preserves 2-pullbacks. Thus, the
\define{opposite} $\dbl{D}^\op$ of a double category $\dbl{D}$ can be defined by
applying the oppositization 2-functor to $\dbl{D}$, viewed as a pseudocategory
in $\Cat$. So the categories underlying $\dbl{D}^\op$ are
$(\dbl{D}^\op)_0 = (\dbl{D}_0)^\op$ and $(\dbl{D}^\op)_1 = (\dbl{D}_1)^\op$.
Similarly, the \define{opposite} $F^\op: \dbl{D}^\op \to \dbl{E}^\op$ of a lax
or colax double functor $F: \dbl{D} \to \dbl{E}$ has underlying functors
$(F^\op)_0 = (F_0)^\op$ and $(F^\op)_1 = (F_1)^\op$. But the orientations of the
composition and identity comparisons are reversed, so that a lax functor has a
colax opposite and vice versa. The orientations of natural transformations
between double functors are also reserved. Altogether, then, there are
\define{oppositization} 2-functors
\begin{equation*}
  \op: \DblLax^\co \to \DblColax
  \qquad\text{and}\qquad
  \op: \DblColax^\co \to \DblLax
\end{equation*}
that are inverse to each other. These both restrict on pseudo double functors to
an involutive 2-functor $\op: \Dbl^\co \to \Dbl$.

\begin{definition}[Contravariant double families]
  The \define{contravariant double category of families} in a double category
  $\dbl{D}$ is the double category $\DblFamOp(\dbl{D})$ defined by the relation
  \begin{equation*}
    \DblFamOp(\dbl{D})^\op = \DblFam(\dbl{D}^\op).
  \end{equation*}
  For contrast, $\DblFam(\dbl{D})$ can also be called the \define{covariant}
  double category of families.
\end{definition}

The covariant and contravariant double categories of families have the same
objects and proarrows but different arrows and cells. Concretely, the double
category $\DblFamOp(\dbl{D})$ has
\begin{itemize}
  \item as objects, an \define{indexing set} $I$ together with an $I$-indexed
    family $\vec{x}: I \to \dbl{D}_0$, consisting of objects $x_i$ in $\dbl{D}$
    for $i \in I$;
  \item as arrows $(I, \vec{x}) \to (J, \vec{y})$, a function $f_0: J \to I$
    together with a $J$-indexed family $f$ of morphisms in $\dbl{D}$ of the form
    $f_j: x_{f_0(j)} \to y_j$ for $j \in J$;
  \item as proarrows $(I, \vec{x}) \proto (J, \vec{y})$, an \define{indexing
    span} $I \xfrom{\ell} A \xto{r} J$ together with an $A$-indexed family
    $\vec{m}: A \to \dbl{D}_1$, consisting of proarrows in $\dbl{D}$ of the form
    $m_a: x_i \proto y_j$ for each $a: i \proto j$;
  \item as cells
    $\inlineCell{(I,\vec{x})}{(J,\vec{y})}{(K,\vec{w})}{(L,\vec{z})}{(A,\vec{m})}{(B,\vec{n})}{(f_0,f)}{(g_0,g)}{}$,
    a map of spans $(f_0,\alpha_0,g_0)$ as shown on the left below together with
    a $B$-indexed family $\alpha$ of cells in $\dbl{D}$ of the form on the
    right:
    \begin{equation*}
      \begin{tikzcd}
        I & A & J \\
        K & B & L
        \arrow["{f_0}", from=2-1, to=1-1]
        \arrow["{g_0}"', from=2-3, to=1-3]
        \arrow["\ell"', from=1-2, to=1-1]
        \arrow["\ell", from=2-2, to=2-1]
        \arrow["r"', from=2-2, to=2-3]
        \arrow["{\alpha_0}", from=2-2, to=1-2]
        \arrow["r", from=1-2, to=1-3]
      \end{tikzcd}
      \hspace{1in}
      \begin{tikzcd}
        {x_{f_0(k)}} & {y_{g_0(\ell)}} \\
        {w_k} & {z_{\ell}}
        \arrow["{f_k}"', from=1-1, to=2-1]
        \arrow["{g_{\ell}}", from=1-2, to=2-2]
        \arrow[""{name=0, anchor=center, inner sep=0}, "{m_{\alpha_0(b)}}", "\shortmid"{marking}, from=1-1, to=1-2]
        \arrow[""{name=1, anchor=center, inner sep=0}, "{n_b}"', "\shortmid"{marking}, from=2-1, to=2-2]
        \arrow["{\alpha_b}"{description}, draw=none, from=0, to=1]
      \end{tikzcd},
      \quad (k \xproto{b} \ell) : (K \xproto{B} L).
    \end{equation*}
\end{itemize}
In this description, recall from Section \ref{sec:dbl-family-construction} that
given a span $I \xleftarrow{\ell} A \xrightarrow{r} J$, the notation
$a: i \proto j$ or $(i \xproto{a} j) : (I \xproto{A} J)$ refers to any element
$a \in A$ of the span's apex and asserts that $i = \ell(a)$ and $j = r(a)$.

The duality principle for double categories extends to double adjunctions. Like
any 2-functor, the oppositization 2-functor $\op: \Cat^\co \to \Cat$ preserves
adjunctions, but the ``co'' implies that left and right are exchanged. So an
adjunction $F \dashv G$ becomes an adjunction $G^\op \dashv F^\op$. Similarly,
the oppositization 2-functor $\op: \DblLax^\co \to \DblColax$ turns a pseudo-lax
adjunction $F \dashv G$ into a colax-pseudo adjunction $G^\op \dashv F^\op$,
whereas the 2-functor $\op: \DblColax^\co \to \DblLax$ turns a colax-pseudo
adjunction into a pseudo-lax adjunction. The following definition is thus seen
to be dual to Definition \ref{def:dbl-coproducts} in the sense that lax products in
$\dbl{D}$ are colax coproducts in $\dbl{D}^\op$.

\begin{definition}[Double products] \label{def:dbl-products}
  A double category $\dbl{D}$ \define{has lax products} if the double functor
  $\Delta: \dbl{D} \to \DblFamOp(\dbl{D})$ has a lax right adjoint:
  \begin{equation*}
    \begin{tikzcd}
      {\dbl{D}} & {\DblFamOp(\dbl{D})}
      \arrow[""{name=0, anchor=center, inner sep=0}, "\Delta", curve={height=-18pt}, from=1-1, to=1-2]
      \arrow[""{name=1, anchor=center, inner sep=0}, "\Pi", curve={height=-18pt}, from=1-2, to=1-1]
      \arrow["\dashv"{anchor=center, rotate=-90}, draw=none, from=0, to=1]
    \end{tikzcd}.
  \end{equation*}
  The double category $\dbl{D}$ \define{has strong products} if the right
  adjoint $\Pi: \DblFamOp(\dbl{D}) \to \dbl{D}$ is pseudo.
\end{definition}

As noted in the Introduction, double products of this kind were first proposed
in a talk by Robert Paré \citep{pare2009}. We proceed with some straightforward
dualizations of Section \ref{sec:coproducts}.

\begin{proposition}[Double products via universal arrows]
  \label{prop:dbl-products-universal}
  A double category $\dbl{D}$ has lax products if and only if
  \begin{enumerate}[(i)]
    \item for every object family $(I,\vec{x})$ in $\dbl{D}$, there is a choice
      of object $\Pi \vec{x} \coloneq \Pi(I,\vec{x})$ in $\dbl{D}$ and a
      universal arrow
      \begin{equation*}
        (!,\pi): \Delta \Pi \vec{x} \to (I, \vec{x})
      \end{equation*}
      from the functor $\Delta_0: \dbl{D}_0 \to \FamOp(\dbl{D}_0)$ to
      $(I,\vec{x})$;
    \item for every proarrow family
      $(A,\vec{m}): (I,\vec{x}) \proto (J,\vec{y})$ in $\dbl{D}$, there is a
      choice of proarrow
      $\Pi \vec{m} \coloneqq \Pi(A,\vec{m}): \Pi\vec{x} \proto \Pi\vec{y}$ in
      $\dbl{D}$, compatible with the choices above, and a universal arrow
      \begin{equation*}
        \begin{tikzcd}
          {\Delta \Pi \vec{x}} & {\Delta \Pi \vec{y}} \\
          {(I, \vec{x})} & {(J, \vec{y})}
          \arrow[""{name=0, anchor=center, inner sep=0}, "{\Delta \Pi \vec{m}}", "\shortmid"{marking}, from=1-1, to=1-2]
          \arrow["{(!,\pi)}"', from=1-1, to=2-1]
          \arrow["{(!,\pi)}", from=1-2, to=2-2]
          \arrow[""{name=1, anchor=center, inner sep=0}, "{(A, \vec{m})}"', "\shortmid"{marking}, from=2-1, to=2-2]
          \arrow["{(!,\pi)}"{description}, draw=none, from=0, to=1]
        \end{tikzcd}
      \end{equation*}
      from the functor $\Delta_1: \dbl{D}_1 \to \DblFamOp(\dbl{D})_1$ to
      $(A,\vec{m})$.
  \end{enumerate}
  Moreover, in this case, the products are strong if and only if, for every pair
  of composable families
  $(I,\vec{x}) \xproto{(A,\vec{m})} (J,\vec{y}) \xproto{(B,\vec{n})} (K,\vec{z})$
  and for every family $(I,\vec{x})$, the canonical comparison cells
  \begin{equation*}
    \begin{tikzcd}
      {\Pi\vec{x}} & {\Pi\vec{y}} & {\Pi\vec{z}} \\
      {\Pi\vec{x}} && {\Pi\vec{z}}
      \arrow[""{name=0, anchor=center, inner sep=0}, "{\Pi(\vec{m} \odot \vec{n})}"', "\shortmid"{marking}, from=2-1, to=2-3]
      \arrow[Rightarrow, no head, from=2-1, to=1-1]
      \arrow[Rightarrow, no head, from=2-3, to=1-3]
      \arrow["{\Pi\vec{m}}", "\shortmid"{marking}, from=1-1, to=1-2]
      \arrow["{\Pi\vec{n}}", "\shortmid"{marking}, from=1-2, to=1-3]
      \arrow["{\Pi_{\vec{m},\vec{n}}}"{description, pos=0.6}, draw=none, from=0, to=1-2]
    \end{tikzcd}
    \qquad\text{and}\qquad
    \begin{tikzcd}
      {\Pi\vec{x}} & {\Pi\vec{x}} \\
      {\Pi\vec{x}} & {\Pi\vec{x}}
      \arrow[""{name=0, anchor=center, inner sep=0}, "{\id_{\Pi\vec{x}}}", "\shortmid"{marking}, from=1-1, to=1-2]
      \arrow[Rightarrow, no head, from=2-1, to=1-1]
      \arrow[Rightarrow, no head, from=2-2, to=1-2]
      \arrow[""{name=1, anchor=center, inner sep=0}, "{\Pi(\id_{\vec{x}})}"', "\shortmid"{marking}, from=2-1, to=2-2]
      \arrow["{\Pi_{\vec{x}}}"{description}, draw=none, from=1, to=0]
    \end{tikzcd}
  \end{equation*}
  are isomorphisms in $\dbl{D}_1$.
\end{proposition}
\begin{proof}
  This proposition is dual to Proposition \ref{prop:dbl-coproducts-universal};
  alternatively, it follows immediately from \citep[\mbox{Theorem
    3.6}]{grandis2004} or \citep[\mbox{Theorem 4.3.6}]{grandis2019}.
\end{proof}

The canonical comparison cells in the final statement of the proposition are the
unique solutions to the equations
\begin{equation} \label{eq:product-composition-comparison}
  \begin{tikzcd}
    {\Pi\vec{x}} & {\Pi\vec{y}} & {\Sigma\vec{z}} \\
    {\Pi\vec{x}} && {\Pi\vec{z}} \\
    {x_i} && {z_k}
    \arrow[""{name=0, anchor=center, inner sep=0}, "{\Pi(\vec{m} \odot \vec{n})}"', "\shortmid"{marking}, from=2-1, to=2-3]
    \arrow[Rightarrow, no head, from=1-1, to=2-1]
    \arrow[Rightarrow, no head, from=1-3, to=2-3]
    \arrow["{\Pi\vec{m}}", "\shortmid"{marking}, from=1-1, to=1-2]
    \arrow["{\Pi\vec{n}}", "\shortmid"{marking}, from=1-2, to=1-3]
    \arrow["{\pi_i}"', from=2-1, to=3-1]
    \arrow["{\pi_k}", from=2-3, to=3-3]
    \arrow[""{name=1, anchor=center, inner sep=0}, "{m_a \odot m_b}"', "\shortmid"{marking}, from=3-1, to=3-3]
    \arrow["{\Pi_{\vec{m},\vec{n}}}"{description, pos=0.4}, draw=none, from=1-2, to=0]
    \arrow["{\pi_{(a,b)}}"{description, pos=0.6}, draw=none, from=0, to=1]
  \end{tikzcd}
  =
  \begin{tikzcd}
    {\Pi\vec{x}} & {\Pi\vec{y}} & {\Pi\vec{z}} \\
    {x_i} & {y_j} & {z_k}
    \arrow[""{name=0, anchor=center, inner sep=0}, "{m_a}"', "\shortmid"{marking}, from=2-1, to=2-2]
    \arrow[""{name=1, anchor=center, inner sep=0}, "{m_b}"', "\shortmid"{marking}, from=2-2, to=2-3]
    \arrow["{\pi_i}"', from=1-1, to=2-1]
    \arrow["{\pi_j}"{description}, from=1-2, to=2-2]
    \arrow["{\pi_k}", from=1-3, to=2-3]
    \arrow[""{name=2, anchor=center, inner sep=0}, "{\Pi\vec{n}}", "\shortmid"{marking}, from=1-2, to=1-3]
    \arrow[""{name=3, anchor=center, inner sep=0}, "{\Pi\vec{m}}", "\shortmid"{marking}, from=1-1, to=1-2]
    \arrow["{\pi_b}"{description}, draw=none, from=2, to=1]
    \arrow["{\pi_a}"{description}, draw=none, from=3, to=0]
  \end{tikzcd},
  \qquad (i \xproto{a} j \xproto{b} k) : (I \xproto{A} J \xproto{B} K).
\end{equation}
and
\begin{equation} \label{eq:product-identity-comparison}
  \begin{tikzcd}
    {\Pi\vec{x}} & {\Pi\vec{x}} \\
    {\Pi\vec{x}} & {\Pi\vec{x}} \\
    {x_i} & {x_i}
    \arrow[""{name=0, anchor=center, inner sep=0}, "{\id_{\Pi\vec{x}}}", "\shortmid"{marking}, from=1-1, to=1-2]
    \arrow[Rightarrow, no head, from=1-1, to=2-1]
    \arrow[Rightarrow, no head, from=1-2, to=2-2]
    \arrow[""{name=1, anchor=center, inner sep=0}, "{\Pi(\id_{\vec{x}})}"', "\shortmid"{marking}, from=2-1, to=2-2]
    \arrow["{\pi_i}"', from=2-1, to=3-1]
    \arrow["{\pi_i}", from=2-2, to=3-2]
    \arrow[""{name=2, anchor=center, inner sep=0}, "{\id_{x_i}}"', "\shortmid"{marking}, from=3-1, to=3-2]
    \arrow["{\pi_i}"{description, pos=0.6}, draw=none, from=1, to=2]
    \arrow["{\Pi_{\vec{x}}}"{description}, draw=none, from=0, to=1]
  \end{tikzcd}
  =
  \begin{tikzcd}
    {\Pi\vec{x}} & {\Pi\vec{x}} \\
    {x_i} & {x_i}
    \arrow[""{name=0, anchor=center, inner sep=0}, "{\id_{x_i}}"', "\shortmid"{marking}, from=2-1, to=2-2]
    \arrow["{\pi_i}"', from=1-1, to=2-1]
    \arrow["{\pi_i}", from=1-2, to=2-2]
    \arrow[""{name=1, anchor=center, inner sep=0}, "{\id_{\Pi\vec{x}}}", "\shortmid"{marking}, from=1-1, to=1-2]
    \arrow["{\id_{\pi_i}}"{description}, draw=none, from=1, to=0]
  \end{tikzcd},
  \qquad i \in I,
\end{equation}
afforded by the universal properties of the universal arrows. The comparison
cells form the laxators and unitors of the lax double functor
$\Pi: \DblFamOp(\dbl{D}) \to \dbl{D}$.

Double categories of spans are the prime example of double categories with lax
products, and we will see that certain products of spans \emph{are} genuinely
lax. As preparation to describe products of spans, recall that the category of
elements construction has a right adjoint \citep[\mbox{Proposition 8}]{pare1990}.

\begin{construction}[Adjoint elements constructions]
  \label{def:elements-construction}
  Fixing a category $\cat{C}$, there is a functor
  $\int^{\cat{C}}: \Set^{\cat{C}} \to \CatOne$ that sends each copresheaf $X$ on
  $\cat{C}$ to its category of elements $\int X$. Conversely, for any category
  $\cat{X}$, there is a copresheaf $\El_{/\cat{C}}(\cat{X})$ on $\cat{C}$ whose
  elements of type $c \in \cat{C}$ are the generalized elements of $\cat{X}$ of
  shape $c/\cat{C}$. In other words,
  \begin{equation*}
    \El_{/\cat{C}}(\cat{X}) \coloneqq
    \CatOne(-/\cat{C}, \cat{X}): \cat{C} \to \Set.
  \end{equation*}
  This construction is functorial in $\cat{X}$, and it is right adjoint to the
  category of elements:
  \begin{equation*}
    \begin{tikzcd}
      {\Set^{\cat{C}}} & \CatOne
      \arrow[""{name=0, anchor=center, inner sep=0}, "{\int^{\cat{C}}}", curve={height=-12pt}, from=1-1, to=1-2]
      \arrow[""{name=1, anchor=center, inner sep=0}, "{\El_{/\cat{C}}}", curve={height=-12pt}, from=1-2, to=1-1]
      \arrow["\dashv"{anchor=center, rotate=-90}, draw=none, from=0, to=1]
    \end{tikzcd}.
  \end{equation*}
\end{construction}

\begin{theorem}[Products of spans] \label{thm:products-span}
  Let $\cat{S}$ be a (finitely) complete category. Then the double category
  $\Span{\cat{S}}$ has lax (finite) products.
\end{theorem}
\begin{proof}
  Consider the infinitary case. By assumption, the underlying category
  $\Span{\cat{S}}_0 = \cat{S}$ has products. Suppose, then, that
  $(A,\vec{m}): (I,\vec{x}) \proto (J,\vec{y})$ is a family of spans in
  $\cat{S}$. Writing $\cat{span} \coloneqq \mbox{$\{\bullet \from \bullet \to \bullet\}$}$ for the
  walking span, the span of indexing sets $I \from A \to J$ can be identified with
  a copresheaf $F: \cat{span} \to \Set$, and the family of spans in $\cat{S}$ on
  the left
  \begin{equation*}
    \begin{tikzcd}
      I & A & J \\
      {\cat{S}} & {\Span{\cat{S}}_1} & {\cat{S}}
      \arrow["{\vec{x}}"', from=1-1, to=2-1]
      \arrow["s", from=2-2, to=2-1]
      \arrow[from=1-2, to=1-1]
      \arrow["t"', from=2-2, to=2-3]
      \arrow[from=1-2, to=1-3]
      \arrow["{\vec{y}}", from=1-3, to=2-3]
      \arrow["{\vec{m}}"', from=1-2, to=2-2]
    \end{tikzcd}
    \qquad\leftrightsquigarrow\qquad
    \begin{tikzcd}
      {\cat{span}} & \Set
      \arrow[""{name=0, anchor=center, inner sep=0}, "F", curve={height=-18pt}, from=1-1, to=1-2]
      \arrow[""{name=1, anchor=center, inner sep=0}, "{\El_{/\cat{span}}(\cat{S})}"', curve={height=18pt}, from=1-1, to=1-2]
      \arrow["{\vec{m}}"', shorten <=5pt, shorten >=5pt, Rightarrow, from=0, to=1]
    \end{tikzcd}
  \end{equation*}
  can be identified with a natural transformation as on the right. The latter is
  in turn equivalent, via the adjunction in Construction
  \ref{def:elements-construction}, to a diagram $M: \int F \to \cat{S}$ whose shape
  is the category of elements of $F: \cat{span} \to \Set$. We now construct the
  product of the spans $(A,\vec{m}): (I,\vec{x}) \proto (J,\vec{y})$ by taking
  limits of the diagrams in $\cat{S}$ shown on the left below:
  \begin{equation*}
    \begin{tikzcd}
      I & {\int F} & J \\
      & {\cat{S}}
      \arrow[hook, from=1-1, to=1-2]
      \arrow["{\vec{x}}"', from=1-1, to=2-2]
      \arrow["M"'{pos=0.4}, from=1-2, to=2-2]
      \arrow["{\vec{y}}", from=1-3, to=2-2]
      \arrow[hook', from=1-3, to=1-2]
    \end{tikzcd}
    \qquad\xmapsto{\lim}\qquad
    \Pi\vec{x} \from \Pi\vec{m} \to \Pi\vec{y}.
  \end{equation*}
  Writing the original spans as $m_a = (x_i \xfrom{\ell_a} s_a \xto{r_a} y_j)$ for
  each $a: i \proto j$, the apex of the product span can be depicted
  schematically as
  \begin{equation*}
    \Pi\vec{m} \coloneqq \lim_{\int F} M =
    \lim_{\substack{\text{$i$ or $j$ or} \\ a: i' \proto j'}} \left(
      \begin{tikzcd}
        {x_i} & {s_a} & {y_j} \\
        {x_{i'}} & {} & {y_{j'}}
        \arrow["\vdots"{description}, draw=none, from=1-1, to=2-1]
        \arrow["{\ell_a}"', from=1-2, to=2-1]
        \arrow["{r_a}", from=1-2, to=2-3]
        \arrow["\vdots"{description}, draw=none, from=1-3, to=2-3]
        \arrow["\vdots"{description}, draw=none, from=1-2, to=2-2]
      \end{tikzcd}
    \right).
  \end{equation*}
  The projection cells are the evident maps of spans
  \begin{equation*}
    \begin{tikzcd}
      {\Pi\vec{x}} & {\Pi\vec{m}} & {\Pi\vec{y}} \\
      {x_i} & {s_a} & {y_k}
      \arrow["{\pi_i}"', from=1-1, to=2-1]
      \arrow["{\ell_a}", from=2-2, to=2-1]
      \arrow["{r_a}"', from=2-2, to=2-3]
      \arrow["{\pi_a}"', from=1-2, to=2-2]
      \arrow[from=1-2, to=1-1]
      \arrow[from=1-2, to=1-3]
      \arrow["{\pi_j}", from=1-3, to=2-3]
    \end{tikzcd},
    \qquad (i \xproto{a} j) : (I \xproto{A} J),
  \end{equation*}
  where the squares commute because the apex $\Pi\vec{m}$ of the top span is, in
  particular, the apex of a cone over the diagram $M: \int F \to \cat{S}$.

  To prove the universal property, suppose given families of maps
  $f: \Delta x \to \vec{x}$ and $g: \Delta y \to \vec{y}$ along with a family of
  maps of spans
  \begin{equation*}
    \begin{tikzcd}
      x & s & y \\
      {x_i} & {s_a} & {y_j}
      \arrow["{f_i}"', from=1-1, to=2-1]
      \arrow["\ell"', from=1-2, to=1-1]
      \arrow["r", from=1-2, to=1-3]
      \arrow["{h_a}"', from=1-2, to=2-2]
      \arrow["{\ell_a}", from=2-2, to=2-1]
      \arrow["{g_j}", from=1-3, to=2-3]
      \arrow["{r_a}"', from=2-2, to=2-3]
    \end{tikzcd},
    \qquad (i \xproto{a} j) : (I \xproto{A} J).
  \end{equation*}
  Taken together, the maps $f_i \circ l$, $g_j \circ r$, and $h_a$, constitute a
  cone over the diagram $M: \int F \to \cat{S}$ with apex $s$. By the universal
  property of the limit $\Pi\vec{m} = \lim M$ in $\cat{S}$, there exists a
  unique map $h: s \to \Pi\vec{m}$ factoring each $f_i \circ \ell$ through
  $\pi_i$, each $g_j \circ r$ through $\pi_j$, and each $h_a$ through $\pi_a$.
  Putting $f \coloneqq \pair{f_i}_{i \in I}$ and
  $g \coloneqq \pair{g_j}_{j \in J}$, this is equivalent to there being a unique
  map of spans
  \begin{equation*}
    \begin{tikzcd}
      x & s & y \\
      {\Pi\vec{x}} & {\Pi\vec{m}} & {\Pi\vec{y}}
      \arrow["h"', dashed, from=1-2, to=2-2]
      \arrow["f"', from=1-1, to=2-1]
      \arrow[from=2-2, to=2-1]
      \arrow["\ell"', from=1-2, to=1-1]
      \arrow["r", from=1-2, to=1-3]
      \arrow["g", from=1-3, to=2-3]
      \arrow[from=2-2, to=2-3]
    \end{tikzcd}
  \end{equation*}
  factoring the span map $(f_i, h_a, g_j)$ through the projection
  $(\pi_i, \pi_a, \pi_j)$ for each $a: i \proto j$.
\end{proof}

\begin{corollary}[Coproducts of cospans]
  Let $\cat{S}$ be a (finitely) cocomplete category. Then the double category
  $\Cospan{\cat{S}}$ has colax (finite) coproducts.
\end{corollary}
\begin{proof}
  This follows by duality, since $\Cospan{\cat{S}}^\op = \Span{\cat{S}^\op}$.
\end{proof}

\begin{corollary}
  The double category $\Span$ has lax products, and $\FinSpan$ has lax finite
  products. Dually, $\Cospan$ has colax coproducts and $\FinCospan$ has colax
  finite coproducts.
\end{corollary}
\begin{proof}
  The category $\Set$ is complete, and $\FinSet$ is finitely complete.
\end{proof}

We now prove that double categories of matrices have lax products when the base
category has products. We do not yet assume that the monoidal product in the
base category is the cartesian product.

\begin{proposition}[Products of matrices] \label{prop:products-mat}
  For any distributive monoidal category $\catV$ with (finite) products, the
  double category $\Mat{\catV}$ has lax (finite) products.
\end{proposition}
\begin{proof}
  Clearly, the underlying category $\Mat{\catV}_0 = \Set$ has products. Given a
  family $(A,\vec{M}): (I,\vec{X}) \proto (J,\vec{Y})$ of $\catV$-matrices, the
  product $\Pi\vec{M}: \Pi\vec{X} \proto \Pi\vec{Y}$ is the $\catV$-matrix
  defined pointwise by the products in $\catV$
  \begin{equation*}
    \Pi\vec{M}(\vec{x}, \vec{y}) \coloneqq \prod_{a: i \proto j} M_a(x_i, y_j),
    \qquad \vec{x} \in \Pi\vec{X} = \prod_{i \in I} X_i,\
      \vec{y} \in \Pi\vec{Y} = \prod_{j \in J} Y_j,
  \end{equation*}
  where each product in $\catV$ ranges over all
  $(i \xproto{a} j) : (I \xproto{A} J)$. The associated projection cell shown on
  the left
  \begin{equation*}
    \begin{tikzcd}
      {\Pi\vec{X}} & {\Pi\vec{Y}} \\
      {X_i} & {Y_j}
      \arrow["{\pi_i}"', from=1-1, to=2-1]
      \arrow[""{name=0, anchor=center, inner sep=0}, "{\Pi\vec{M}}", "\shortmid"{marking}, from=1-1, to=1-2]
      \arrow["{\pi_j}", from=1-2, to=2-2]
      \arrow[""{name=1, anchor=center, inner sep=0}, "{M_a}"', "\shortmid"{marking}, from=2-1, to=2-2]
      \arrow["{\pi_a}"{description}, draw=none, from=0, to=1]
    \end{tikzcd}
    \qquad\leftrightsquigarrow\qquad
    \begin{tikzcd}[row sep=small]
      {\Pi\vec{M}(\vec{x},\vec{y})} \\
      {M_a(x_i, y_j)} \\
      {M_a(\pi_i(\vec{x}), \pi_j(\vec{y}))}
      \arrow[Rightarrow, no head, from=2-1, to=3-1]
      \arrow["{\pi_a}"', from=1-1, to=2-1]
    \end{tikzcd},
    \quad \begin{aligned}
      & \vec{x} \in \Pi\vec{X} \\
      & \vec{y} \in \Pi\vec{Y}
    \end{aligned},
  \end{equation*}
  is given by the family of projection maps in $\catV$ on the right. The
  universal property follows straightforwardly from that of products in $\catV$.
\end{proof}

Products in a double category specialize to reproduce a number of familiar
concepts. Two especially important cases are parallel products and local
products, both highlighted in Paré's original talk on double products
\citep{pare2009}. We remark that parallel products have also been defined as
limits of identity loose transformations \citep[\S{4.3}]{grandis1999}, whereas
local products have been defined using Kan extensions in double categories
\citep[\S{7.5}]{grandis2008}.

\begin{example}[Parallel products] \label{ex:parallel-products}
  A \define{parallel product} in a double category $\dbl{D}$ is a product in
  $\dbl{D}$ of a family of proarrows of form
  $(\id_I,\vec{m}): (I,\vec{x}) \proto (I,\vec{y})$, where $\id_I$ is the
  identity span \mbox{$I \xfrom{1} I \xto{1} I$}. The product
  $\Pi\vec{m}: \Pi\vec{x} \to \Pi\vec{y}$ is traditionally denoted
  \begin{equation*}
    \prod_{i \in I} m_i: \prod_{i \in I} x_i \proto \prod_{i \in I} y_i.
  \end{equation*}
  Let us isolate this special case by saying that a double category $\dbl{D}$
  \define{has lax parallel products} if, for every set $I$, the diagonal double
  functor $\Delta_I: \dbl{D} \to \dbl{D}^I$ has a lax right adjoint
  $\times_I: \dbl{D}^I \to \dbl{D}$, and \define{has strong parallel products}
  if the right adjoints are pseudo. With this terminology, a double category is
  \define{precartesian} in the sense of Aleiferi \citep[\S{4.1}]{aleiferi2018}
  just when it has lax finite parallel products, and it is \define{cartesian}
  just when it has strong finite parallel products.

  In the double category $\Span{\cat{S}}$, a parallel product of spans is
  computed pointwise as the products in $\cat{S}$ of the apexes and of the left
  and right feet. This can be seen directly, or as a special case of
  Theorem \ref{thm:products-span}. In $\Mat{\catV}$, the product of matrices
  $M_i: X_i \proto Y_i$, $i \in I$, is
  \begin{equation*}
    \Pi\vec{M}(\vec{x}, \vec{y}) = \prod_{i \in I} M_i(x_i, y_i),
    \qquad \vec{x} \in \prod_{i \in I} X_i,\ \vec{y} \in \prod_{i \in I} Y_i.
  \end{equation*}
\end{example}

\begin{example}[Local products] \label{ex:local-products}
  A \define{local product} in a double category $\dbl{D}$ is a product in
  $\dbl{D}$ of a proarrow family of form $(t_I,\vec{m}): (1,x) \proto (1,y)$,
  where $t_I$ is the span $1 \xfrom{!} I \xto{!} 1$. The local product
  $\Pi\vec{m} \coloneqq \Pi(t_I,\vec{m})$ can be denoted
  $\prod_{i \in I} m_i: x \proto y$.

  As the name suggests, a local product in a double category $\dbl{D}$ is, in
  particular, a local product in the underlying bicategory of $\dbl{D}$, i.e., a
  product in the hom-category $\dbl{D}(x,y)$ of proarrows $x \proto y$ and
  globular cells between them. However, it is important to realize that the
  double-categorical universal property is stronger than the bicategorical one.
  The double-categorical universal property says that for any arrows
  $f: w \to x$ and $g: z \to y$ and any family of cells
  $\inlineCell{w}{z}{x}{y}{n}{m_i}{f}{g}{\alpha_i}$ for $i \in I$, there exists
  a unique cell $\alpha$ factoring each cell $\alpha_i$ through the projection
  $\pi_i$:
  \begin{equation*}
    \begin{tikzcd}
      w & z \\
      x & y
      \arrow["f"', from=1-1, to=2-1]
      \arrow["g", from=1-2, to=2-2]
      \arrow[""{name=0, anchor=center, inner sep=0}, "n", "\shortmid"{marking}, from=1-1, to=1-2]
      \arrow[""{name=1, anchor=center, inner sep=0}, "{m_i}"', "\shortmid"{marking}, from=2-1, to=2-2]
      \arrow["{\alpha_i}"{description}, draw=none, from=0, to=1]
    \end{tikzcd}
    =
    \begin{tikzcd}
      w & z \\
      x & y \\
      x & y
      \arrow["f"', from=1-1, to=2-1]
      \arrow["g", from=1-2, to=2-2]
      \arrow[""{name=0, anchor=center, inner sep=0}, "n", "\shortmid"{marking}, from=1-1, to=1-2]
      \arrow[""{name=1, anchor=center, inner sep=0}, "{\Pi\vec{m}}", "\shortmid"{marking}, from=2-1, to=2-2]
      \arrow[Rightarrow, no head, from=2-1, to=3-1]
      \arrow[Rightarrow, no head, from=2-2, to=3-2]
      \arrow[""{name=2, anchor=center, inner sep=0}, "{m_i}"', "\shortmid"{marking}, from=3-1, to=3-2]
      \arrow["\alpha"{description, pos=0.4}, draw=none, from=0, to=1]
      \arrow["{\pi_i}"{description}, draw=none, from=1, to=2]
    \end{tikzcd},
    \qquad i \in I.
  \end{equation*}
  The bicategorical universal property restricts $f$ and $g$ to be identity
  arrows.

  For example, in $\Span{\cat{S}}$, the local product of spans
  $m_i = (x \xfrom{\ell_i} s_i \xto{r_i} y)$, $i \in I$, has as its apex the
  limit of the diagram in $\cat{S}$
  \begin{equation*}
    \begin{tikzcd}[row sep=tiny]
      & {s_i} \\
      x & \vdots & y \\
      & {s_{i'}}
      \arrow["{\ell_i}"', from=1-2, to=2-1]
      \arrow["{\ell_{i'}}", from=3-2, to=2-1]
      \arrow["{r_i}", from=1-2, to=2-3]
      \arrow["{r_{i'}}"', from=3-2, to=2-3]
    \end{tikzcd}
  \end{equation*}
  (Theorem \ref{thm:products-span}) or equivalently the product of the objects
  $(s_i, \pair{\ell_i,r_i})_{i \in I}$ in the slice category
  \mbox{$\cat{S}/(x \times y)$}. In $\Mat{\catV}$, the local product of matrices
  $M_i: X \proto Y$, $i \in I$, is simply the pointwise product in $\catV$:
  \begin{equation*}
    \Pi\vec{M}(x,y) = \prod_{i \in I} M_i(x,y), \qquad x \in X,\ y \in Y.
  \end{equation*}
\end{example}

Another special class of products in a double category, those built out of
identity proarrows, will be studied Section \ref{sec:restrictions}. Before that, we
need to determine when products can be expected to commute with external
composition.

\section{Iso-strong products in double categories}
\label{sec:iso-strong-products}

It is easy to find examples showing that products in a double category $\dbl{S}$
of spans or matrices constitute a genuinely lax functor
$\Pi: \DblFamOp(\dbl{S}) \to \dbl{S}$. Indeed, a ``generic'' choice of proarrow
families has this property. Here is a simple example involving local products.
Let $\catV$ be a distributive category and let
\begin{equation*}
  \begin{tikzcd}
    X & Y & Z
    \arrow["{M_1}", "\shortmid"{marking}, curve={height=-12pt}, from=1-1, to=1-2]
    \arrow["{M_2}"', "\shortmid"{marking}, curve={height=12pt}, from=1-1, to=1-2]
    \arrow["N", "\shortmid"{marking}, from=1-2, to=1-3]
  \end{tikzcd}
\end{equation*}
be $\catV$-matrices. Then, on the one hand, we have
\begin{equation} \label{eq:lax-local-product-1}
  (\Pi\vec{M} \odot N)(x,z) = ((M_1 \times M_2) \odot N)(x,z) \cong
    \sum_{y \in Y} M_1(x,y) \times M_2(x,y) \times N(y,z)
\end{equation}
for each $x \in X$ and $z \in Z$. On the other hand, by distributivity, we have
\begin{equation} \label{eq:lax-local-product-2}
  \begin{aligned}
  \Pi(\vec{M} \odot \Delta N)(x,z)
    &= ((M_1 \odot N) \times (M_2 \odot N))(x,z) \\
    &\cong \sum_{(y,y') \in Y \times Y}
      M_1(x,y) \times N(y,z) \times M_2(x,y') \times N(y',z).
  \end{aligned}
\end{equation}
In the canonical comparison, each constituent map in $\catV$ sends the summand
indexed by $y$ in Equation \eqref{eq:lax-local-product-1} to the summand indexed
by the diagonal pair $(y,y)$ in \eqref{eq:lax-local-product-2}. These maps are
hardly ever isomorphisms.

The class of proarrow products expected to commute, up to isomorphism, with
external composition and identities are singled out in the following definition.

\begin{definition}[Iso-strong products] \label{def:iso-strong-products}
  A double category $\dbl{D}$ has \define{iso-strong products} when
  \begin{enumerate}[(i),noitemsep]
    \item $\dbl{D}$ has lax products;
    \item for every pair of proarrow families
      $(I,\vec{x}) \xproto{(A,\vec{m})} (J,\vec{y}) \xproto{(B,\vec{n})} (K,\vec{z})$
      in $\dbl{D}$ such that the legs $A \to J$ and $J \from B$ are isomorphisms
      (bijections), the composition comparison cell
      \begin{equation*}
        \Pi_{\vec{m},\vec{n}}: \Pi\vec{m} \odot \Pi\vec{n} \to \Pi(\vec{m} \odot \vec{n})
      \end{equation*}
      is an isomorphism; and
    \item for every object family $(I,\vec{x})$ in $\dbl{D}$, the identity
      comparison cell
      \begin{equation*}
        \Pi_{\vec{x}}: \id_{\Pi\vec{x}} \to \Pi(\id_{\vec{x}})
      \end{equation*}
      is an isomorphism.
  \end{enumerate}
\end{definition}

A double category with iso-strong finite products has, in particular, strong
finite parallel products (Example \ref{ex:parallel-products}) and hence is a
cartesian double category in the sense of \citep{aleiferi2018}.

\begin{remark}[Normal and unitary products] \label{def:normal-products}
  The last condition in Definition \ref{def:iso-strong-products} can be called having
  \define{normal lax products} since it is equivalent to the existence of a
  \emph{normal} lax functor \mbox{$\Pi: \DblFamOp(\dbl{D}) \to \dbl{D}$}. It is
  reasonable to include in the definition of iso-strong products because it
  involves the identity family $\id_{(I,\vec{x})} = (\id_I, \id_{\vec{x}})$,
  whose indexing span has bijective legs.

  As a practical matter, in double categories with normal lax products, it is
  straightforward to choose products so that
  $\Pi: \DblFamOp(\dbl{D}) \to \dbl{D}$ is a \emph{unitary} lax functor
  \citep[\S{3.5.2}]{grandis2019}, i.e., the equations
  $\Pi(\id_{\vec{x}}) = \id_{\Pi\vec{x}}$ hold strictly. A double category with
  such a choice of products is said to be equipped with \define{unitary lax
    products}. We will not assume that such a choice as been made, although it
  can be convenient to do so.
\end{remark}

\begin{proposition}[Iso-strong products of spans]
  \label{prop:iso-strong-products-span}
  Let $\cat{S}$ be (finitely) complete category. Then the double category
  $\Span{\cat{S}}$ has iso-strong (finite) products.
\end{proposition}
\begin{proof}
  That $\Span{\cat{S}}$ has lax products has been shown in
  Theorem \ref{thm:products-span}, and that they are normal is both obvious and known
  from the study of $\Span{\cat{S}}$ as a cartesian double category. We
  explicitly construct the composition comparison $\Pi_{\vec{m},\vec{n}}$ for
  families of spans
  $(I,\vec{x}) \xproto{(A,\vec{m})} (J,\vec{y}) \xproto{(B,\vec{n})} (K,\vec{z})$.

  By reducing the pullback of limits to a single limit, the composite of
  products $\Pi\vec{m} \odot \Pi\vec{n}$ can be described similarly to the proof
  of Theorem \ref{thm:products-span}. Continuing to use that notation, let
  \begin{equation*}
    \cat{span}_\odot \coloneqq \cat{span} +_{\cat{1}} \cat{span}
      = \{\bullet \from \bullet \to \bullet \from \bullet \to \bullet\}
  \end{equation*}
  be the walking pair of composable spans. Then the given data on the left
  \begin{equation*}
    \begin{tikzcd}[column sep=scriptsize]
      I & A & J & B & K \\
      {\cat{S}} & {\Span{\cat{S}}_1} & {\cat{S}} & {\Span{\cat{S}}_1} & {\cat{S}}
      \arrow["s", from=2-2, to=2-1]
      \arrow["t"', from=2-2, to=2-3]
      \arrow["s", from=2-4, to=2-3]
      \arrow["t"', from=2-4, to=2-5]
      \arrow["{\vec{z}}", from=1-5, to=2-5]
      \arrow[from=1-4, to=1-5]
      \arrow[from=1-4, to=1-3]
      \arrow[from=1-2, to=1-3]
      \arrow[from=1-2, to=1-1]
      \arrow["{\vec{n}}"', from=1-4, to=2-4]
      \arrow["{\vec{y}}"', from=1-3, to=2-3]
      \arrow["{\vec{m}}"', from=1-2, to=2-2]
      \arrow["{\vec{x}}"', from=1-1, to=2-1]
    \end{tikzcd}
    \quad\leftrightsquigarrow\quad
    \begin{tikzcd}
      {\cat{span}_\odot} & \Set
      \arrow[""{name=0, anchor=center, inner sep=0}, "F", curve={height=-18pt}, from=1-1, to=1-2]
      \arrow[""{name=1, anchor=center, inner sep=0}, "{\El_{/\cat{span}_\odot}(\cat{S})}"', curve={height=18pt}, from=1-1, to=1-2]
      \arrow[shorten <=5pt, shorten >=5pt, Rightarrow, from=0, to=1]
    \end{tikzcd}
  \end{equation*}
  can be identified with a copresheaf $F: \cat{span}_\odot \to \Set$ along with
  a natural transformation as on the right, which is in turn equivalent to a
  diagram $D: \int F \to \cat{S}$ by Construction \ref{def:elements-construction}. The apex
  of the composite span $\Pi\vec{m} \odot \Pi\vec{n}$ is isomorphic to the limit
  of the diagam $D$ in $\cat{S}$.

  The product of composites $\Pi(\vec{m} \odot \vec{n})$ can also be described
  in this way by reducing the limit of pullbacks to a single limit. As before,
  the data on the left
  \begin{equation*}
    \begin{tikzcd}
      I & {A \times_J B} & {A \times_J B} & {A \times_J B} & K \\
      {\cat{S}} & {\Span{\cat{S}}_1} & {\cat{S}} & {\Span{\cat{S}}_1} & {\cat{S}}
      \arrow[from=1-2, to=1-1]
      \arrow["{\vec{x}}"', from=1-1, to=2-1]
      \arrow[Rightarrow, no head, from=1-2, to=1-3]
      \arrow[Rightarrow, no head, from=1-4, to=1-3]
      \arrow[from=1-4, to=1-5]
      \arrow["{\vec{y} \circ \pi_J}"', from=1-3, to=2-3]
      \arrow["{\vec{z}}", from=1-5, to=2-5]
      \arrow["s", from=2-2, to=2-1]
      \arrow["t"', from=2-2, to=2-3]
      \arrow["s", from=2-4, to=2-3]
      \arrow["t"', from=2-4, to=2-5]
      \arrow["{\vec{m} \circ \pi_A}"', from=1-2, to=2-2]
      \arrow["{\vec{n} \circ \pi_B}"', from=1-4, to=2-4]
    \end{tikzcd}
    \quad\leftrightsquigarrow\quad
    \begin{tikzcd}
      {\cat{span}_\odot} & \Set
      \arrow[""{name=0, anchor=center, inner sep=0}, "{F'}", curve={height=-18pt}, from=1-1, to=1-2]
      \arrow[""{name=1, anchor=center, inner sep=0}, "{\El_{/\cat{span}_\odot}(\cat{S})}"', curve={height=18pt}, from=1-1, to=1-2]
      \arrow[shorten <=5pt, shorten >=5pt, Rightarrow, from=0, to=1]
    \end{tikzcd}
  \end{equation*}
  is equivalent to a copresheaf $F'$ and a natural transformation as on the
  right, which is equivalent to a diagram $D': \int F' \to \cat{S}$. The apex of
  the span $\Pi(\vec{m} \odot \vec{n})$ is isomorphic to the limit of $D'$ in
  $\cat{S}$.

  By the functorality of the elements construction, the morphism of copresheaves
  \begin{equation*}
    \begin{tikzcd}
      I & A & J & B & K \\
      I & {A \times_J B} & {A \times_J B} & {A \times_J B} & K
      \arrow[from=2-2, to=2-1]
      \arrow[Rightarrow, no head, from=2-2, to=2-3]
      \arrow[Rightarrow, no head, from=2-4, to=2-3]
      \arrow[from=2-4, to=2-5]
      \arrow[Rightarrow, no head, from=2-1, to=1-1]
      \arrow[Rightarrow, no head, from=2-5, to=1-5]
      \arrow["{\pi_B}", from=2-4, to=1-4]
      \arrow["{\pi_A}", from=2-2, to=1-2]
      \arrow["{\pi_J}", from=2-3, to=1-3]
      \arrow[from=1-2, to=1-1]
      \arrow[from=1-2, to=1-3]
      \arrow[from=1-4, to=1-3]
      \arrow[from=1-4, to=1-5]
    \end{tikzcd}
    \quad\leftrightsquigarrow\quad
    \begin{tikzcd}
      {\cat{span}_\odot} & \Set
      \arrow[""{name=0, anchor=center, inner sep=0}, "F", curve={height=-18pt}, from=1-1, to=1-2]
      \arrow[""{name=1, anchor=center, inner sep=0}, "{F'}"', curve={height=18pt}, from=1-1, to=1-2]
      \arrow[shorten <=5pt, shorten >=5pt, Rightarrow, from=1, to=0]
    \end{tikzcd}
  \end{equation*}
  induces a functor $\int F' \to \int F$ that is also a morphism $D' \to D$ in
  the slice category $\CatOne/\cat{S}$ of diagrams in $\cat{S}$. The
  functorality of limits with respect to morphisms in $(\CatOne/\cat{S})^\op$
  yields the apex map in the comparison cell
  $\Pi_{\vec{m},\vec{n}}: \Pi\vec{m} \odot \Pi\vec{n} \to \Pi(\vec{m} \odot \vec{n})$.
  Now, if the legs $A \to J$ and $J \from B$ are isomorphisms, then so are all
  the projections $\pi_A$, $\pi_B$, and $\pi_J$ out of the pullback
  $A \times_J B$. Thus, $F' \To F$ is a natural isomorphism and, by the
  functorality of the subsequent constructions, the comparison cell
  $\Pi_{\vec{m},\vec{n}}$ is also an isomorphism.
\end{proof}

\begin{corollary}[Iso-strong coproducts of cospans]
  Let $\cat{S}$ be a (finitely) cocomplete category. Then the double category
  $\Cospan{\cat{S}}$ has iso-strong (finite) coproducts.
\end{corollary}

The proof makes it clear that essentially the \emph{only} condition under which
products of spans, or dually coproducts of cospans, commute with external
composition is the iso-strong condition in Definition \ref{def:iso-strong-products}.

Under reasonable hypotheses, double categories of matrices also have iso-strong
products. By a \emph{distributive category}, we will mean a category with finite
products and arbitrary coproducts, over which the products distribute.

\begin{proposition}[Iso-strong products of matrices]
  \label{prop:iso-strong-products-mat}
  For any distributive category $\catV$, the double category $\Mat{\catV}$ has
  iso-strong finite products.
\end{proposition}
\begin{proof}
  That $\Mat{\catV}$ has lax finite products has been shown under weaker
  hypotheses in Proposition \ref{prop:products-mat}, and that they are normal under the
  present hypotheses follows from the isomorphisms $1 \times 1 \cong 1$ and
  $1 \times 0 \cong 0$ in $\catV$. So suppose that
  $(I,\vec{X}) \xto{(A,\vec{M})} (J,\vec{Y}) \xto{(B,\vec{N})} (K,\vec{Z})$ are
  families of $\catV$-matrices such that the legs $A \to J$ and $B \to J$ are
  isomorphisms. Denoting the inverse functions by $a: J \to A$ and $b: J \to B$,
  we have the isomorphisms
  \begin{equation*}
    (\Pi\vec{M} \odot \Pi\vec{N})(\vec{x},\vec{z}) \cong
      \sum_{\vec{y} \in \Pi\vec{Y}} \prod_{j \in J}
        M_{a(j)}(x_{\ell(a(j))}, y_j) \times N_{b(j)}(y_j, z_{r(b(j))})
  \end{equation*}
  and
  \begin{equation*}
    \Pi(\vec{M} \odot \vec{N})(\vec{x},\vec{z}) \cong
      \prod_{j \in J} \sum_{y_j \in Y_j}
        M_{a(j)}(x_{\ell(a(j))}, y_j) \times N_{b(j)}(y_j, z_{r(b(j))})
  \end{equation*}
  in $\catV$ for each $\vec{x} \in \Pi\vec{X}$ and $\vec{y} \in \Pi\vec{Y}$. The
  canonical distributivity morphism
  \begin{equation*}
    \sum_{\vec{y} \in \Pi\vec{Y}} \prod_{j \in J} \longrightarrow
      \prod_{j \in J} \sum_{y_j \in Y_j}
  \end{equation*}
  from the first object in $\catV$ to the second, involving the dependent
  product $\Pi\vec{Y} = \prod_{j \in J} Y_j$ of sets, is an isomorphism since
  $\catV$ is a distributive category.
\end{proof}

\section{Restrictions in double categories with products}
\label{sec:restrictions}

In any double category with lax products, one can ask about the proarrows that
arise as products of identity proarrows for different choices of indexing span.
Such proarrows and accompanying cells carry a surprisingly intricate structure,
all following from the universal property of products. An example will
illustrate the general situation.

\begin{example}[Diagonal proarrows] \label{ex:diagonal-proarrows}
  Let $\dbl{D}$ be a double category with lax finite products. The
  \define{diagonal proarrow} on an object $x \in \dbl{D}$ is the product
  \begin{equation} \label{eq:diagonal-proarrow}
    \delta_x \coloneqq
    \Pi\left(
      \begin{tikzcd}[cramped,row sep=tiny,column sep=small]
        & x \\
        x \\
        & x
        \arrow["{\id_x}", "\shortmid"{marking}, from=2-1, to=1-2]
        \arrow["{\id_x}"', "\shortmid"{marking}, from=2-1, to=3-2]
      \end{tikzcd}
    \right) \coloneqq
    \Pi(1 \xfrom{!} 2 = 2,\ (\id_x, \id_x)),
  \end{equation}
  where the diagram of proarrows on the left is shorthand for the family of
  proarrows on the right. Besides the proarrow $\delta_x: x \proto x^2$, this
  product consists of two projection cells
  \begin{equation*}
    \begin{tikzcd}
      x & {x^2} \\
      x & x
      \arrow[""{name=0, anchor=center, inner sep=0}, "{\delta_x}", "\shortmid"{marking}, from=1-1, to=1-2]
      \arrow["{\pi_i}", from=1-2, to=2-2]
      \arrow[Rightarrow, no head, from=1-1, to=2-1]
      \arrow[""{name=1, anchor=center, inner sep=0}, "{\id_x}"', "\shortmid"{marking}, from=2-1, to=2-2]
      \arrow["{\pi_i}"{description}, draw=none, from=0, to=1]
    \end{tikzcd},
    \qquad i = 1,2.
  \end{equation*}
  From the diagonal $\delta_x$ and the parallel product $\id_x^2$, we can form
  \define{unit} and \define{counit} cells
  \begin{equation*}
    \begin{tikzcd}
      x & x \\
      x & {x^2}
      \arrow[""{name=0, anchor=center, inner sep=0}, "{\id_x}", "\shortmid"{marking}, from=1-1, to=1-2]
      \arrow[Rightarrow, no head, from=1-1, to=2-1]
      \arrow["{\Delta_x}", from=1-2, to=2-2]
      \arrow[""{name=1, anchor=center, inner sep=0}, "{\delta_x}"', "\shortmid"{marking}, from=2-1, to=2-2]
      \arrow["\eta"{description}, draw=none, from=0, to=1]
    \end{tikzcd}
    \qquad\text{and}\qquad
    \begin{tikzcd}
      x & {x^2} \\
      {x^2} & {x^2}
      \arrow[""{name=0, anchor=center, inner sep=0}, "{\delta_x}", "\shortmid"{marking}, from=1-1, to=1-2]
      \arrow["{\Delta_x}"', from=1-1, to=2-1]
      \arrow[Rightarrow, no head, from=1-2, to=2-2]
      \arrow[""{name=1, anchor=center, inner sep=0}, "{\id_x^2}"', "\shortmid"{marking}, from=2-1, to=2-2]
      \arrow["\varepsilon"{description}, draw=none, from=0, to=1]
    \end{tikzcd},
  \end{equation*}
  where $\Delta_x: x \to x^2$ is the usual diagonal map. The unit cell is the
  unique solution to the equations
  \begin{equation*}
    \begin{tikzcd}[row sep=scriptsize]
      x & x \\
      x & {x^2} \\
      x & x
      \arrow[""{name=0, anchor=center, inner sep=0}, "{\id_x}", "\shortmid"{marking}, from=1-1, to=1-2]
      \arrow[Rightarrow, no head, from=1-1, to=2-1]
      \arrow["{\Delta_x}", from=1-2, to=2-2]
      \arrow[""{name=1, anchor=center, inner sep=0}, "{\delta_x}", "\shortmid"{marking}, from=2-1, to=2-2]
      \arrow["{\pi_i}", from=2-2, to=3-2]
      \arrow[""{name=2, anchor=center, inner sep=0}, "{\id_x}"', "\shortmid"{marking}, from=3-1, to=3-2]
      \arrow[Rightarrow, no head, from=2-1, to=3-1]
      \arrow["{\pi_i}"{description}, draw=none, from=1, to=2]
      \arrow["\eta"{description, pos=0.4}, draw=none, from=0, to=1]
    \end{tikzcd}
    = 1_{\id_x} = \id_{1_x},
    \qquad i = 1,2,
  \end{equation*}
  given by the universal property of the product $\delta_x$. Similarly, the
  counit cell is the unique solution to the equations
  \begin{equation*}
    \begin{tikzcd}[row sep=scriptsize]
      x & {x^2} \\
      {x^2} & {x^2} \\
      x & x
      \arrow["{\Delta_x}"', from=1-1, to=2-1]
      \arrow["{\pi_i}"', from=2-1, to=3-1]
      \arrow[Rightarrow, no head, from=1-2, to=2-2]
      \arrow[""{name=0, anchor=center, inner sep=0}, "{\delta_x}", "\shortmid"{marking}, from=1-1, to=1-2]
      \arrow[""{name=1, anchor=center, inner sep=0}, "{\id_x^2}", "\shortmid"{marking}, from=2-1, to=2-2]
      \arrow["{\pi_i}", from=2-2, to=3-2]
      \arrow[""{name=2, anchor=center, inner sep=0}, "{\id_x}"', "\shortmid"{marking}, from=3-1, to=3-2]
      \arrow["{\pi_i}"{description}, draw=none, from=1, to=2]
      \arrow["\varepsilon"{description, pos=0.4}, draw=none, from=0, to=1]
    \end{tikzcd}
    =
    \begin{tikzcd}[sep=scriptsize]
      x & {x^2} \\
      x & x
      \arrow[""{name=0, anchor=center, inner sep=0}, "{\delta_x}", "\shortmid"{marking}, from=1-1, to=1-2]
      \arrow["{\pi_i}", from=1-2, to=2-2]
      \arrow[Rightarrow, no head, from=1-1, to=2-1]
      \arrow[""{name=1, anchor=center, inner sep=0}, "{\id_x}"', "\shortmid"{marking}, from=2-1, to=2-2]
      \arrow["{\pi_i}"{description}, draw=none, from=0, to=1]
    \end{tikzcd},
    \qquad i = 1,2,
  \end{equation*}
  given by the universal property of the parallel product $\id_x^2$.

  Now suppose that the double category $\dbl{D}$ has \emph{normal} lax finite
  products (Remark \ref{def:normal-products}). Then the codomain
  $\id_x^2: x^2 \proto x^2$ of the counit cell can be interchanged with
  $\id_{x^2}: x^2 \proto x^2$. We obtain the data that would make the arrow
  $\Delta_x: x \to x^2$ and proarrow $\delta_x: x \proto x^2$ into a companion
  pair, provided the cells satisfy the two compatibility equations. Later, we
  will prove that they do by a more abstract approach.

  Conversely, suppose that we start with a cartesian equipment $\dbl{E}$. For
  simplicity, assume that we have made a unitary choice of finite parallel
  products, so that $\id_x^I = \id_{x^I}$ for each finite set $I$ and object
  $x \in \dbl{E}$. Define the \define{diagonal proarrow}
  $\delta_x: x \proto x^2$ to be a companion to the diagonal arrow
  $\Delta_x: x \to x^2$. Then there is a counit cell of the form
  \begin{equation*}
    \begin{tikzcd}
      x & {x^2} \\
      {x^2} & {x^2}
      \arrow[""{name=0, anchor=center, inner sep=0}, "{\delta_x}", "\shortmid"{marking}, from=1-1, to=1-2]
      \arrow["{\Delta_x}"', from=1-1, to=2-1]
      \arrow[Rightarrow, no head, from=1-2, to=2-2]
      \arrow[""{name=1, anchor=center, inner sep=0}, "{\id_x^2}"', "\shortmid"{marking}, from=2-1, to=2-2]
      \arrow["\varepsilon"{description}, draw=none, from=0, to=1]
    \end{tikzcd},
  \end{equation*}
  which is also the restriction of the parallel product $\id_x^2 = \id_{x^2}$
  along the arrows $\Delta_x$ and $1_{x^2}$. Using the projections for the
  product $\id_x^2$, we can form a new pair of \define{projections}
  \begin{equation*}
    \begin{tikzcd}
      x & {x^2} \\
      x & x
      \arrow[""{name=0, anchor=center, inner sep=0}, "{\delta_x}", "\shortmid"{marking}, from=1-1, to=1-2]
      \arrow["{\pi_i}", from=1-2, to=2-2]
      \arrow[Rightarrow, no head, from=1-1, to=2-1]
      \arrow[""{name=1, anchor=center, inner sep=0}, "{\id_x}"', "\shortmid"{marking}, from=2-1, to=2-2]
      \arrow["{\pi_i}"{description}, draw=none, from=0, to=1]
    \end{tikzcd}
    \quad\coloneqq\quad
    \begin{tikzcd}[row sep=scriptsize]
      x & {x^2} \\
      {x^2} & {x^2} \\
      x & x
      \arrow[""{name=0, anchor=center, inner sep=0}, "{\delta_x}", "\shortmid"{marking}, from=1-1, to=1-2]
      \arrow["{\Delta_x}"', from=1-1, to=2-1]
      \arrow[Rightarrow, no head, from=1-2, to=2-2]
      \arrow[""{name=1, anchor=center, inner sep=0}, "{\id_x^2}", "\shortmid"{marking}, from=2-1, to=2-2]
      \arrow["{\pi_i}", from=2-2, to=3-2]
      \arrow["{\pi_i}"', from=2-1, to=3-1]
      \arrow[""{name=2, anchor=center, inner sep=0}, "{\id_x}"', "\shortmid"{marking}, from=3-1, to=3-2]
      \arrow["\varepsilon"{description, pos=0.4}, draw=none, from=0, to=1]
      \arrow["{\pi_i}"{description}, draw=none, from=1, to=2]
    \end{tikzcd},
    \qquad i = 1,2.
  \end{equation*}
  We obtain the data that would make $\delta_x$ into a product in $\dbl{E}$ as
  in Equation \eqref{eq:diagonal-proarrow}, provided the universal property is satisfied.
  That is indeed the case, as we will show later.
\end{example}

The rest of this section is dedicated to systematizing this example and
explaining the connection between cartesian equipments and double categories
with finite products. The key observations are that when a double category has
lax products, restrictions are closed under products, and when the products are
normal, companions and conjoints are also closed under products. The latter
result is quicker to prove, so we begin with it.

\begin{proposition}[Products of companions and conjoints]
  \label{prop:companions-product}
  Let $\dbl{D}$ be a double category with normal lax products. Suppose that
  $(f_0,f): (I,\vec{x}) \to (J,\vec{y})$ is a family of arrows
  $f_j: x_{f_0(j)} \to y_j$ in $\dbl{D}$ each having a companion
  $(f_j)_!: x_{f_0(j)} \proto y_j$. Then the product arrow
  $\Pi(f_0,f): \Pi\vec{x} \to \Pi\vec{y}$ also has a companion, namely the
  product $\Pi(f_0^*, f_!): \Pi\vec{x} \proto \Pi\vec{y}$ of the family of
  companions $(f_!)_j \coloneqq (f_j)_!$ indexed by the conjoint span
  $f_0^* = (I \xfrom{f_0} J = J)$.

  Dually, if each component $f_j$ has a conjoint $f_j^*: y_j \proto x_{f_0(j)}$,
  then the product arrow $\Pi(f_0,f): \Pi\vec{x} \to \Pi\vec{y}$ also has a
  conjoint, namely the product $\Pi((f_0)_!, f^*): \Pi\vec{y} \proto \Pi\vec{x}$
  of the family of conjoints $(f^*)_j \coloneqq f_j^*$ indexed by the companion
  span $(f_0)_! = (J = J \xto{f_0} I)$.
\end{proposition}
\begin{proof}
  Since conjoints in a double category are companions in the opposite double
  category and vice versa, and since
  $\DblFamOp(\dbl{D})^\op = \DblFam(\dbl{D}^\op)$, the dual of
  Proposition \ref{prop:companions-fam} states that the companion and conjoint of the arrow
  $(f_0,f)$ in $\DblFamOp(\dbl{D})$ are $(f_0^*, f_!)$ and $((f_0)_!, f^*)$,
  respectively. The result follows since the product operation
  $\Pi: \DblFamOp(\dbl{D}) \to \dbl{D}$, being a normal lax functor, preserves
  companions and conjoints \citep[\mbox{Proposition 3.8}]{dawson2010}.
\end{proof}

\begin{corollary}[Companions and conjoints of pairings]
  Let $f_i: x \to y_i$ be a family of arrows with common domain in a double
  category with normal lax products. If each arrow $f_i$ has a companion, then
  so does the pairing $\pair{f_i}_{i \in I}: x \to \prod_{i \in I} y_i$, and
  likewise for conjoints.
\end{corollary}
\begin{proof}
  This is the special case of Proposition \ref{prop:companions-product} with reindexing map
  of the form $f_0: I \xto{!} 1$.
\end{proof}

\begin{corollary}[Structure proarrows between products]
  \label{cor:structure-proarrows}
  Let $\dbl{D}$ be a double category with normal lax products. For any families
  of objects $(I,\vec{x})$ and $(J,\vec{y})$ in $\dbl{D}$ and function between
  indexing sets $f_0: J \to I$ such that $\vec{x} \circ f_0 = \vec{y}$, the
  product arrow $\Pi(f_0) \coloneqq \Pi(f_0,1): \Pi\vec{x} \to \Pi\vec{y}$ has
  both a companion $\Pi(f_0)_! = \Pi(f_0^*,\id): \Pi\vec{x} \proto \Pi\vec{y}$
  and a conjoint $\Pi(f_0)^* = \Pi((f_0)_!,\id): \Pi\vec{y} \proto \Pi\vec{x}$.
\end{corollary}
\begin{proof}
  This is another special case of Proposition \ref{prop:companions-product} since, for any
  object $x$ in a double category, the identity arrow $1_x: x \to x$ always has
  the identity proarrow $\id_x: x \proto x$ as both its companion and its
  conjoint \citep[\mbox{Lemma 3.12}]{shulman2010}.
\end{proof}

Such results can be stated more succinctly using a special notation for
reindexed families of objects and proarrows. Given a function $f_0: J \to I$ and
an $I$-indexed family of objects $\vec{x}$, denote by
$f_0^*(\vec{x}) \coloneq \vec{x} \circ f_0$ the $J$-indexed family obtained by
precomposing $\vec{x}$ with $f_0$. Similarly, given an $I$-indexed parallel
family of proarrows $\vec{m}: \vec{x} \proto \vec{y}$, denote by
$f_0^* \vec{m}: f_0^* \vec{x} \proto f_0^* \vec{y}$ the $J$-indexed parallel
family of proarrows obtained by precomposing all of $\vec{x}$, $\vec{y}$, and
$\vec{m}$ with $f_0$.

With this notation, Corollary \ref{cor:structure-proarrows} can be restated as
follows: for any function \mbox{$f_0: J \to I$}, the \define{structure arrow}
$\Pi(f_0): \Pi\vec{x} \to \Pi f_0^* \vec{x}$ between products has both a companion
$\Pi(f_0)_!: \Pi\vec{x} \proto \Pi f_0^* \vec{x}$ and a conjoint
$\Pi(f_0)^*: \Pi f_0^* \vec{x} \proto \Pi\vec{x}$. This corollary encompasses familiar
structure arrows like diagonals and projections. Specifically, in any double
category with normal lax finite products,
\begin{itemize}[noitemsep]
  \item diagonals $\Delta_x: x \to x^2$ have companions $\delta_x: x \proto x^2$
    and conjoints $\delta_x^*: x^2 \proto x$, via the unique function
    $f_0: 2 \xto{!} 1$ as constructed directly in Example \ref{ex:diagonal-proarrows};
  \item deletion arrows $!_x: x \to 1$ have companions
    $\varepsilon_x: x \proto 1$ and conjoints $\varepsilon_x^*: 1 \proto x$, via
    the function $f_0: 0 \xto{!} 1$;
  \item projections $\pi_x: x \times y \to x$ and $\pi_y: x \times y \to y$ have
    companions and conjoints, via the two inclusions $f_0: 1 \hookrightarrow 2$;
    and
  \item braidings $\sigma_{x,y}: x \times y \to y \times x$ have companions and
    conjoints, via the swap function $f_0: 2 \xto{\cong} 2$, which are mutually
    inverse up to globular isomorphism \citep[\mbox{Lemma 3.20}]{shulman2010}.
\end{itemize}
In fact, since companions and conjoints are functorial up to isomorphism, these
special cases essentially exhaust the content of Corollary \ref{cor:structure-proarrows}
for finite products.

As hinted by Example \ref{ex:diagonal-proarrows}, the normality assumption in
Proposition \ref{prop:companions-product} and its corollaries can be discarded while
retaining nearly the same universal property. We will see that structure
proarrows such as the diagonal $\delta_x: x \proto x^2$ and codiagonal
$\delta_x^*: x^2 \proto x$, while not necessarily companions or conjoints in a
double category with merely lax finite products, are still restrictions of
parallel products of identity proarrows.

\begin{proposition}[Restricting along structure arrows between products]
  \label{prop:restrictions-product}
  Let $\dbl{D}$ be a double category with lax products. Suppose given a niche in
  $\DblFamOp(\dbl{D})$
  \begin{equation*}
    \begin{tikzcd}[row sep=scriptsize]
      {(I,\vec{x})} & {(J,\vec{y})} \\
      {(K,\vec{w})} & {(L,\vec{z})}
      \arrow["{(B,\vec{n})}"', "\shortmid"{marking}, from=2-1, to=2-2]
      \arrow["{(f_0,f)}"', from=1-1, to=2-1]
      \arrow["{(g_0,g)}", from=1-2, to=2-2]
    \end{tikzcd}
  \end{equation*}
  such that each constituent niche in $\dbl{D}$
  \begin{equation*}
    \begin{tikzcd}[row sep=scriptsize]
      {x_{f_0(k)}} & {y_{g_0(\ell)}} \\
      {w_k} & {z_\ell}
      \arrow["{f_k}"', from=1-1, to=2-1]
      \arrow["{g_\ell}", from=1-2, to=2-2]
      \arrow["{n_b}"', "\shortmid"{marking}, from=2-1, to=2-2]
    \end{tikzcd},
    \qquad (k \xproto{b} \ell) : (K \xproto{B} L),
  \end{equation*}
  has a restriction. Then $\dbl{D}$ has a restriction cell of the form
  \begin{equation} \label{eq:product-restriction}
    \begin{tikzcd}[column sep=large]
      {\Pi\vec{x}} & {\Pi\vec{y}} \\
      {\Pi\vec{w}} & {\Pi\vec{z}}
      \arrow[""{name=0, anchor=center, inner sep=0}, "{\Pi\vec{n}}"', "\shortmid"{marking}, from=2-1, to=2-2]
      \arrow["{\Pi(f_0,f)}"', from=1-1, to=2-1]
      \arrow["{\Pi(g_0,g)}", from=1-2, to=2-2]
      \arrow[""{name=1, anchor=center, inner sep=0}, "{\Pi(\vec{n}(f,g))}", "\shortmid"{marking}, from=1-1, to=1-2]
      \arrow["\res"{description}, draw=none, from=1, to=0]
    \end{tikzcd}.
  \end{equation}
  Here the domain of the restriction cell is the product of the proarrow family
  $(B, \vec{n}(f,g))$, indexed by the span
  $I \xfrom{f_0} K \from B \to L \xto{g_0} J$ and comprising the restricted
  proarrows $n_b(f_k, g_\ell): x_{f_0(k)} \proto y_{g_0(\ell)}$ for each element
  $b: k \proto \ell$.
\end{proposition}
\begin{proof}
  By the dual of Proposition \ref{prop:extensions-fam}, the family of restriction cells
  \begin{equation*}
    \begin{tikzcd}[column sep=large]
      {(I,\vec{x})} & {(J,\vec{y})} \\
      {(K,\vec{w})} & {(L,\vec{z})}
      \arrow[""{name=0, anchor=center, inner sep=0}, "{(B,\vec{n})}"', "\shortmid"{marking}, from=2-1, to=2-2]
      \arrow["{(f_0,f)}"', from=1-1, to=2-1]
      \arrow["{(g_0,g)}", from=1-2, to=2-2]
      \arrow[""{name=1, anchor=center, inner sep=0}, "{(B, \vec{n}(f,g))}", "\shortmid"{marking}, from=1-1, to=1-2]
      \arrow["{(1_B, \res)}"{description}, draw=none, from=1, to=0]
    \end{tikzcd}
  \end{equation*}
  is itself a restriction cell in $\DblFamOp(\dbl{D})$. We define the required
  cell \eqref{eq:product-restriction} to be the product of this family of
  restrictions. We must show that the product is itself a restriction.

  Now, \emph{if} the double category $\dbl{D}$ were an equipment, then
  $\DblFamOp(\dbl{D})$ would also be an equipment by the dual of
  Corollary \ref{cor:equipment-fam}. The result would then follow immediately since the
  product operation $\Pi: \DblFamOp(\dbl{D}) \to \dbl{D}$, like any lax functor
  between equipments, preserves restrictions \citep[\mbox{Proposition
    6.4}]{shulman2008}. However, we are \emph{not} assuming that $\dbl{D}$ is an
  equipment, and the cited proof does not go through when that assumption is
  dropped. We instead present a more specialized analysis that uses the
  universal properties of both products and restrictions.

  We must show that any cell $\alpha$ of the form on the left
  \begin{equation} \label{eq:product-restriction-univeral-prop}
    \begin{tikzcd}
      x & y \\
      {\Pi\vec{x}} & {\Pi\vec{y}} \\
      {\Pi\vec{w}} & {\Pi\vec{z}}
      \arrow["{\Pi(f_0,f)}"', from=2-1, to=3-1]
      \arrow["q"', from=1-1, to=2-1]
      \arrow[""{name=0, anchor=center, inner sep=0}, "m", "\shortmid"{marking}, from=1-1, to=1-2]
      \arrow[""{name=1, anchor=center, inner sep=0}, "{\Pi\vec{n}}"', "\shortmid"{marking}, from=3-1, to=3-2]
      \arrow["{\Pi(g_0,g)}", from=2-2, to=3-2]
      \arrow["r", from=1-2, to=2-2]
      \arrow["\alpha"{description}, draw=none, from=0, to=1]
    \end{tikzcd}
    \quad=\quad
    \begin{tikzcd}
      x & y \\
      {\Pi\vec{x}} & {\Pi\vec{y}} \\
      {\Pi\vec{w}} & {\Pi\vec{z}}
      \arrow["{\Pi(f_0,f)}"', from=2-1, to=3-1]
      \arrow["q"', from=1-1, to=2-1]
      \arrow["r", from=1-2, to=2-2]
      \arrow[""{name=0, anchor=center, inner sep=0}, "m", "\shortmid"{marking}, from=1-1, to=1-2]
      \arrow[""{name=1, anchor=center, inner sep=0}, "{\Pi\vec{n}}"', "\shortmid"{marking}, from=3-1, to=3-2]
      \arrow["{\Pi(g_0,g)}", from=2-2, to=3-2]
      \arrow[""{name=2, anchor=center, inner sep=0}, "{\Pi \vec{n}(f,g)}", "\shortmid"{marking}, from=2-1, to=2-2]
      \arrow["\res"{description}, draw=none, from=2, to=1]
      \arrow["{\exists!\ \bar\alpha}"{description, pos=0.4}, draw=none, from=0, to=2]
    \end{tikzcd}
  \end{equation}
  has a unique factorization through the claimed restriction cell as on the
  right. Decomposing the arrows $q = \pair{q_i}_{i \in I}$ and
  $r = \pair{r_j}_{j \in J}$ into pairings, the given cell $\alpha$ is the
  pairing of the cells
  \begin{equation*}
    \begin{tikzcd}[column sep=scriptsize]
      x & y \\
      {x_{f_0(k)}} & {y_{g_0(\ell)}} \\
      {w_k} & {z_\ell}
      \arrow["{q_{f_0(k)}}"', from=1-1, to=2-1]
      \arrow[""{name=0, anchor=center, inner sep=0}, "m", "\shortmid"{marking}, from=1-1, to=1-2]
      \arrow["{r_{g_0(\ell)}}", from=1-2, to=2-2]
      \arrow["{f_k}"', from=2-1, to=3-1]
      \arrow["{g_\ell}", from=2-2, to=3-2]
      \arrow[""{name=1, anchor=center, inner sep=0}, "{n_b}"', "\shortmid"{marking}, from=3-1, to=3-2]
      \arrow["{\alpha_b}"{description}, draw=none, from=0, to=1]
    \end{tikzcd}
    \quad\coloneqq\quad
    \begin{tikzcd}
      x & y \\
      {\Pi\vec{w}} & {\Pi\vec{z}} \\
      {w_k} & {z_\ell}
      \arrow[""{name=0, anchor=center, inner sep=0}, "m", "\shortmid"{marking}, from=1-1, to=1-2]
      \arrow[""{name=1, anchor=center, inner sep=0}, "{\Pi\vec{n}}", "\shortmid"{marking}, from=2-1, to=2-2]
      \arrow[from=1-1, to=2-1]
      \arrow[from=1-2, to=2-2]
      \arrow["{\pi_k}"', from=2-1, to=3-1]
      \arrow["{\pi_\ell}", from=2-2, to=3-2]
      \arrow[""{name=2, anchor=center, inner sep=0}, "{n_b}"', "\shortmid"{marking}, from=3-1, to=3-2]
      \arrow["\alpha"{description, pos=0.4}, draw=none, from=0, to=1]
      \arrow["{\pi_b}"{description}, draw=none, from=1, to=2]
    \end{tikzcd},
    \qquad (k \xproto{b} \ell): (K \xproto{B} L),
  \end{equation*}
  with respect to the product $\Pi\vec{n}: \Pi\vec{w} \proto \Pi\vec{z}$. For
  each $b: k \proto \ell$, the universal property of the restriction
  $n_b(f_k, g_\ell)$ then gives a unique factorization
  \begin{equation*}
    \begin{tikzcd}[column sep=scriptsize]
      x & y \\
      {x_{f_0(k)}} & {y_{g_0(\ell)}} \\
      {w_k} & {z_\ell}
      \arrow["{q_{f_0(k)}}"', from=1-1, to=2-1]
      \arrow[""{name=0, anchor=center, inner sep=0}, "m", "\shortmid"{marking}, from=1-1, to=1-2]
      \arrow["{r_{g_0(\ell)}}", from=1-2, to=2-2]
      \arrow["{f_k}"', from=2-1, to=3-1]
      \arrow["{g_\ell}", from=2-2, to=3-2]
      \arrow[""{name=1, anchor=center, inner sep=0}, "{n_b}"', "\shortmid"{marking}, from=3-1, to=3-2]
      \arrow["{\alpha_b}"{description}, draw=none, from=0, to=1]
    \end{tikzcd}
    \quad=\quad
    \begin{tikzcd}
      x & y \\
      {x_{f_0(k)}} & {y_{g_0(\ell)}} \\
      {w_k} & {z_\ell}
      \arrow["{q_{f_0(k)}}"', from=1-1, to=2-1]
      \arrow[""{name=0, anchor=center, inner sep=0}, "m", "\shortmid"{marking}, from=1-1, to=1-2]
      \arrow["{r_{g_0(\ell)}}", from=1-2, to=2-2]
      \arrow["{f_k}"', from=2-1, to=3-1]
      \arrow["{g_\ell}", from=2-2, to=3-2]
      \arrow[""{name=1, anchor=center, inner sep=0}, "{n_b}"', "\shortmid"{marking}, from=3-1, to=3-2]
      \arrow[""{name=2, anchor=center, inner sep=0}, "{n_b(f_k,g_\ell)}", "\shortmid"{marking}, from=2-1, to=2-2]
      \arrow["{\res_b}"{description}, draw=none, from=2, to=1]
      \arrow["{\exists!\ \bar\alpha_b}"{description, pos=0.4}, draw=none, from=0, to=2]
    \end{tikzcd}.
  \end{equation*}
  Finally, we define the needed cell $\bar\alpha$ to be the pairing of the cells
  $\bar\alpha_b$ with respect to the product
  $\Pi(\vec{n}(f,g)): \Pi\vec{x} \proto \Pi\vec{y}$, i.e., the unique solution
  to the equations
  \begin{equation*}
    \begin{tikzcd}
      x & y \\
      {\Pi\vec{x}} & {\Pi\vec{y}} \\
      {x_{f_0(k)}} & {y_{g_0(\ell)}}
      \arrow["q"', from=1-1, to=2-1]
      \arrow["r", from=1-2, to=2-2]
      \arrow[""{name=0, anchor=center, inner sep=0}, "m", "\shortmid"{marking}, from=1-1, to=1-2]
      \arrow[""{name=1, anchor=center, inner sep=0}, "{\Pi \vec{n}(f,g)}", "\shortmid"{marking}, from=2-1, to=2-2]
      \arrow["{\pi_{f_0(k)}}"', from=2-1, to=3-1]
      \arrow["{\pi_{g_0(\ell)}}", from=2-2, to=3-2]
      \arrow[""{name=2, anchor=center, inner sep=0}, "{n_b(f_k, g_\ell)}"', "\shortmid"{marking}, from=3-1, to=3-2]
      \arrow["\bar\alpha"{description, pos=0.4}, draw=none, from=0, to=1]
      \arrow["{\pi_b}"{description}, draw=none, from=1, to=2]
    \end{tikzcd}
    \quad=\quad
    \begin{tikzcd}
      x & y \\
      {x_{f_0(k)}} & {y_{g_0(\ell)}}
      \arrow["{q_{f_0(k)}}"', from=1-1, to=2-1]
      \arrow[""{name=0, anchor=center, inner sep=0}, "m", "\shortmid"{marking}, from=1-1, to=1-2]
      \arrow["{r_{g_0(\ell)}}", from=1-2, to=2-2]
      \arrow[""{name=1, anchor=center, inner sep=0}, "{n_b(f_k,g_\ell)}"', "\shortmid"{marking}, from=2-1, to=2-2]
      \arrow["{\bar\alpha_b}"{description}, draw=none, from=0, to=1]
    \end{tikzcd},
    \qquad (k \xproto{b} \ell): (K \xproto{B} L).
  \end{equation*}

  By the above relations, supplemented by the equation
  \begin{equation*}
    \begin{tikzcd}
      {\Pi\vec{x}} & {\Pi\vec{y}} \\
      {x_{f_0(k)}} & {y_{g_0(\ell)}} \\
      {w_k} & {z_\ell}
      \arrow["{f_k}"', from=2-1, to=3-1]
      \arrow["{g_\ell}", from=2-2, to=3-2]
      \arrow[""{name=0, anchor=center, inner sep=0}, "{n_b}"', "\shortmid"{marking}, from=3-1, to=3-2]
      \arrow[""{name=1, anchor=center, inner sep=0}, "{n_b(f_k,g_\ell)}", "\shortmid"{marking}, from=2-1, to=2-2]
      \arrow["{\pi_{f_0(k)}}"', from=1-1, to=2-1]
      \arrow["{\pi_{g_0(\ell)}}", from=1-2, to=2-2]
      \arrow[""{name=2, anchor=center, inner sep=0}, "{\Pi\vec{n}(f,g)}", "\shortmid"{marking}, from=1-1, to=1-2]
      \arrow["{\res_b}"{description}, draw=none, from=1, to=0]
      \arrow["{\pi_b}"{description, pos=0.4}, draw=none, from=2, to=1]
    \end{tikzcd}
    \quad=\quad
    \begin{tikzcd}
      {\Pi\vec{x}} & {\Pi\vec{y}} \\
      {\Pi\vec{w}} & {\Pi\vec{z}} \\
      {w_k} & {z_\ell}
      \arrow[""{name=0, anchor=center, inner sep=0}, "{n_b}"', "\shortmid"{marking}, from=3-1, to=3-2]
      \arrow[""{name=1, anchor=center, inner sep=0}, "{\Pi\vec{n}(f,g)}", "\shortmid"{marking}, from=1-1, to=1-2]
      \arrow["{\Pi(f_0,f)}"', from=1-1, to=2-1]
      \arrow["{\Pi(g_0,g)}", from=1-2, to=2-2]
      \arrow[""{name=2, anchor=center, inner sep=0}, "{\Pi\vec{n}}", "\shortmid"{marking}, from=2-1, to=2-2]
      \arrow["{\pi_k}"', from=2-1, to=3-1]
      \arrow["{\pi_\ell}", from=2-2, to=3-2]
      \arrow["\res"{description, pos=0.4}, draw=none, from=1, to=2]
      \arrow["{\pi_b}"{description}, draw=none, from=2, to=0]
    \end{tikzcd}
  \end{equation*}
  which is a consequence of the functorality of products, the post-composite of
  Equation \eqref{eq:product-restriction-univeral-prop} with the projection
  $\pi_b: \Pi\vec{n} \to n_b$ holds for every $b: k \proto \ell$, hence
  Equation \eqref{eq:product-restriction-univeral-prop} holds by the universal property of
  the product $\Pi\vec{n}$. This proves the existence part of the universal
  property of the restriction. Uniqueness follows from the uniqueness in each
  step of the above construction.
\end{proof}

We can go further to characterize products in a double category in terms of
restrictions and parallel products (Example \ref{ex:parallel-products}). The preceding
proposition gives one half of the proof.

\begin{theorem}[Double products via restrictions]
  \label{thm:dbl-products-characterization}
  A double category has lax (finite) products if and only if it has
  \begin{enumerate}[(i),nosep]
    \item lax (finite) parallel products, and
    \item restrictions of parallel products along structure arrows between
      products.
  \end{enumerate}
  Moreover, in this case, the products are iso-strong if and only if the
  parallel products are strong.
\end{theorem}
\begin{proof}
  For the proof of the first characterization, let
  $(A,\vec{m}): (I,\vec{x}) \proto (J,\vec{y})$ be any family of proarrows in a
  double category $\dbl{D}$, indexed by the span $I \xfrom{f_0} A \xto{g_0} J$.

  Suppose $\dbl{D}$ has lax products. Then the forward implication follows
  immediately from Proposition \ref{prop:restrictions-product}. However, it will
  be instructive to unpack this further. For every $a \in A$, the proarrow
  $m_a: x_{f_0(a)} \proto y_{g_0(a)}$ trivially has a restriction along identity
  arrows, namely itself. Thus, by Proposition \ref{prop:restrictions-product},
  the product $\Pi\vec{m}: \Pi\vec{x} \proto \Pi\vec{y}$ is the restriction of the
  parallel product $\Pi\vec{m}: \Pi f_0^* \vec{x} \to \Pi g_0^* \vec{y}$ along the
  structure arrows $\Pi(f_0)$ and $\Pi(g_0)$:
  \begin{equation} \label{eq:parallel-product-restriction}
    \begin{tikzcd}
      {\Pi\vec{x}} & {\Pi\vec{y}} \\
      {\Pi f_0^* \vec{x}} & {\Pi g_0^* \vec{y}}
      \arrow["{\Pi(f_0)}"', from=1-1, to=2-1]
      \arrow["{\Pi(g_0)}", from=1-2, to=2-2]
      \arrow[""{name=0, anchor=center, inner sep=0}, "{\Pi\vec{m}}"', "\shortmid"{marking}, from=2-1, to=2-2]
      \arrow[""{name=1, anchor=center, inner sep=0}, "{\Pi\vec{m}}", "\shortmid"{marking}, from=1-1, to=1-2]
      \arrow["\res"{description}, draw=none, from=1, to=0]
    \end{tikzcd}.
  \end{equation}
  We have shown that an \emph{arbitrary} product in $\dbl{D}$ admits a canonical
  decomposition as the restriction of a parallel product along structure arrows.

  Now suppose $\dbl{D}$ is a double category with lax parallel products and
  restrictions of these along structure arrows. We show that, conversely, the
  proarrow $\Pi\vec{m}: \Pi\vec{x} \proto \Pi\vec{y}$ now \emph{defined} to be the
  restriction \eqref{eq:parallel-product-restriction} is a product of the family
  $(A,\vec{m}): (I,\vec{x}) \proto (J,\vec{y})$ when it is equipped with the
  projection cells
  \begin{equation*}
    \begin{tikzcd}
      {\Pi\vec{x}} & {\Pi\vec{y}} \\
      {x_i} & {y_j}
      \arrow[""{name=0, anchor=center, inner sep=0}, "{\Pi\vec{m}}", "\shortmid"{marking}, from=1-1, to=1-2]
      \arrow["{\pi_i}"', from=1-1, to=2-1]
      \arrow["{\pi_j}", from=1-2, to=2-2]
      \arrow[""{name=1, anchor=center, inner sep=0}, "{m_a}"', "\shortmid"{marking}, from=2-1, to=2-2]
      \arrow["{\pi_a}"{description}, draw=none, from=0, to=1]
    \end{tikzcd}
    \quad\coloneqq\quad
    \begin{tikzcd}
      {\Pi\vec{x}} & {\Pi\vec{y}} \\
      {\Pi f_0^* \vec{x}} & {\Pi g_0^* \vec{y}} \\
      {x_{f_0(a)}} & {y_{g_0(a)}}
      \arrow[""{name=0, anchor=center, inner sep=0}, "{\Pi\vec{m}}", "\shortmid"{marking}, from=1-1, to=1-2]
      \arrow[""{name=1, anchor=center, inner sep=0}, "{m_a}"', "\shortmid"{marking}, from=3-1, to=3-2]
      \arrow["{\Pi(f_0)}"', from=1-1, to=2-1]
      \arrow["{\Pi(g_0)}", from=1-2, to=2-2]
      \arrow["{\pi_a}"', from=2-1, to=3-1]
      \arrow["{\pi_a}", from=2-2, to=3-2]
      \arrow[""{name=2, anchor=center, inner sep=0}, "{\Pi\vec{m}}", "\shortmid"{marking}, from=2-1, to=2-2]
      \arrow["\res"{description, pos=0.4}, draw=none, from=0, to=2]
      \arrow["{\pi_a}"{description}, draw=none, from=2, to=1]
    \end{tikzcd},
    \qquad (i \xproto{a} j): (I \xproto{A} J).
  \end{equation*}
  To prove the universal property, suppose given a family of cells
  $\inlineCell{w}{z}{x_i}{y_j}{n}{m_a}{f_i}{g_j}{\alpha_a}$ indexed by elements
  \mbox{$a: i \proto j$}. Abbreviating the families $h_a \coloneqq f_{f_0(a)}$ and
  $k_a \coloneqq g_{g_0(a)}$, we have equations
  \begin{equation*}
    \Pi(f_0) \circ \pair{f_i}_{i \in I} = \pair{h_a}_{a \in A}
    \qquad\text{and}\qquad
    \Pi(g_0) \circ \pair{g_j}_{j \in J} = \pair{k_a}_{a \in A}.
  \end{equation*}
  We can therefore form the pairing with respect to the parallel product as on
  the left
  \begin{equation*}
    \begin{tikzcd}
      w & z \\
      {\Pi f_0^* \vec{x}} & {\Pi g_0^* \vec{y}}
      \arrow[""{name=0, anchor=center, inner sep=0}, "{\Pi\vec{m}}"', "\shortmid"{marking}, from=2-1, to=2-2]
      \arrow["{\pair{h_a}_{a \in A}}"', from=1-1, to=2-1]
      \arrow["{\pair{k_a}_{a \in A}}", from=1-2, to=2-2]
      \arrow[""{name=1, anchor=center, inner sep=0}, "n", "\shortmid"{marking}, from=1-1, to=1-2]
      \arrow["{\pair{\alpha_a}_{a \in A}}"{description}, draw=none, from=1, to=0]
    \end{tikzcd}
    \quad=\quad
    \begin{tikzcd}
      w & z \\
      {\Pi\vec{x}} & {\Pi\vec{y}} \\
      {\Pi f_0^* \vec{x}} & {\Pi g_0^* \vec{y}}
      \arrow[""{name=0, anchor=center, inner sep=0}, "{\Pi\vec{m}}"', "\shortmid"{marking}, from=3-1, to=3-2]
      \arrow[""{name=1, anchor=center, inner sep=0}, "n", "\shortmid"{marking}, from=1-1, to=1-2]
      \arrow["{\pair{f_i}_{i \in I}}"', from=1-1, to=2-1]
      \arrow["{\pair{g_j}_{j \in J}}", from=1-2, to=2-2]
      \arrow[""{name=2, anchor=center, inner sep=0}, "{\Pi\vec{m}}", "\shortmid"{marking}, from=2-1, to=2-2]
      \arrow["{\Pi(f_0)}"', from=2-1, to=3-1]
      \arrow["{\Pi(g_0)}", from=2-2, to=3-2]
      \arrow["{\exists!\ \alpha}"{description, pos=0.4}, draw=none, from=1, to=2]
      \arrow["\res"{description}, draw=none, from=2, to=0]
    \end{tikzcd}.
  \end{equation*}
  and then factorize it through the restriction cell to obtain a unique cell
  $\alpha$ as on the right. By construction, the cell $\alpha$ is a solution to
  the equations
  \begin{equation*}
    \begin{tikzcd}
      w & z \\
      {\Pi\vec{x}} & {\Pi\vec{y}} \\
      {x_i} & {y_j}
      \arrow[""{name=0, anchor=center, inner sep=0}, "n", "\shortmid"{marking}, from=1-1, to=1-2]
      \arrow["{\pair{f_i}_{i \in I}}"', from=1-1, to=2-1]
      \arrow["{\pair{g_j}_{j \in J}}", from=1-2, to=2-2]
      \arrow[""{name=1, anchor=center, inner sep=0}, "{\Pi\vec{m}}", "\shortmid"{marking}, from=2-1, to=2-2]
      \arrow["{\pi_i}"', from=2-1, to=3-1]
      \arrow["{\pi_j}", from=2-2, to=3-2]
      \arrow[""{name=2, anchor=center, inner sep=0}, "{m_a}"', "\shortmid"{marking}, from=3-1, to=3-2]
      \arrow["\alpha"{description, pos=0.4}, draw=none, from=0, to=1]
      \arrow["{\pi_a}"{description}, draw=none, from=1, to=2]
    \end{tikzcd}
    \quad=\quad
    \begin{tikzcd}
      w & z \\
      {x_i} & {y_j}
      \arrow[""{name=0, anchor=center, inner sep=0}, "n", "\shortmid"{marking}, from=1-1, to=1-2]
      \arrow[""{name=1, anchor=center, inner sep=0}, "{m_a}"', "\shortmid"{marking}, from=2-1, to=2-2]
      \arrow["{f_i}"', from=1-1, to=2-1]
      \arrow["{g_j}", from=1-2, to=2-2]
      \arrow["{\alpha_a}"{description}, draw=none, from=0, to=1]
    \end{tikzcd},
    \qquad (i \xproto{a} j): (I \xproto{A} J),
  \end{equation*}
  and $\alpha$ is the unique solution by the uniqueness property above. Using
  the description of lax products as universal arrows in
  Proposition \ref{prop:dbl-products-universal}, we conclude that $\dbl{D}$ has lax
  products. This completes the proof characterizing lax products.

  To characterize iso-strong products, we need only prove the reverse
  implication. Suppose that a double category $\dbl{D}$ has strong parallel
  products and restrictions of these along structure arrows. Then, by the
  foregoing, $\dbl{D}$ already has lax products, and the products are normal by
  assumption, so we just need to establish the iso-strong condition on
  composites (Definition \ref{def:iso-strong-products}).

  First, notice that the general case of composing along bijective legs can be
  reduced to the special case of composing along legs that are identities. Given
  indexing spans $I \xfrom{\ell_A} A \xto{r_A} J$ and
  $J \xfrom{\ell_B} A \xto{r_B} K$, where $r_A$ and $\ell_B$ are bijections as in
  Definition \ref{def:iso-strong-products}, the maps of spans
  \begin{equation*}
    \begin{tikzcd}
      I & A & J \\
      I & J & J
      \arrow["{\ell_A}"', from=1-2, to=1-1]
      \arrow["{r_A}", from=1-2, to=1-3]
      \arrow[Rightarrow, no head, from=1-3, to=2-3]
      \arrow[Rightarrow, no head, from=2-2, to=2-3]
      \arrow[Rightarrow, no head, from=1-1, to=2-1]
      \arrow["{r_A}"', from=1-2, to=2-2]
      \arrow["{\ell_a \circ r_A^{-1}}", from=2-2, to=2-1]
    \end{tikzcd}
    \qquad\text{and}\qquad
    \begin{tikzcd}
      J & B & K \\
      J & J & K
      \arrow["{\ell_B}"', from=1-2, to=1-1]
      \arrow["{r_B}", from=1-2, to=1-3]
      \arrow[Rightarrow, no head, from=1-1, to=2-1]
      \arrow[Rightarrow, no head, from=2-2, to=2-1]
      \arrow[Rightarrow, no head, from=1-3, to=2-3]
      \arrow["{r_B \circ \ell_B^{-1}}"', from=2-2, to=2-3]
      \arrow["{\ell_B}", from=1-2, to=2-2]
    \end{tikzcd}
  \end{equation*}
  are isomorphisms. Thus, they induce isomorphisms in $\DblFamOp(\dbl{D})$
  between families indexed by those spans and, by the functorality of products,
  between products of such families. By the naturality axiom for laxators, the
  composition comparisons for products indexed by the original spans are
  invertible precisely when the comparisons for products indexed by the specials
  spans $I \from J = J$ and $J = J \to K$ are invertible.

  So assume that we have families of proarrows
  $(A,\vec{m}): (I,\vec{x}) \proto (A,\vec{y})$ and
  $(A,\vec{n}): (A,\vec{y}) \proto (J,\vec{z})$ indexed by spans
  $I \xfrom{f_0} A = A$ and $A = A \xto{g_0} J$, where $f_0$ and $g_0$ are
  arbitrary functions. We can express the restrictions giving the products of
  these families as
  \begin{equation*}
    \begin{tikzcd}
      {\Pi\vec{x}} & {\Pi f_0^*\vec{x}} & {\Pi\vec{y}} \\
      {\Pi f_0^* \vec{x}} & {\Pi f_0^* \vec{x}} & {\Pi\vec{y}}
      \arrow["{\Pi(f_0)}"', from=1-1, to=2-1]
      \arrow[""{name=0, anchor=center, inner sep=0}, "{\Pi(f_0)_!}", "\shortmid"{marking}, from=1-1, to=1-2]
      \arrow[Rightarrow, no head, from=1-2, to=2-2]
      \arrow[""{name=1, anchor=center, inner sep=0}, "\shortmid"{marking}, Rightarrow, no head, from=2-1, to=2-2]
      \arrow[""{name=2, anchor=center, inner sep=0}, "{\Pi\vec{m}}"', "\shortmid"{marking}, from=2-2, to=2-3]
      \arrow[""{name=3, anchor=center, inner sep=0}, "{\Pi\vec{m}}", "\shortmid"{marking}, from=1-2, to=1-3]
      \arrow[Rightarrow, no head, from=1-3, to=2-3]
      \arrow["\res"{description}, draw=none, from=0, to=1]
      \arrow["1"{description}, draw=none, from=3, to=2]
    \end{tikzcd}
    \qquad\text{and}\qquad
    \begin{tikzcd}
      {\Pi\vec{y}} & {\Pi g_0^* \vec{z}} & {\Pi\vec{z}} \\
      {\Pi\vec{y}} & {\Pi g_0^*\vec{z}} & {\Pi g_0^*\vec{z}}
      \arrow[Rightarrow, no head, from=1-1, to=2-1]
      \arrow[""{name=0, anchor=center, inner sep=0}, "{\Pi\vec{n}}", "\shortmid"{marking}, from=1-1, to=1-2]
      \arrow[""{name=1, anchor=center, inner sep=0}, "{\Pi\vec{n}}"', "\shortmid"{marking}, from=2-1, to=2-2]
      \arrow[Rightarrow, no head, from=1-2, to=2-2]
      \arrow[""{name=2, anchor=center, inner sep=0}, "\shortmid"{marking}, Rightarrow, no head, from=2-2, to=2-3]
      \arrow[""{name=3, anchor=center, inner sep=0}, "{\Pi(g_0)^*}", "\shortmid"{marking}, from=1-2, to=1-3]
      \arrow["{\Pi(g_0)}", from=1-3, to=2-3]
      \arrow["1"{description}, draw=none, from=0, to=1]
      \arrow["\res"{description}, draw=none, from=3, to=2]
    \end{tikzcd},
  \end{equation*}
  using Corollary \ref{cor:structure-proarrows} and the formula for general restrictions
  in terms of companions and conjoints \citep[\mbox{Equation 4.7}]{shulman2008}.
  But, by the same formula, the composite cell
  \begin{equation*}
    \begin{tikzcd}
      {\Pi\vec{x}} & {\Pi f_0^*\vec{x}} & {\Pi\vec{y}} & {\Pi g_0^* \vec{z}} & {\Pi\vec{z}} \\
      {\Pi f_0^* \vec{x}} & {\Pi f_0^* \vec{x}} & {\Pi\vec{y}} & {\Pi g_0^*\vec{z}} & {\Pi g_0^* \vec{z}}
      \arrow["{\Pi(f_0)}"', from=1-1, to=2-1]
      \arrow[""{name=0, anchor=center, inner sep=0}, "{\Pi(f_0)_!}", "\shortmid"{marking}, from=1-1, to=1-2]
      \arrow[Rightarrow, no head, from=1-2, to=2-2]
      \arrow[""{name=1, anchor=center, inner sep=0}, "\shortmid"{marking}, Rightarrow, no head, from=2-1, to=2-2]
      \arrow["{\Pi\vec{m}}"', "\shortmid"{marking}, from=2-2, to=2-3]
      \arrow["{\Pi\vec{m}}", "\shortmid"{marking}, from=1-2, to=1-3]
      \arrow["{\Pi\vec{n}}"', "\shortmid"{marking}, from=2-3, to=2-4]
      \arrow["{\Pi\vec{m}}", "\shortmid"{marking}, from=1-3, to=1-4]
      \arrow["1"{description}, draw=none, from=1-3, to=2-3]
      \arrow[Rightarrow, no head, from=1-4, to=2-4]
      \arrow[""{name=2, anchor=center, inner sep=0}, "\shortmid"{marking}, Rightarrow, no head, from=2-4, to=2-5]
      \arrow["{\Pi(g_0)}", from=1-5, to=2-5]
      \arrow[""{name=3, anchor=center, inner sep=0}, "{\Pi(g_0)^*}", "\shortmid"{marking}, from=1-4, to=1-5]
      \arrow["\res"{description}, draw=none, from=0, to=1]
      \arrow["\res"{description}, draw=none, from=3, to=2]
    \end{tikzcd}
  \end{equation*}
  is a restriction of the composite of parallel products
  $\Pi\vec{m} \odot \Pi\vec{n} \cong \Pi(\vec{m} \odot \vec{n}): \Pi f_0^* \vec{x} \proto \Pi g_0^* \vec{z}$.
  This completes the proof of the iso-strong condition and the theorem.
\end{proof}

\begin{corollary}
  Any precartesian equipment has lax finite products, and any cartesian
  equipment has iso-strong finite products.
\end{corollary}

Since the double categories of spans in a finitely complete category, and of
matrices valued in a distributive category, are known to be cartesian
equipments, we obtain independent proofs of
Propositions \ref{prop:iso-strong-products-span} and \ref{prop:iso-strong-products-mat}. Nevertheless,
the proofs in Sections \ref{sec:products} and \ref{sec:iso-strong-products} are useful in giving
unbiased formulas for products in these double categories. On the other hand, we
have not directly shown that double categories of relations or of profunctors
have iso-strong finite products. That follows from the corollary since these
double categories are known to be cartesian equipments.

We emphasize the converse to the corollary is not true: a double category with
products, even strong ones, need not be an equipment. Here is a simple
counterexample. Let $\dbl{D}$ be the strict double category freely generated by
a nontrival arrow $f: x \to y$. Then $f$ does not have a companion since
$\dbl{D}$ does not even have a proarrow $x \proto y$. Now the double category
$\DblFamOp(\dbl{D})$ has strong products, being the free product completion of
$\dbl{D}$ (Theorem \ref{thm:free-dbl-coproduct-completion}), yet the image of $f$ under
the embedding $\Delta: \dbl{D} \to \DblFamOp(\dbl{D})$ still does not have a
companion. Again, $\DblFamOp(\dbl{D})$ does not even have a proarrow
$\Delta x \proto \Delta y$. So the concepts of a cartesian equipment and a
double category with finite products do not coincide.

\section{Lax functors between double categories with products}
\label{sec:lax-functors}

Though related to cartesian equipments, as we have just seen, double categories
with finite products turn out to be better behaved as domains of lax functors. A
lax functor between equipments preserves restrictions, while dually a colax
functor between equipments preserves extensions \citep[\mbox{Proposition
  6.4}]{shulman2008}. Only a pseudo double functor preserves the full range of
operations available in an equipment. Specifically, the axioms of a lax functor
between equipments seem to only weakly constrain the laxators for general
composites of companions and conjoints. By contrast, although a generic double
category with iso-strong products has certain companions and conjoints, the
corresponding laxators are well controlled. The reason is that these special
restrictions are actually products, as shown in Section \ref{sec:restrictions}, and lax
functors interact well with products, as we will show in this section.

A lax functor $F$ can be defined to preserve products when its opposite, the
colax functor $F^\op$, preserves coproducts
(Definition \ref{def:preserve-dbl-coproducts}). Or, more directly:

\begin{definition}[Preservation of double products]
  \label{def:preserve-dbl-products}
  Let $\dbl{D}$ and $\dbl{E}$ be double categories with lax products. A lax
  functor $F: \dbl{D} \to \dbl{E}$ \define{preserves products} if the morphism
  of adjunctions
  \begin{equation*}
    \begin{tikzcd}
      {\dbl{D}} & {\DblFamOp(\dbl{D})} \\
      {\dbl{E}} & {\DblFamOp(\dbl{E})}
      \arrow[""{name=0, anchor=center, inner sep=0}, "{\Delta \dashv \Pi}", "\shortmid"{marking}, from=1-1, to=1-2]
      \arrow["F"', from=1-1, to=2-1]
      \arrow["{\DblFamOp(F)}", from=1-2, to=2-2]
      \arrow[""{name=1, anchor=center, inner sep=0}, "{\Delta \dashv \Pi}"', "\shortmid"{marking}, from=2-1, to=2-2]
      \arrow["{(1,\Phi)}"{description}, draw=none, from=0, to=1]
    \end{tikzcd}
  \end{equation*}
  is strong.
\end{definition}

To elaborate, the components of the natural transformation
\begin{equation*}
  \begin{tikzcd}
    {\DblFamOp(\dbl{D})} & {\dbl{D}} \\
    {\DblFamOp(\dbl{E})} & {\dbl{E}}
    \arrow["F", from=1-2, to=2-2]
    \arrow["{\Pi_{\dbl{D}}}", from=1-1, to=1-2]
    \arrow["{\DblFamOp(F)}"', from=1-1, to=2-1]
    \arrow["{\Pi_{\dbl{E}}}"', from=2-1, to=2-2]
    \arrow["\Phi"', shorten <=7pt, shorten >=7pt, Rightarrow, from=1-2, to=2-1]
  \end{tikzcd}
\end{equation*}
are, at the family of objects $(I,\vec{x})$ in $\dbl{D}$, the usual canonical
comparison $\Phi_{\vec{x}}: F \Pi \vec{x} \to \Pi F \vec{x}$ between products,
and at the family of proarrows $(A,\vec{m}): (I,\vec{x}) \proto (J,\vec{y})$
in $\dbl{D}$, the comparison cell
$\Phi_{\vec{m}}: F \Pi \vec{m} \to \Pi F \vec{m}$ that is the unique solution to
the equations
\begin{equation*}
  \begin{tikzcd}
    {F \Pi \vec{x}} & {F \Pi \vec{y}} \\
    {\Pi F \vec{x}} & {\Pi F \vec{y}} \\
    {Fx_i} & {Fy_j}
    \arrow[""{name=0, anchor=center, inner sep=0}, "{F\Pi \vec{m}}", "\shortmid"{marking}, from=1-1, to=1-2]
    \arrow["{\Phi_{\vec{x}}}"', from=1-1, to=2-1]
    \arrow["{\Phi_{\vec{y}}}", from=1-2, to=2-2]
    \arrow[""{name=1, anchor=center, inner sep=0}, "{\Pi F \vec{m}}", "\shortmid"{marking}, from=2-1, to=2-2]
    \arrow["{\pi_i}"', from=2-1, to=3-1]
    \arrow["{\pi_j}", from=2-2, to=3-2]
    \arrow[""{name=2, anchor=center, inner sep=0}, "{F m_a}"', "\shortmid"{marking}, from=3-1, to=3-2]
    \arrow["{\pi_a}"{description}, draw=none, from=1, to=2]
    \arrow["{\Phi_{\vec{m}}}"{description, pos=0.4}, draw=none, from=0, to=1]
  \end{tikzcd}
  \quad=\quad
  \begin{tikzcd}
    {F \Pi \vec{x}} & {F \Pi \vec{y}} \\
    {Fx_i} & {Fy_j}
    \arrow[""{name=0, anchor=center, inner sep=0}, "{F\Pi \vec{m}}", "\shortmid"{marking}, from=1-1, to=1-2]
    \arrow[""{name=1, anchor=center, inner sep=0}, "{F m_a}"', "\shortmid"{marking}, from=2-1, to=2-2]
    \arrow["{F\pi_i}"', from=1-1, to=2-1]
    \arrow["{F\pi_j}", from=1-2, to=2-2]
    \arrow["{F\pi_a}"{description}, draw=none, from=0, to=1]
  \end{tikzcd},
  \qquad (i \xproto{a} j): (I \xproto{A} J).
\end{equation*}
A lax functor $F: \dbl{D} \to \dbl{E}$ preserves products just when all
components $\Phi_{\vec{x}}$ and $\Phi_{\vec{m}}$ are isomorphisms in $\dbl{E}_0$ and
$\dbl{E}_1$, respectively.

\begin{example}[Preserving products of spans]
  Recalling Example \ref{ex:preserve-coproducts-span}, any pullback-preserving functor
  $F: \cat{C} \to \cat{D}$ between categories with pullbacks induces a (pseudo)
  double functor
  \begin{equation*}
    \Span{F}: \Span{\cat{C}} \to \Span{\cat{D}}
  \end{equation*}
  between double categories of spans. Suppose that $\cat{C}$ and $\cat{D}$ also
  have products, i.e., are complete categories, so that $\Span{\cat{C}}$ and
  $\Span{\cat{D}}$ have iso-strong products. Then $\Span{F}$ preserves products
  if and only if $F$ preserves products, i.e., is a continuous functor.
\end{example}

Let $F: \dbl{D} \to \dbl{E}$ be a lax functor between double categories. To state
the next lemma, we need an explicit description of the comparisons for the lax
functor
\begin{equation*}
  \DblFamOp(F): \DblFamOp(\dbl{D}) \to \DblFamOp(\dbl{E}),
\end{equation*}
given by dual of Construction \ref{def:dbl-fam-functor}, that applies $F$
elementwise to families. The laxator of $\DblFamOp(F)$ at a pair of proarrow
families
$(I,\vec{x}) \xproto{(A,\vec{m})} (J,\vec{y}) \xproto{(B,\vec{n})} (K,\vec{z})$
in $\dbl{D}$, denoted $F_{\vec{m},\vec{n}}$ on the left
\begin{equation*}
  \begin{tikzcd}[column sep=scriptsize]
    {F\vec{x}} & {F\vec{y}} & {F\vec{z}} \\
    {F\vec{x}} && {F\vec{z}}
    \arrow[""{name=0, anchor=center, inner sep=0}, "{F(\vec{m} \odot \vec{n})}"', "\shortmid"{marking}, from=2-1, to=2-3]
    \arrow[Rightarrow, no head, from=1-1, to=2-1]
    \arrow[Rightarrow, no head, from=1-3, to=2-3]
    \arrow["{F\vec{m}}", "\shortmid"{marking}, from=1-1, to=1-2]
    \arrow["{F\vec{n}}", "\shortmid"{marking}, from=1-2, to=1-3]
    \arrow["{F_{\vec{m},\vec{n}}}"{description, pos=0.4}, draw=none, from=1-2, to=0]
  \end{tikzcd}
  \quad\leftrightsquigarrow\quad
  \begin{tikzcd}[column sep=scriptsize]
    {F x_i} & {F y_j} & {F z_k} \\
    {F x_i} && {F z_k}
    \arrow[""{name=0, anchor=center, inner sep=0}, "{F(m_a \odot n_b)}"', "\shortmid"{marking}, from=2-1, to=2-3]
    \arrow[Rightarrow, no head, from=1-1, to=2-1]
    \arrow[Rightarrow, no head, from=1-3, to=2-3]
    \arrow["{F m_a}", "\shortmid"{marking}, from=1-1, to=1-2]
    \arrow["{F n_b}", "\shortmid"{marking}, from=1-2, to=1-3]
    \arrow["{F_{m_a,n_b}}"{description, pos=0.4}, draw=none, from=1-2, to=0]
  \end{tikzcd},
  \quad (i \xproto{a} j \xproto{b} k) : (I \xproto{A} J \xproto{B} K),
\end{equation*}
is the family of laxators of $F$ on the right, indexed by pairs
$(a,b) \in A \times_J B$. Similarly, the unitor of $\DblFamOp(F)$ at an object
family $(I,\vec{x})$ in $\dbl{D}$, denoted $F_{\vec{x}}$ on the left
\begin{equation*}
  \begin{tikzcd}
    {F\vec{x}} & {F\vec{x}} \\
    {F\vec{x}} & {F\vec{x}}
    \arrow[Rightarrow, no head, from=1-1, to=2-1]
    \arrow[""{name=0, anchor=center, inner sep=0}, "{\id_{F\vec{x}}}", "\shortmid"{marking}, from=1-1, to=1-2]
    \arrow[Rightarrow, no head, from=1-2, to=2-2]
    \arrow[""{name=1, anchor=center, inner sep=0}, "{F \id_{\vec{x}}}"', "\shortmid"{marking}, from=2-1, to=2-2]
    \arrow["{F_{\vec{x}}}"{description}, draw=none, from=0, to=1]
  \end{tikzcd}
  \qquad\leftrightsquigarrow\qquad
  \begin{tikzcd}
    {Fx_i} & {Fx_i} \\
    {Fx_i} & {Fx_i}
    \arrow[""{name=0, anchor=center, inner sep=0}, "{\id_{Fx_i}}", "\shortmid"{marking}, from=1-1, to=1-2]
    \arrow[""{name=1, anchor=center, inner sep=0}, "{F \id_{x_i}}"', "\shortmid"{marking}, from=2-1, to=2-2]
    \arrow[Rightarrow, no head, from=1-1, to=2-1]
    \arrow[Rightarrow, no head, from=1-2, to=2-2]
    \arrow["{F_{x_i}}"{description}, draw=none, from=0, to=1]
  \end{tikzcd},
  \quad i \in I,
\end{equation*}
is the family of unitors of $F$ on the right. With this notation, we state an
important technical lemma relating laxators and unitors of products to products
of laxators and unitors.

\begin{lemma}[Laxators and unitors for products]
  \label{lem:laxator-unitor-product}
  Let $F: \dbl{D} \to \dbl{E}$ be a lax functor between double categories with
  lax products. For any composable families of proarrows
  $(I,\vec{x}) \xproto{(A,\vec{m})} (J,\vec{y}) \xproto{(B,\vec{n})} (K,\vec{z})$
  in $\dbl{D}$, we have
  \begin{equation} \label{eq:laxator-product}
    \begin{tikzcd}
      {F\Pi\vec{x}} & {F\Pi\vec{y}} & {F\Pi\vec{z}} \\
      {F\Pi\vec{x}} && {F\Pi\vec{z}} \\
      {F\Pi\vec{x}} && {F\Pi\vec{z}} \\
      {\Pi F\vec{x}} && {\Pi F\vec{z}}
      \arrow["{F\Pi\vec{m}}"{inner sep=.8ex}, "\shortmid"{marking}, from=1-1, to=1-2]
      \arrow[equals, from=1-1, to=2-1]
      \arrow["{F\Pi\vec{n}}"{inner sep=.8ex}, "\shortmid"{marking}, from=1-2, to=1-3]
      \arrow[equals, from=1-3, to=2-3]
      \arrow[""{name=0, anchor=center, inner sep=0}, "{F(\Pi\vec{m} \odot \Pi\vec{n})}"'{inner sep=.8ex}, "\shortmid"{marking}, from=2-1, to=2-3]
      \arrow[equals, from=2-1, to=3-1]
      \arrow[equals, from=2-3, to=3-3]
      \arrow[""{name=1, anchor=center, inner sep=0}, "{F\Pi(\vec{m} \odot \vec{n})}"'{inner sep=.8ex}, "\shortmid"{marking}, from=3-1, to=3-3]
      \arrow["{\Phi_{\vec{x}}}"', from=3-1, to=4-1]
      \arrow["{\Phi_{\vec{z}}}", from=3-3, to=4-3]
      \arrow[""{name=2, anchor=center, inner sep=0}, "{\Pi F(\vec{m} \odot \vec{n})}"'{inner sep=.8ex}, "\shortmid"{marking}, from=4-1, to=4-3]
      \arrow["{F_{\Pi\vec{m},\Pi\vec{n}}}"{description, pos=0.4}, draw=none, from=1-2, to=0]
      \arrow["{F\Pi_{\vec{m},\vec{n}}}"{description, pos=0.6}, draw=none, from=0, to=1]
      \arrow["{\Phi_{\vec{m} \odot \vec{n}}}"{description, pos=0.6}, draw=none, from=1, to=2]
    \end{tikzcd}
    \quad=\quad
    \begin{tikzcd}
      {F\Pi\vec{x}} & {F\Pi\vec{y}} & {F\Pi\vec{z}} \\
      {\Pi F\vec{x}} & {\Pi F\vec{y}} & {\Pi F\vec{z}} \\
      {\Pi F\vec{x}} && {\Pi F\vec{z}} \\
      {\Pi F\vec{x}} && {\Pi F\vec{z}}
      \arrow[""{name=0, anchor=center, inner sep=0}, "{F\Pi\vec{m}}", "\shortmid"{marking}, from=1-1, to=1-2]
      \arrow[""{name=1, anchor=center, inner sep=0}, "{F\Pi\vec{n}}", "\shortmid"{marking}, from=1-2, to=1-3]
      \arrow["{\Phi_{\vec{x}}}"', from=1-1, to=2-1]
      \arrow["{\Phi_{\vec{y}}}"{description}, from=1-2, to=2-2]
      \arrow["{\Phi_{\vec{z}}}", from=1-3, to=2-3]
      \arrow[""{name=2, anchor=center, inner sep=0}, "{\Pi F\vec{m}}"', "\shortmid"{marking}, from=2-1, to=2-2]
      \arrow[""{name=3, anchor=center, inner sep=0}, "{\Pi F\vec{n}}"', "\shortmid"{marking}, from=2-2, to=2-3]
      \arrow[Rightarrow, no head, from=2-1, to=3-1]
      \arrow[Rightarrow, no head, from=2-3, to=3-3]
      \arrow[""{name=4, anchor=center, inner sep=0}, "{\Pi(F\vec{m} \odot F\vec{n})}"', "\shortmid"{marking}, from=3-1, to=3-3]
      \arrow[Rightarrow, no head, from=3-1, to=4-1]
      \arrow[Rightarrow, no head, from=3-3, to=4-3]
      \arrow[""{name=5, anchor=center, inner sep=0}, "{\Pi F(\vec{m} \odot \vec{n})}"', "\shortmid"{marking}, from=4-1, to=4-3]
      \arrow["{\Phi_{\vec{m}}}"{description}, draw=none, from=0, to=2]
      \arrow["{\Phi_{\vec{n}}}"{description}, draw=none, from=1, to=3]
      \arrow["{\Pi_{F\vec{m}, F\vec{n}}}"{description}, draw=none, from=2-2, to=4]
      \arrow["{\Pi F_{\vec{m},\vec{n}}}"{description, pos=0.6}, draw=none, from=4, to=5]
    \end{tikzcd}.
  \end{equation}
  Similarly, for any family of objects $(I,\vec{x})$ in $\dbl{D}$, we have
  \begin{equation} \label{eq:unitor-product}
    \begin{tikzcd}
      {F\Pi\vec{x}} & {F\Pi\vec{x}} \\
      {F\Pi\vec{x}} & {F\Pi\vec{x}} \\
      {F\Pi\vec{x}} & {F\Pi\vec{x}} \\
      {\Pi F\vec{x}} & {\Pi F\vec{x}}
      \arrow[""{name=0, anchor=center, inner sep=0}, "{\id_{F\Pi\vec{x}}}", "\shortmid"{marking}, from=1-1, to=1-2]
      \arrow[Rightarrow, no head, from=1-1, to=2-1]
      \arrow[Rightarrow, no head, from=1-2, to=2-2]
      \arrow[""{name=1, anchor=center, inner sep=0}, "{F\id_{\Pi\vec{x}}}"', "\shortmid"{marking}, from=2-1, to=2-2]
      \arrow[Rightarrow, no head, from=2-1, to=3-1]
      \arrow[Rightarrow, no head, from=2-2, to=3-2]
      \arrow[""{name=2, anchor=center, inner sep=0}, "{F\Pi\id_{\vec{x}}}"', "\shortmid"{marking}, from=3-1, to=3-2]
      \arrow["{\Phi_{\vec{x}}}"', from=3-1, to=4-1]
      \arrow["{\Phi_{\vec{x}}}", from=3-2, to=4-2]
      \arrow[""{name=3, anchor=center, inner sep=0}, "{\Pi F\id_{\vec{x}}}"', "\shortmid"{marking}, from=4-1, to=4-2]
      \arrow["{F_{\Pi\vec{x}}}"{description}, draw=none, from=0, to=1]
      \arrow["{F\Pi_{\vec{x}}}"{description, pos=0.6}, draw=none, from=1, to=2]
      \arrow["{\Phi_{\id_{\vec{x}}}}"{description, pos=0.6}, draw=none, from=2, to=3]
    \end{tikzcd}
    \quad=\quad
    \begin{tikzcd}
      {F\Pi\vec{x}} & {F\Pi\vec{x}} \\
      {\Pi F\vec{x}} & {\Pi F\vec{x}} \\
      {\Pi F\vec{x}} & {\Pi F\vec{x}} \\
      {\Pi F\vec{x}} & {\Pi F\vec{x}}
      \arrow[""{name=0, anchor=center, inner sep=0}, "{\id_{F\Pi\vec{x}}}", "\shortmid"{marking}, from=1-1, to=1-2]
      \arrow[""{name=1, anchor=center, inner sep=0}, "{\Pi F\id_{\vec{x}}}"', "\shortmid"{marking}, from=4-1, to=4-2]
      \arrow[Rightarrow, no head, from=2-1, to=3-1]
      \arrow[Rightarrow, no head, from=2-2, to=3-2]
      \arrow["{\Phi_{\vec{x}}}"', from=1-1, to=2-1]
      \arrow["{\Phi_{\vec{x}}}", from=1-2, to=2-2]
      \arrow[""{name=2, anchor=center, inner sep=0}, "{\id_{\Pi F\vec{x}}}"', "\shortmid"{marking}, from=2-1, to=2-2]
      \arrow[""{name=3, anchor=center, inner sep=0}, "{\Pi \id_{F\vec{x}}}"', "\shortmid"{marking}, from=3-1, to=3-2]
      \arrow[Rightarrow, no head, from=3-1, to=4-1]
      \arrow[Rightarrow, no head, from=3-2, to=4-2]
      \arrow["{\id_{\Phi_{\vec{x}}}}"{description}, draw=none, from=0, to=2]
      \arrow["{\Pi_{F\vec{x}}}"{description, pos=0.6}, draw=none, from=2, to=3]
      \arrow["{\Pi F_{\vec{x}}}"{description, pos=0.6}, draw=none, from=3, to=1]
    \end{tikzcd}.
  \end{equation}
  In particular, if $F$ preserves products and the domain $\dbl{D}$ has strong
  products, then the laxator $F_{\Pi\vec{m},\Pi\vec{n}}$ for a composite of
  products is uniquely determined by the product of laxators
  $\Pi F_{\vec{m},\vec{n}}$ and the unitor $F_{\Pi\vec{x}}$ for a product is
  uniquely determined by the product of unitors $\Pi F_{\vec{x}}$.
\end{lemma}
\begin{proof}
  The proof is a direct calculation, even if the notation is heavy. Beginning
  with the laxators, choose elements $i \xproto{a} j \xproto{b} k$ of the
  indexing spans. Then post-composing the left-hand side of
  Equation \eqref{eq:laxator-product} with the projection $\pi_{(a,b)}$ in $\dbl{E}$ gives
  \begin{equation*}
    \begin{tikzcd}[column sep=scriptsize]
      {F\Pi\vec{x}} & {F\Pi\vec{y}} & {F\Pi\vec{z}} \\
      {\Pi F\vec{x}} && {\Pi F\vec{z}} \\
      {Fx_i} && {Fz_k}
      \arrow["{F\Pi\vec{m}}", "\shortmid"{marking}, from=1-1, to=1-2]
      \arrow["{F\Pi\vec{n}}", "\shortmid"{marking}, from=1-2, to=1-3]
      \arrow[""{name=0, anchor=center, inner sep=0}, "{\Pi F(\vec{m} \odot \vec{n})}"', "\shortmid"{marking}, from=2-1, to=2-3]
      \arrow["{\Phi_{\vec{x}}}"', from=1-1, to=2-1]
      \arrow["{\Phi_{\vec{x}}}", from=1-3, to=2-3]
      \arrow["{\pi_i}"', from=2-1, to=3-1]
      \arrow["{\pi_k}", from=2-3, to=3-3]
      \arrow[""{name=1, anchor=center, inner sep=0}, "{F(m_a \odot n_b)}"', "\shortmid"{marking}, from=3-1, to=3-3]
      \arrow["{\text{LHS}}"{description, pos=0.4}, draw=none, from=1-2, to=0]
      \arrow["{\pi_{(a,b)}}"{description, pos=0.6}, draw=none, from=0, to=1]
    \end{tikzcd}
    =
    \begin{tikzcd}[column sep=scriptsize]
      {F\Pi\vec{x}} & {F\Pi\vec{y}} & {F\Pi\vec{z}} \\
      {F\Pi\vec{x}} && {F\Pi\vec{z}} \\
      {F\Pi\vec{x}} && {F\Pi\vec{z}} \\
      {F x_i} && {F z_k}
      \arrow["{F\Pi\vec{m}}", "\shortmid"{marking}, from=1-1, to=1-2]
      \arrow["{F\Pi\vec{n}}", "\shortmid"{marking}, from=1-2, to=1-3]
      \arrow[Rightarrow, no head, from=1-1, to=2-1]
      \arrow[Rightarrow, no head, from=1-3, to=2-3]
      \arrow[""{name=0, anchor=center, inner sep=0}, "{F(\Pi\vec{m} \odot \Pi\vec{n})}"', "\shortmid"{marking}, from=2-1, to=2-3]
      \arrow[Rightarrow, no head, from=2-1, to=3-1]
      \arrow[Rightarrow, no head, from=2-3, to=3-3]
      \arrow[""{name=1, anchor=center, inner sep=0}, "{F\Pi(\vec{m} \odot \vec{n})}"', "\shortmid"{marking}, from=3-1, to=3-3]
      \arrow["{F\pi_i}"', from=3-1, to=4-1]
      \arrow["{F\pi_k}", from=3-3, to=4-3]
      \arrow[""{name=2, anchor=center, inner sep=0}, "{F(m_a \odot n_b)}"', "\shortmid"{marking}, from=4-1, to=4-3]
      \arrow["{F_{\Pi\vec{m},\Pi\vec{n}}}"{description, pos=0.4}, draw=none, from=1-2, to=0]
      \arrow["{F\Pi_{\vec{m},\vec{n}}}"{description, pos=0.6}, draw=none, from=0, to=1]
      \arrow["{F\pi_{(a,b)}}"{description, pos=0.6}, draw=none, from=1, to=2]
    \end{tikzcd}
    =
    \begin{tikzcd}[column sep=scriptsize]
      {F\Pi\vec{x}} & {F\Pi\vec{y}} & {F\Pi\vec{z}} \\
      {F\Pi\vec{x}} && {F\Pi\vec{z}} \\
      {F x_i} && {F z_k}
      \arrow["{F\Pi\vec{m}}", "\shortmid"{marking}, from=1-1, to=1-2]
      \arrow["{F\Pi\vec{n}}", "\shortmid"{marking}, from=1-2, to=1-3]
      \arrow[Rightarrow, no head, from=1-1, to=2-1]
      \arrow[Rightarrow, no head, from=1-3, to=2-3]
      \arrow[""{name=0, anchor=center, inner sep=0}, "{F(\Pi\vec{m} \odot \Pi\vec{n})}"', "\shortmid"{marking}, from=2-1, to=2-3]
      \arrow[""{name=1, anchor=center, inner sep=0}, "{F(m_a \odot n_b)}"', "\shortmid"{marking}, from=3-1, to=3-3]
      \arrow["{F\pi_i}"', from=2-1, to=3-1]
      \arrow["{F\pi_k}", from=2-3, to=3-3]
      \arrow["{F_{\Pi\vec{m},\Pi\vec{n}}}"{description, pos=0.4}, draw=none, from=1-2, to=0]
      \arrow["{F(\pi_a \odot \pi_b)}"{description, pos=0.6}, draw=none, from=0, to=1]
    \end{tikzcd},
  \end{equation*}
  where the last equality uses Equation
  \eqref{eq:product-composition-comparison}. Meanwhile, post-composing the
  right-hand side of Equation \eqref{eq:laxator-product} with this projection
  gives \begingroup \allowdisplaybreaks
  \begin{align*}
    \begin{tikzcd}[ampersand replacement=\&,column sep=scriptsize]
      {F\Pi\vec{x}} \& {F\Pi\vec{y}} \& {F\Pi\vec{z}} \\
      {\Pi F\vec{x}} \&\& {\Pi F\vec{z}} \\
      {Fx_i} \&\& {Fz_k}
      \arrow["{F\Pi\vec{m}}", "\shortmid"{marking}, from=1-1, to=1-2]
      \arrow["{F\Pi\vec{n}}", "\shortmid"{marking}, from=1-2, to=1-3]
      \arrow[""{name=0, anchor=center, inner sep=0}, "{\Pi F(\vec{m} \odot \vec{n})}"', "\shortmid"{marking}, from=2-1, to=2-3]
      \arrow["{\Phi_{\vec{x}}}"', from=1-1, to=2-1]
      \arrow["{\Phi_{\vec{x}}}", from=1-3, to=2-3]
      \arrow["{\pi_i}"', from=2-1, to=3-1]
      \arrow["{\pi_k}", from=2-3, to=3-3]
      \arrow[""{name=1, anchor=center, inner sep=0}, "{F(m_a \odot n_b)}"', "\shortmid"{marking}, from=3-1, to=3-3]
      \arrow["{\text{RHS}}"{description, pos=0.4}, draw=none, from=1-2, to=0]
      \arrow["{\pi_{(a,b)}}"{description, pos=0.6}, draw=none, from=0, to=1]
    \end{tikzcd}
    &=
    \begin{tikzcd}[ampersand replacement=\&,column sep=scriptsize]
      {F\Pi\vec{x}} \& {F\Pi\vec{y}} \& {F\Pi\vec{z}} \\
      {\Pi F\vec{x}} \& {\Pi F\vec{y}} \& {\Pi F\vec{z}} \\
      {\Pi F\vec{x}} \&\& {\Pi F\vec{z}} \\
      {Fx_i} \& {Fy_j} \& {Fz_k} \\
      {Fx_i} \&\& {Fz_k}
      \arrow[""{name=0, anchor=center, inner sep=0}, "{F\Pi\vec{m}}", "\shortmid"{marking}, from=1-1, to=1-2]
      \arrow[""{name=1, anchor=center, inner sep=0}, "{F\Pi\vec{n}}", "\shortmid"{marking}, from=1-2, to=1-3]
      \arrow["{\Phi_{\vec{x}}}"', from=1-1, to=2-1]
      \arrow["{\Phi_{\vec{y}}}"{description}, from=1-2, to=2-2]
      \arrow["{\Phi_{\vec{z}}}", from=1-3, to=2-3]
      \arrow[""{name=2, anchor=center, inner sep=0}, "{\Pi F\vec{m}}"', "\shortmid"{marking}, from=2-1, to=2-2]
      \arrow[""{name=3, anchor=center, inner sep=0}, "{\Pi F\vec{n}}"', "\shortmid"{marking}, from=2-2, to=2-3]
      \arrow[Rightarrow, no head, from=2-1, to=3-1]
      \arrow[Rightarrow, no head, from=2-3, to=3-3]
      \arrow[""{name=4, anchor=center, inner sep=0}, "{\Pi(F\vec{m} \odot F\vec{n})}"', "\shortmid"{marking}, from=3-1, to=3-3]
      \arrow["{F m_a}"', "\shortmid"{marking}, from=4-1, to=4-2]
      \arrow["{F n_b}"', "\shortmid"{marking}, from=4-2, to=4-3]
      \arrow[""{name=5, anchor=center, inner sep=0}, "{F(m_a \odot n_b)}"', "\shortmid"{marking}, from=5-1, to=5-3]
      \arrow[Rightarrow, no head, from=4-1, to=5-1]
      \arrow[Rightarrow, no head, from=4-3, to=5-3]
      \arrow["{\pi_i}"', from=3-1, to=4-1]
      \arrow["{\pi_k}", from=3-3, to=4-3]
      \arrow["{\Phi_{\vec{m}}}"{description}, draw=none, from=0, to=2]
      \arrow["{\Phi_{\vec{n}}}"{description}, draw=none, from=1, to=3]
      \arrow["{\Pi_{F\vec{m}, F\vec{n}}}"{description}, draw=none, from=2-2, to=4]
      \arrow["{F_{m_a,n_b}}"{description}, draw=none, from=4-2, to=5]
      \arrow["{\pi_{(a,b)}}"{description, pos=0.7}, draw=none, from=4, to=4-2]
    \end{tikzcd}
    =
    \begin{tikzcd}[ampersand replacement=\&,column sep=scriptsize]
      {F\Pi\vec{x}} \& {F\Pi\vec{y}} \& {F\Pi\vec{z}} \\
      {\Pi F\vec{x}} \& {\Pi F\vec{y}} \& {\Pi F\vec{z}} \\
      {Fx_i} \& {Fy_j} \& {Fz_k} \\
      {Fx_i} \&\& {Fz_k}
      \arrow[""{name=0, anchor=center, inner sep=0}, "{F\Pi\vec{m}}", "\shortmid"{marking}, from=1-1, to=1-2]
      \arrow[""{name=1, anchor=center, inner sep=0}, "{F\Pi\vec{n}}", "\shortmid"{marking}, from=1-2, to=1-3]
      \arrow["{\Phi_{\vec{x}}}"', from=1-1, to=2-1]
      \arrow["{\Phi_{\vec{y}}}"{description}, from=1-2, to=2-2]
      \arrow["{\Phi_{\vec{z}}}", from=1-3, to=2-3]
      \arrow[""{name=2, anchor=center, inner sep=0}, "{\Pi F\vec{m}}"', "\shortmid"{marking}, from=2-1, to=2-2]
      \arrow[""{name=3, anchor=center, inner sep=0}, "{\Pi F\vec{n}}"', "\shortmid"{marking}, from=2-2, to=2-3]
      \arrow[""{name=4, anchor=center, inner sep=0}, "{F m_a}"', "\shortmid"{marking}, from=3-1, to=3-2]
      \arrow[""{name=5, anchor=center, inner sep=0}, "{F n_b}"', "\shortmid"{marking}, from=3-2, to=3-3]
      \arrow[""{name=6, anchor=center, inner sep=0}, "{F(m_a \odot n_b)}"', "\shortmid"{marking}, from=4-1, to=4-3]
      \arrow[Rightarrow, no head, from=3-1, to=4-1]
      \arrow[Rightarrow, no head, from=3-3, to=4-3]
      \arrow["{\pi_i}"', from=2-1, to=3-1]
      \arrow["{\pi_k}", from=2-3, to=3-3]
      \arrow["{\pi_j}"{description}, from=2-2, to=3-2]
      \arrow["{\Phi_{\vec{m}}}"{description}, draw=none, from=0, to=2]
      \arrow["{F_{m_a,n_b}}"{description}, draw=none, from=3-2, to=6]
      \arrow["{\pi_a}"{description}, draw=none, from=2, to=4]
      \arrow["{\Phi_{\vec{n}}}"{description}, draw=none, from=1, to=3]
      \arrow["{\pi_b}"{description}, draw=none, from=3, to=5]
    \end{tikzcd} \\
    &=
    \begin{tikzcd}[ampersand replacement=\&,column sep=scriptsize]
      {F\Pi\vec{x}} \& {F\Pi\vec{y}} \& {F\Pi\vec{z}} \\
      {Fx_i} \& {Fy_j} \& {Fz_k} \\
      {Fx_i} \&\& {Fz_k}
      \arrow[""{name=0, anchor=center, inner sep=0}, "{F\Pi\vec{m}}", "\shortmid"{marking}, from=1-1, to=1-2]
      \arrow[""{name=1, anchor=center, inner sep=0}, "{F\Pi\vec{n}}", "\shortmid"{marking}, from=1-2, to=1-3]
      \arrow[""{name=2, anchor=center, inner sep=0}, "{F m_a}"', "\shortmid"{marking}, from=2-1, to=2-2]
      \arrow[""{name=3, anchor=center, inner sep=0}, "{F n_b}"', "\shortmid"{marking}, from=2-2, to=2-3]
      \arrow[""{name=4, anchor=center, inner sep=0}, "{F(m_a \odot n_b)}"', "\shortmid"{marking}, from=3-1, to=3-3]
      \arrow[Rightarrow, no head, from=2-1, to=3-1]
      \arrow[Rightarrow, no head, from=2-3, to=3-3]
      \arrow["{F\pi_i}"', from=1-1, to=2-1]
      \arrow["{F\pi_j}"{description}, from=1-2, to=2-2]
      \arrow["{F\pi_k}", from=1-3, to=2-3]
      \arrow["{F_{m_a,n_b}}"{description}, draw=none, from=2-2, to=4]
      \arrow["{F\pi_a}"{description}, draw=none, from=0, to=2]
      \arrow["{F\pi_b}"{description}, draw=none, from=1, to=3]
    \end{tikzcd}
    =
    \begin{tikzcd}[ampersand replacement=\&,column sep=scriptsize]
      {F\Pi\vec{x}} \& {F\Pi\vec{y}} \& {F\Pi\vec{z}} \\
      {F\Pi\vec{x}} \&\& {F\Pi\vec{z}} \\
      {F x_i} \&\& {F z_k}
      \arrow["{F\Pi\vec{m}}", "\shortmid"{marking}, from=1-1, to=1-2]
      \arrow["{F\Pi\vec{n}}", "\shortmid"{marking}, from=1-2, to=1-3]
      \arrow[Rightarrow, no head, from=1-1, to=2-1]
      \arrow[Rightarrow, no head, from=1-3, to=2-3]
      \arrow[""{name=0, anchor=center, inner sep=0}, "{F(\Pi\vec{m} \odot \Pi\vec{n})}"', "\shortmid"{marking}, from=2-1, to=2-3]
      \arrow[""{name=1, anchor=center, inner sep=0}, "{F(m_a \odot n_b)}"', "\shortmid"{marking}, from=3-1, to=3-3]
      \arrow["{F\pi_i}"', from=2-1, to=3-1]
      \arrow["{F\pi_k}", from=2-3, to=3-3]
      \arrow["{F_{\Pi\vec{m},\Pi\vec{n}}}"{description, pos=0.4}, draw=none, from=1-2, to=0]
      \arrow["{F(\pi_a \odot \pi_b)}"{description, pos=0.6}, draw=none, from=0, to=1]
    \end{tikzcd},
  \end{align*}
  \endgroup
  where the last equation is the naturality of the laxators of $F$ with respect
  to the projection cells $\pi_a$ and $\pi_b$ in $\dbl{D}$. Thus, post-composite
  of Equation \eqref{eq:laxator-product} with the projection $\pi_{(a,b)}$ in $\dbl{E}$ is
  true for any choice of elements $i \xproto{a} j \xproto{b} k$. By the
  universal property of the product $\Pi F(\vec{m} \odot \vec{n})$,
  Equation \eqref{eq:laxator-product} itself is true.

  Turning to the unitors, fix an element $i \in I$. The post-composite of the
  left-hand side of Equation \eqref{eq:unitor-product} with the projection $\pi_i$ in
  $\dbl{E}$ is
  \begin{equation*}
    \begin{tikzcd}
      {F\Pi\vec{x}} & {F\Pi\vec{x}} \\
      {\Pi F\vec{x}} & {\Pi F\vec{x}} \\
      {Fx_i} & {Fx_i}
      \arrow[""{name=0, anchor=center, inner sep=0}, "{\id_{F\Pi\vec{x}}}", "\shortmid"{marking}, from=1-1, to=1-2]
      \arrow[""{name=1, anchor=center, inner sep=0}, "{\Pi F\id_{\vec{x}}}"', "\shortmid"{marking}, from=2-1, to=2-2]
      \arrow["{\Phi_{\vec{x}}}"', from=1-1, to=2-1]
      \arrow["{\Phi_{\vec{x}}}", from=1-2, to=2-2]
      \arrow[""{name=2, anchor=center, inner sep=0}, "{F\id_{x_i}}"', "\shortmid"{marking}, from=3-1, to=3-2]
      \arrow["{\pi_i}"', from=2-1, to=3-1]
      \arrow["{\pi_i}", from=2-2, to=3-2]
      \arrow["{\text{LHS}}"{description}, draw=none, from=0, to=1]
      \arrow["{\pi_i}"{description, pos=0.6}, draw=none, from=1, to=2]
    \end{tikzcd}
    =
    \begin{tikzcd}
      {F\Pi\vec{x}} & {F\Pi\vec{x}} \\
      {F\Pi\vec{x}} & {F\Pi\vec{x}} \\
      {F\Pi\vec{x}} & {F\Pi\vec{x}} \\
      {Fx_i} & {Fx_i}
      \arrow[""{name=0, anchor=center, inner sep=0}, "{\id_{F\Pi\vec{x}}}", "\shortmid"{marking}, from=1-1, to=1-2]
      \arrow[Rightarrow, no head, from=1-1, to=2-1]
      \arrow[Rightarrow, no head, from=1-2, to=2-2]
      \arrow[""{name=1, anchor=center, inner sep=0}, "{F\id_{\Pi\vec{x}}}"', "\shortmid"{marking}, from=2-1, to=2-2]
      \arrow[Rightarrow, no head, from=2-1, to=3-1]
      \arrow[Rightarrow, no head, from=2-2, to=3-2]
      \arrow[""{name=2, anchor=center, inner sep=0}, "{F\Pi\id_{\vec{x}}}"', "\shortmid"{marking}, from=3-1, to=3-2]
      \arrow["{F\pi_i}"', from=3-1, to=4-1]
      \arrow["{F\pi_i}", from=3-2, to=4-2]
      \arrow[""{name=3, anchor=center, inner sep=0}, "{F\id_{x_i}}"', "\shortmid"{marking}, from=4-1, to=4-2]
      \arrow["{F_{\Pi\vec{x}}}"{description}, draw=none, from=0, to=1]
      \arrow["{F\Pi_{\vec{x}}}"{description, pos=0.6}, draw=none, from=1, to=2]
      \arrow["{F\pi_i}"{description, pos=0.6}, draw=none, from=2, to=3]
    \end{tikzcd}
    =
    \begin{tikzcd}
      {F\Pi\vec{x}} & {F\Pi\vec{x}} \\
      {F\Pi\vec{x}} & {F\Pi\vec{x}} \\
      {Fx_i} & {Fx_i}
      \arrow[""{name=0, anchor=center, inner sep=0}, "{\id_{F\Pi\vec{x}}}", "\shortmid"{marking}, from=1-1, to=1-2]
      \arrow[Rightarrow, no head, from=1-1, to=2-1]
      \arrow[Rightarrow, no head, from=1-2, to=2-2]
      \arrow[""{name=1, anchor=center, inner sep=0}, "{F\id_{\Pi\vec{x}}}"', "\shortmid"{marking}, from=2-1, to=2-2]
      \arrow[""{name=2, anchor=center, inner sep=0}, "{F\id_{x_i}}"', "\shortmid"{marking}, from=3-1, to=3-2]
      \arrow["{F\pi_i}", from=2-2, to=3-2]
      \arrow["{F\pi_i}"', from=2-1, to=3-1]
      \arrow["{F_{\Pi\vec{x}}}"{description}, draw=none, from=0, to=1]
      \arrow["{F\id_{\pi_i}}"{description, pos=0.6}, draw=none, from=1, to=2]
    \end{tikzcd},
  \end{equation*}
  where the last equality uses Equation \eqref{eq:product-identity-comparison}.
  On the other hand, the post-composite with the right-hand side of Equation
  \eqref{eq:unitor-product} is
  \begin{equation*}
    \begin{tikzcd}
      {F\Pi\vec{x}} & {F\Pi\vec{x}} \\
      {\Pi F\vec{x}} & {\Pi F\vec{x}} \\
      {Fx_i} & {Fx_i}
      \arrow[""{name=0, anchor=center, inner sep=0}, "{\id_{F\Pi\vec{x}}}", "\shortmid"{marking}, from=1-1, to=1-2]
      \arrow[""{name=1, anchor=center, inner sep=0}, "{\Pi F\id_{\vec{x}}}"', "\shortmid"{marking}, from=2-1, to=2-2]
      \arrow["{\Phi_{\vec{x}}}"', from=1-1, to=2-1]
      \arrow["{\Phi_{\vec{x}}}", from=1-2, to=2-2]
      \arrow[""{name=2, anchor=center, inner sep=0}, "{F\id_{x_i}}"', "\shortmid"{marking}, from=3-1, to=3-2]
      \arrow["{\pi_i}"', from=2-1, to=3-1]
      \arrow["{\pi_i}", from=2-2, to=3-2]
      \arrow["{\text{RHS}}"{description}, draw=none, from=0, to=1]
      \arrow["{\pi_i}"{description, pos=0.6}, draw=none, from=1, to=2]
    \end{tikzcd}
    =
    \begin{tikzcd}
      {F\Pi\vec{x}} & {F\Pi\vec{x}} \\
      {\Pi F\vec{x}} & {\Pi F\vec{x}} \\
      {\Pi F\vec{x}} & {\Pi F\vec{x}} \\
      {Fx_i} & {Fx_i} \\
      {Fx_i} & {Fx_i}
      \arrow[""{name=0, anchor=center, inner sep=0}, "{\id_{F\Pi\vec{x}}}", "\shortmid"{marking}, from=1-1, to=1-2]
      \arrow[""{name=1, anchor=center, inner sep=0}, "{\id_{Fx_i}}"', "\shortmid"{marking}, from=4-1, to=4-2]
      \arrow[Rightarrow, no head, from=2-1, to=3-1]
      \arrow[Rightarrow, no head, from=2-2, to=3-2]
      \arrow["{\Phi_{\vec{x}}}"', from=1-1, to=2-1]
      \arrow["{\Phi_{\vec{x}}}", from=1-2, to=2-2]
      \arrow[""{name=2, anchor=center, inner sep=0}, "{\id_{\Pi F\vec{x}}}"', "\shortmid"{marking}, from=2-1, to=2-2]
      \arrow[""{name=3, anchor=center, inner sep=0}, "{\Pi \id_{F\vec{x}}}"', "\shortmid"{marking}, from=3-1, to=3-2]
      \arrow[Rightarrow, no head, from=4-1, to=5-1]
      \arrow["{\pi_i}"', from=3-1, to=4-1]
      \arrow["{\pi_i}", from=3-2, to=4-2]
      \arrow[Rightarrow, no head, from=4-2, to=5-2]
      \arrow[""{name=4, anchor=center, inner sep=0}, "{F \id_{x_i}}"', "\shortmid"{marking}, from=5-1, to=5-2]
      \arrow["{\id_{\Phi_{\vec{x}}}}"{description}, draw=none, from=0, to=2]
      \arrow["{\Pi_{F\vec{x}}}"{description, pos=0.6}, draw=none, from=2, to=3]
      \arrow["{F_{x_i}}"{description, pos=0.6}, draw=none, from=1, to=4]
      \arrow["{\pi_i}"{description, pos=0.6}, draw=none, from=3, to=1]
    \end{tikzcd}
    =
    \begin{tikzcd}
      {F\Pi\vec{x}} & {F\Pi\vec{x}} \\
      {\Pi F\vec{x}} & {\Pi F\vec{x}} \\
      {Fx_i} & {Fx_i} \\
      {Fx_i} & {Fx_i}
      \arrow[""{name=0, anchor=center, inner sep=0}, "{\id_{F\Pi\vec{x}}}", "\shortmid"{marking}, from=1-1, to=1-2]
      \arrow[""{name=1, anchor=center, inner sep=0}, "{\id_{Fx_i}}"', "\shortmid"{marking}, from=3-1, to=3-2]
      \arrow["{\Phi_{\vec{x}}}"', from=1-1, to=2-1]
      \arrow["{\Phi_{\vec{x}}}", from=1-2, to=2-2]
      \arrow[""{name=2, anchor=center, inner sep=0}, "{\id_{\Pi F\vec{x}}}"', "\shortmid"{marking}, from=2-1, to=2-2]
      \arrow[Rightarrow, no head, from=3-1, to=4-1]
      \arrow[Rightarrow, no head, from=3-2, to=4-2]
      \arrow[""{name=3, anchor=center, inner sep=0}, "{F \id_{x_i}}"', "\shortmid"{marking}, from=4-1, to=4-2]
      \arrow["{\pi_i}"', from=2-1, to=3-1]
      \arrow["{\pi_i}", from=2-2, to=3-2]
      \arrow["{\id_{\Phi_{\vec{x}}}}"{description}, draw=none, from=0, to=2]
      \arrow["{F_{x_i}}"{description, pos=0.6}, draw=none, from=1, to=3]
      \arrow["{\id_{\pi_i}}"{description, pos=0.6}, draw=none, from=2, to=1]
    \end{tikzcd}
    =
    \begin{tikzcd}
      {F\Pi\vec{x}} & {F\Pi\vec{x}} \\
      {Fx_i} & {Fx_i} \\
      {Fx_i} & {Fx_i}
      \arrow[""{name=0, anchor=center, inner sep=0}, "{\id_{F\Pi\vec{x}}}", "\shortmid"{marking}, from=1-1, to=1-2]
      \arrow[""{name=1, anchor=center, inner sep=0}, "{\id_{Fx_i}}"', "\shortmid"{marking}, from=2-1, to=2-2]
      \arrow[Rightarrow, no head, from=2-1, to=3-1]
      \arrow[Rightarrow, no head, from=2-2, to=3-2]
      \arrow[""{name=2, anchor=center, inner sep=0}, "{F \id_{x_i}}"', "\shortmid"{marking}, from=3-1, to=3-2]
      \arrow["{F\pi_i}"', from=1-1, to=2-1]
      \arrow["{F\pi_i}", from=1-2, to=2-2]
      \arrow["{F_{x_i}}"{description, pos=0.6}, draw=none, from=1, to=2]
      \arrow["{\id_{F\pi_i}}"{description}, draw=none, from=0, to=1]
    \end{tikzcd}.
  \end{equation*}
  But these are equal by the naturality of the unitors of $F$ with respect to
  the projection arrow $\pi_i: \Pi\vec{x} \to x_i$ in $\dbl{D}$. By the
  universal property of the product $\Pi F \id_{\vec{x}}$,
  Equation \eqref{eq:unitor-product} holds.
\end{proof}

An analogous lemma for lax functors between cartesian double categories has been
used to study models of cartesian double theories \citep[\mbox{Lemma
  5.2}]{lambert2024}. We will use Lemma \ref{lem:laxator-unitor-product} for the
same purpose when we define finite-product double theories and their models in
the next section.

\section{Finite-product double theories}
\label{sec:theories}

As an application of the tools developed so far, we extend the cartesian double
theories recently proposed in \citep{lambert2024} by allowing arbitrary finite
products instead of only finite parallel products as in a cartesian double
category. Double theories with finite products have greater expressivity,
encompassing among other things an important class of enriched categories, yet
require minimal modifications to the machinery producing a virtual double
category of models. This development brings us full circle to Paré's early work
on products in double categories \citep{pare2009}, which was also motivated by
double theories, albeit with different intended semantics and models.

\begin{definition}[Finite-product double theories and models]
  A \define{finite-product double theory} is a small, strict double category
  with strong finite products. A \define{morphism} of finite-product double
  theories is a strict double functor that preserves finite products.

  A \define{model} of a finite-product double theory $\dbl{T}$ in a double
  category $\dbl{S}$ with lax finite products is a lax double functor
  $M: \dbl{T} \to \dbl{S}$ that preserves finite products. If left unstated, the
  semantics for a model is assumed to be $\dbl{S} = \Span$.
\end{definition}

The name ``finite-product double theory'' is clearly intended to evoke the
finite-product theories familiar from categorical logic. Indeed, they categorify
that concept. The differing strengths of products in the syntax and the
semantics is important. A typical semantics, such as spans or matrices, will
have at most iso-strong products, whereas the theory must have strong products
in order to fully control the comparison cells of its models, as shown by
Lemma \ref{lem:laxator-unitor-product}.

All of the cartesian double theories in \citep{lambert2024} can be seen as
finite-product double theories that happen to use only finite parallel products.
The examples presented here take advantage of the extra flexibility afforded by
arbitrary finite products. To avoid confusion with parallel products, we write
local products using conjunctive notation, so that $m \wedge n: x \proto y$ is
the local product of two proarrows $m,n: x \proto y$ with common source and
target, and $\top \coloneqq \top_{x,y}: x \proto y$ is the local terminal
between two objects $x,y$. The same notation is used in the literature on
cartesian bicategories \citep[\S{3.1}]{carboni2008}.

\begin{example}[Categories enriched in commutative monoids]
  \label{ex:cmon-enrichment}
  The \define{theory of local commutative monoids} $\Th{\cat{lcMon}}$ is the
  finite-product double theory generated by a single object $x$ and two globular
  cells, the \define{local multiplication} and \define{unit}
  \begin{equation*}
    \begin{tikzcd}
      x & x \\
      x & x
      \arrow[""{name=0, anchor=center, inner sep=0}, "{\id_x \wedge \id_x}", "\shortmid"{marking}, from=1-1, to=1-2]
      \arrow[""{name=1, anchor=center, inner sep=0}, "{\id_x}"', "\shortmid"{marking}, from=2-1, to=2-2]
      \arrow[Rightarrow, no head, from=1-1, to=2-1]
      \arrow[Rightarrow, no head, from=1-2, to=2-2]
      \arrow["\mu"{description}, draw=none, from=0, to=1]
    \end{tikzcd}
    \qquad\text{and}\qquad
    \begin{tikzcd}
      x & x \\
      x & x
      \arrow[""{name=0, anchor=center, inner sep=0}, "\top", "\shortmid"{marking}, from=1-1, to=1-2]
      \arrow[""{name=1, anchor=center, inner sep=0}, "{\id_x}"', "\shortmid"{marking}, from=2-1, to=2-2]
      \arrow[Rightarrow, no head, from=1-1, to=2-1]
      \arrow[Rightarrow, no head, from=1-2, to=2-2]
      \arrow["\eta"{description}, draw=none, from=0, to=1]
    \end{tikzcd},
  \end{equation*}
  subject to the usual equations of associativity, unitality, and commutativity:
  \begin{equation*}
    \begin{tikzcd}[column sep=scriptsize]
      {\id_x \wedge \id_x \wedge \id_x} & {\id_x \wedge \id_x} \\
      {\id_x \wedge \id_x} & {\id_x}
      \arrow["{1 \wedge \mu}"', from=1-1, to=2-1]
      \arrow["{\mu \wedge 1}", from=1-1, to=1-2]
      \arrow["\mu"', from=2-1, to=2-2]
      \arrow["\mu", from=1-2, to=2-2]
    \end{tikzcd}
    \quad
    \begin{tikzcd}[column sep=scriptsize]
      {\top \wedge\id_x} & {\id_x \wedge \id_x} & {\id_x \wedge \top} \\
      & {\id_x}
      \arrow["\cong"', from=1-1, to=2-2]
      \arrow["\mu"', from=1-2, to=2-2]
      \arrow["{\eta \wedge 1}", from=1-1, to=1-2]
      \arrow["{1 \wedge \eta}"', from=1-3, to=1-2]
      \arrow["\cong", from=1-3, to=2-2]
    \end{tikzcd}
    \quad
    \begin{tikzcd}
      {\id_x \wedge \id_x} & {\id_x \wedge \id_x} \\
      & {\id_x}
      \arrow["{\sigma_{\id_x,\id_x}}", from=1-1, to=1-2]
      \arrow["\mu", from=1-2, to=2-2]
      \arrow["\mu"', from=1-1, to=2-2]
    \end{tikzcd}.
  \end{equation*}

  We claim that a model of the theory of local commutative monoids, namely a
  finite-product-preserving lax functor $M: \Th{\cat{lcMon}} \to \Span$, is
  precisely a category enriched in commutative monoids. First, the data
  $(Mx, M\id_x, M_{x,x}, M_x)$, comprising the image of $x$ and
  $\id_x: x \proto x$ and corresponding laxator and unitor, is a category
  \citep[\S{2}]{lambert2024}, call it $\cat{C}$. Moreover, Lemma
  \ref{lem:laxator-unitor-product} ensures that all of the other laxators and
  unitors of $M$ are uniquely determined by products of $M_{x,x}$ and $M_x$. The
  images of $\mu$ and $\eta$ under $M$ are natural transformations
  \citep[\S{2}]{lambert2024}, conventionally written in additive notation as
  \begin{equation*}
    +: \cat{C}(a,b)^2 \to \cat{C}(a,b)
    \qquad\text{and}\qquad
    0: 1 \to \cat{C}(a,b)
  \end{equation*}
  for each pair of objects $a,b \in \cat{C}$. The three equational axioms on
  $\mu$ and $\eta$ say that each hom-set of $\cat{C}$ is endowed with the
  structure of a commutative monoid. Finally, the naturality conditions say that
  for all morphisms $h: a' \to a$ and $k: b \to b'$ in $\cat{C}$, the diagrams
  \begin{equation*}
    \begin{tikzcd}
      {\cat{C}(a,b)^2} & {\cat{C}(a,b)} \\
      {\cat{C}(a',b')^2} & {\cat{C}(a',b')}
      \arrow["{\cat{C}(h,k)^2}"', from=1-1, to=2-1]
      \arrow["{+}", from=1-1, to=1-2]
      \arrow["{\cat{C}(h,k)}", from=1-2, to=2-2]
      \arrow["{+}"', from=2-1, to=2-2]
    \end{tikzcd}
    \qquad\text{and}\qquad
    \begin{tikzcd}
      1 & {\cat{C}(a,b)} \\
      & {\cat{C}(a',b')}
      \arrow["0", from=1-1, to=1-2]
      \arrow["0"', from=1-1, to=2-2]
      \arrow["{\cat{C}(h,k)}", from=1-2, to=2-2]
    \end{tikzcd}
  \end{equation*}
  commute, that is, for all morphisms $f,g: a \to b$ in $\cat{C}$,
  \begin{equation*}
    h \cdot (f + g) \cdot k = (h \cdot f \cdot k) + (h \cdot g \cdot k)
    \qquad\text{and}\qquad
    h \cdot 0_{a,b} \cdot k = 0_{a',b'}.
  \end{equation*}
  These equations are equivalent to the biadditivity that makes the category
  $\cat{C}$ be enriched in commutative monoids.
\end{example}

Enriched categories are usually defined for a base of enrichment that is a
monoidal category or a multicategory. Perhaps surprisingly, the double theory in
Example \ref{ex:cmon-enrichment} makes no reference to the monoidal category of
commutative monoids under its tensor product or to the multicategory of
commutative monoids and multi-additive maps. Instead, the naturality
requirements of a lax functor automatically ensure that composition is
biadditive. That models of double theories are functorial and natural by
construction is a key advantage compared with models of one-dimensional theories
such as finite-limit theories.

The property of commutative monoids that makes Example \ref{ex:cmon-enrichment}
work is that the category of commutative monoids is algebraic. Recall that an
\define{algebraic category} is a category concretely equivalent to the category
of models of an algebraic theory, which we assume to be single-sorted; see
\citep[\mbox{Chapter 11}]{adamek2010} or \citep[\mbox{Vol.\ 2}, \mbox{Chapter
  3}]{borceux1994}. In general, given a commutative algebraic theory $\cat{T}$,
we can construct a finite-product double theory whose models are categories
enriched in models of $\cat{T}$. This includes the important example of
categories enriched in abelian groups or, more generally, enriched in
$R$-modules over a commutative ring $R$.

\begin{construction}[Categories enriched in an algebraic category]
  \label{ex:algebraic-category-enrichment}
  Let $\cat{T}$ be an algebraic theory (not necessarily commutative) with
  generating object $x$. We define a finite-product double theory $\dbl{T}$
  generated by a single object, also denoted $x$, and by, for each morphism
  $f: x^m \to x^n$ in $\cat{T}$, a globular cell
  $F(f): \id_x^{\wedge m} \to \id_x^{\wedge n}$, where $\id_x^{\wedge n}$ denotes the $n$-fold
  local product $\id_x \wedge \cdots \wedge \id_x: x \proto x$ of the identity
  $\id_x: x \proto x$ with itself. We impose relations making the assignment
  $F: \cat{T} \to \dbl{T}(x,x)$ into a finite-product-preserving functor from
  $\cat{T}$ into the hom-category $\dbl{T}(x,x)$ of proarrows $x \proto x$ and
  globular cells. Abstractly, the double category $\dbl{T}$ with finite products
  is the \emph{delooping} of the category $\cat{T}$ with finite products.

  When the algebraic theory $\cat{T}$ is commutative, its category of models,
  $\cat{Mod}_{\cat{T}}$, has a symmetric monoidal closed structure
  \citep[\mbox{Vol.\ 2}, \mbox{Theorem 3.10.3}]{borceux1994}. Models of the
  finite-product double theory $\dbl{T}$ are then equivalent to categories
  enriched over $\cat{Mod}_{\cat{T}}$, similarly to Example \ref{ex:cmon-enrichment}.
  In fact, as Borceux remarks, the hypothesis of commutativity is not actually
  needed to construct the tensor product on $\cat{Mod}_{\cat{T}}$, only to make
  it well behaved; nor, as we have already noted, is the tensor product used at
  all in constructing the double theory.
\end{construction}

The enrichment construction can be combined with other double theories from
\citep{lambert2024} to present finite-product double theories whose models are,
say, monoidal categories or multicategories enriched over an algebraic category.
We leave these examples to the interested reader and turn to a rather different
application of finite products: using restrictions of products along structure
arrows (Theorem \ref{thm:dbl-products-characterization}) to axiomatize several
flavors of generalized multicategory \citep[\S{4}]{pisani2014},
\citep{shulman2016}.

\begin{example}[Cartesian and symmetric multicategories]
  The \define{theory of promonoids} is the finite-product double theory
  generated by one object $x$ and two proarrows, $m: x^2 \proto x$ and
  $j: 1 \proto x$, subject to axioms of associativity and unitality,
  \begin{equation*}
    \begin{tikzcd}
      {x^3} & {x^2} \\
      {x^2} & x
      \arrow["{(p \times \id_x)}", "\shortmid"{marking}, from=1-1, to=1-2]
      \arrow["p", "\shortmid"{marking}, from=1-2, to=2-2]
      \arrow["p"', "\shortmid"{marking}, from=2-1, to=2-2]
      \arrow["{(\id_x \times p)}"', "\shortmid"{marking}, from=1-1, to=2-1]
    \end{tikzcd}
    \qquad\text{and}\qquad
    \begin{tikzcd}
      {1 \times x} & {x^2} & {x \times 1} \\
      & x
      \arrow["{j \times \id_x}", "\shortmid"{marking}, from=1-1, to=1-2]
      \arrow["p"', "\shortmid"{marking}, from=1-2, to=2-2]
      \arrow["{\id_x \times j}"', "\shortmid"{marking}, from=1-3, to=1-2]
      \arrow["\cong"', "\shortmid"{marking}, from=1-1, to=2-2]
      \arrow["\cong", "\shortmid"{marking}, from=1-3, to=2-2]
    \end{tikzcd}.
  \end{equation*}
  The axioms ensure that, for each arity $n \geq 0$, there is a unique proarrow
  $p_n: x^n \proto x$ built out of the generators $p$ and $j$. For example,
  $p_1 = \id_x$ and $p_4 = (p \times p) \odot p$. So far we have simply recalled
  the cartesian double theory of promonoids \citep[\mbox{Theory
    6.9}]{lambert2024} but we now regard it as a finite-product double theory.
  Models of the theory are equivalent to multicategories.

  Let $\cat{F}$ be the skeleton of $\FinSet$ spanned by the sets
  $[n] \coloneqq \{1,\dots,n\}$ for each $n \geq 0$. The \define{theory of
    cartesian promonoids} is the finite-product double theory presented as the
  theory of promonoids augmented with, for each function $\sigma: [m] \to [n]$
  in $\cat{F}$, a globular cell
  \begin{equation*}
    \begin{tikzcd}[row sep=scriptsize]
      {x^n} & {x^m} & x \\
      {x^n} && x
      \arrow["{x^\sigma_!}", "\shortmid"{marking}, from=1-1, to=1-2]
      \arrow["{p_m}", "\shortmid"{marking}, from=1-2, to=1-3]
      \arrow[Rightarrow, no head, from=1-3, to=2-3]
      \arrow[Rightarrow, no head, from=1-1, to=2-1]
      \arrow[""{name=0, anchor=center, inner sep=0}, "{p_n}"', "\shortmid"{marking}, from=2-1, to=2-3]
      \arrow["{\rho(\sigma)}"{description, pos=0.4}, draw=none, from=1-2, to=0]
    \end{tikzcd},
  \end{equation*}
  called the \define{action by $\sigma$}, where
  $x^\sigma \coloneq \Pi(\sigma): x^n \to x^m$ is the structure arrow between
  products induced by $\sigma$ and $x^\sigma_! = \Pi(\sigma)_!: x^n \proto x^m$
  is its companion (Corollary \ref{cor:structure-proarrows}). The composite proarrow
  $p_m^\sigma \coloneqq x^\sigma_! \odot p_m$ is thus the restriction of $p_m$
  along the arrows $x^\sigma$ and $1_x$. The action cells obey the following
  relations.
  \begin{itemize}
    \item Functorality of action: for every pair of composable maps
      $[m] \xto{\sigma} [n] \xto{\tau} [q]$ in $\cat{F}$, we have
      \begin{equation*}
        \begin{tikzcd}[row sep=scriptsize]
          {x^q} & {x^n} & {x^m} & x \\
          {x^q} & {x^n} && x \\
          {x^q} &&& x
          \arrow[Rightarrow, no head, from=1-2, to=2-2]
          \arrow[""{name=0, anchor=center, inner sep=0}, "{x^\sigma_!}", "\shortmid"{marking}, from=1-2, to=1-3]
          \arrow["{p_m}", "\shortmid"{marking}, from=1-3, to=1-4]
          \arrow[Rightarrow, no head, from=1-4, to=2-4]
          \arrow[""{name=1, anchor=center, inner sep=0}, "{p_n}"', "\shortmid"{marking}, from=2-2, to=2-4]
          \arrow[Rightarrow, no head, from=2-4, to=3-4]
          \arrow[""{name=2, anchor=center, inner sep=0}, "{x^\tau_!}", "\shortmid"{marking}, from=1-1, to=1-2]
          \arrow[Rightarrow, no head, from=1-1, to=2-1]
          \arrow[""{name=3, anchor=center, inner sep=0}, "{x^\tau_!}"', "\shortmid"{marking}, from=2-1, to=2-2]
          \arrow[""{name=4, anchor=center, inner sep=0}, "{p_q}"', "\shortmid"{marking}, from=3-1, to=3-4]
          \arrow[Rightarrow, no head, from=2-1, to=3-1]
          \arrow["{\rho(\sigma)}"{description, pos=0.4}, draw=none, from=1-3, to=1]
          \arrow["1"{description}, draw=none, from=2, to=3]
          \arrow["{\rho(\tau)}"{description, pos=0.8}, draw=none, from=0, to=4]
        \end{tikzcd}
        \quad=\quad
        \begin{tikzcd}[row sep=scriptsize]
          {x^q} & {x^n} & {x^m} & x \\
          {x^q} && {x^m} & x \\
          {x^q} &&& x
          \arrow[Rightarrow, no head, from=2-1, to=3-1]
          \arrow[""{name=0, anchor=center, inner sep=0}, "{x^{\sigma \cdot \tau}_!}"', "\shortmid"{marking}, from=2-1, to=2-3]
          \arrow[""{name=1, anchor=center, inner sep=0}, "{p_m}"', "\shortmid"{marking}, from=2-3, to=2-4]
          \arrow[Rightarrow, no head, from=2-4, to=3-4]
          \arrow[""{name=2, anchor=center, inner sep=0}, "{p_q}"', "\shortmid"{marking}, from=3-1, to=3-4]
          \arrow[Rightarrow, no head, from=1-3, to=2-3]
          \arrow[Rightarrow, no head, from=1-4, to=2-4]
          \arrow[""{name=3, anchor=center, inner sep=0}, "{p_m}", "\shortmid"{marking}, from=1-3, to=1-4]
          \arrow[Rightarrow, no head, from=1-1, to=2-1]
          \arrow[""{name=4, anchor=center, inner sep=0}, "{x^\sigma_!}", "\shortmid"{marking}, from=1-2, to=1-3]
          \arrow["{x^\tau_!}", "\shortmid"{marking}, from=1-1, to=1-2]
          \arrow["1"{description}, draw=none, from=3, to=1]
          \arrow["\cong"{description, pos=0.4}, draw=none, from=1-2, to=0]
          \arrow["{\rho(\sigma \cdot \tau)}"{description, pos=0.8}, draw=none, from=4, to=2]
        \end{tikzcd},
      \end{equation*}
      and for every $n \geq 0$, the cell $\rho(1_{[n]})$ is the canonical
      isomorphism $x^{1_{[n]}}_! \odot p_n \cong p_n$.
    \item Naturality of action (i): for all $k \geq 0$ and maps
      $\sigma_i: [m_i] \to [n_i]$, $1 \leq i \leq k$, in $\cat{F}$, we have
      \begin{equation*}
        \begin{tikzcd}[row sep=scriptsize]
          {x^n} && {x^k} & x \\
          {x^n} && {x^k} & x \\
          {x^n} &&& x
          \arrow[""{name=0, anchor=center, inner sep=0}, "{p_{m_1}^{\sigma_1} \times \cdots \times p_{m_k}^{\sigma_k}}", "\shortmid"{marking}, from=1-1, to=1-3]
          \arrow[""{name=1, anchor=center, inner sep=0}, "{p_k}", "\shortmid"{marking}, from=1-3, to=1-4]
          \arrow[""{name=2, anchor=center, inner sep=0}, "{p_k}"', "\shortmid"{marking}, from=2-3, to=2-4]
          \arrow[Rightarrow, no head, from=1-3, to=2-3]
          \arrow[Rightarrow, no head, from=1-4, to=2-4]
          \arrow[Rightarrow, no head, from=1-1, to=2-1]
          \arrow[""{name=3, anchor=center, inner sep=0}, "{p_{n_1} \times \cdots \times p_{n_k}}"', "\shortmid"{marking}, from=2-1, to=2-3]
          \arrow[Rightarrow, no head, from=2-1, to=3-1]
          \arrow[Rightarrow, no head, from=2-4, to=3-4]
          \arrow[""{name=4, anchor=center, inner sep=0}, "{p_n}"', "\shortmid"{marking}, from=3-1, to=3-4]
          \arrow["{1_{p_k}}"{description}, draw=none, from=1, to=2]
          \arrow["{\rho(\sigma_1) \times \cdots \times \rho(\sigma_k)}"{description}, draw=none, from=0, to=3]
          \arrow["1"{description}, draw=none, from=2-3, to=4]
        \end{tikzcd}
        \quad=\quad
        \begin{tikzcd}[row sep=scriptsize]
          {x^n} && {x^k} & x \\
          {x^n} &&& x \\
          {x^n} &&& x
          \arrow[""{name=0, anchor=center, inner sep=0}, "{p_n}"', "\shortmid"{marking}, from=3-1, to=3-4]
          \arrow[""{name=1, anchor=center, inner sep=0}, "{p_m^{\sigma_1 + \cdots + \sigma_k}}", "\shortmid"{marking}, from=2-1, to=2-4]
          \arrow[Rightarrow, no head, from=2-1, to=3-1]
          \arrow[Rightarrow, no head, from=2-4, to=3-4]
          \arrow["{p_k}", "\shortmid"{marking}, from=1-3, to=1-4]
          \arrow[Rightarrow, no head, from=1-4, to=2-4]
          \arrow[Rightarrow, no head, from=1-1, to=2-1]
          \arrow["{p_{m_1}^{\sigma_1} \times \cdots \times p_{m_k}^{\sigma_k}}", "\shortmid"{marking}, from=1-1, to=1-3]
          \arrow["{\rho(\sigma_1 + \cdots + \sigma_k)}"{description}, draw=none, from=1, to=0]
          \arrow["\cong"{description, pos=0.4}, shift right=3, draw=none, from=1-3, to=1]
        \end{tikzcd},
      \end{equation*}
      where we put $m \coloneqq m_1 + \cdots + m_k$ and
      $n \coloneqq n_1 + \cdots + n_k$ and the isomorphism on the right-hand
      side is built out of the canonical isomorphisms
      (Proposition \ref{prop:companions-product})
      \begin{equation*}
        x^{\sigma_1}_! \times \cdots \times x^{\sigma_k}_!
          \cong (x^{\sigma_1} \times \cdots \times x^{\sigma_k})_!
          \cong x^{\sigma_1 + \cdots + \sigma_k}_!.
      \end{equation*}
    \item Naturality of action (ii): for every map $\sigma: [k] \to [\ell]$ in
      $\cat{F}$ and all $n_j \geq 0$, $1 \leq j \leq \ell$, we have
      \begin{equation*}
        \begin{tikzcd}[row sep=scriptsize]
          {x^n} && {x^\ell} & x \\
          {x^n} && {x^\ell} & x \\
          {x^n} &&& x
          \arrow[""{name=0, anchor=center, inner sep=0}, "{p_{n_1} \times \cdots \times p_{n_\ell}}", "\shortmid"{marking}, from=1-1, to=1-3]
          \arrow[""{name=1, anchor=center, inner sep=0}, "{p_{n_1} \times \cdots \times p_{n_\ell}}"', "\shortmid"{marking}, from=2-1, to=2-3]
          \arrow[Rightarrow, no head, from=1-1, to=2-1]
          \arrow[Rightarrow, no head, from=1-3, to=2-3]
          \arrow[""{name=2, anchor=center, inner sep=0}, "{p_k^\sigma}", "\shortmid"{marking}, from=1-3, to=1-4]
          \arrow[""{name=3, anchor=center, inner sep=0}, "{p_\ell}"', "\shortmid"{marking}, from=2-3, to=2-4]
          \arrow[Rightarrow, no head, from=1-4, to=2-4]
          \arrow[""{name=4, anchor=center, inner sep=0}, "{p_n}"', "\shortmid"{marking}, from=3-1, to=3-4]
          \arrow[Rightarrow, no head, from=2-1, to=3-1]
          \arrow[Rightarrow, no head, from=2-4, to=3-4]
          \arrow["1"{description}, draw=none, from=0, to=1]
          \arrow["{\rho(\sigma)}"{description}, draw=none, from=2, to=3]
          \arrow["1"{description}, draw=none, from=2-3, to=4]
        \end{tikzcd}
        \quad=\quad
        \begin{tikzcd}[row sep=scriptsize]
          {x^n} && {x^\ell} & x \\
          {x^n} &&& x \\
          {x^n} &&& x
          \arrow[""{name=0, anchor=center, inner sep=0}, "{p_m^{\sigma \wr (n_1,\dots,n_\ell)}}", "\shortmid"{marking}, from=2-1, to=2-4]
          \arrow[""{name=1, anchor=center, inner sep=0}, "{p_n}"', "\shortmid"{marking}, from=3-1, to=3-4]
          \arrow[Rightarrow, no head, from=2-1, to=3-1]
          \arrow[Rightarrow, no head, from=2-4, to=3-4]
          \arrow[Rightarrow, no head, from=1-4, to=2-4]
          \arrow[Rightarrow, no head, from=1-1, to=2-1]
          \arrow["{p_{n_1} \times \cdots \times p_{n_\ell}}", "\shortmid"{marking}, from=1-1, to=1-3]
          \arrow["{p_k^\sigma}", "\shortmid"{marking}, from=1-3, to=1-4]
          \arrow["{\rho(\sigma \wr (n_1,\dots,n_\ell))}"{description}, draw=none, from=0, to=1]
          \arrow["\cong"{description, pos=0.4}, shift right=3, draw=none, from=1-3, to=0]
        \end{tikzcd},
      \end{equation*}
      where we put $m \coloneqq n_{\sigma 1} + \cdots + n_{\sigma k}$ and
      $n \coloneqq n_1 + \cdots + n_\ell$, and the map
      $\sigma \wr (n_1,\dots,n_\ell): [m] \to [n]$ applies $\sigma$ blockwise
      and is the identity within each block of size $n_{\sigma i}$.

      The isomorphism on the right-hand side uses a canonical isomorphism
      between products obtained as follows. The family of proarrows
      $p_{n_j}: x^{n_j} \proto x$, $1 \leq j \leq \ell$, indexed by the identity
      span $\id_{[\ell]}$, and the family of identity proarrows $\id_x$, indexed
      by the span $\sigma^* = ([\ell] \xfrom{\sigma} [k] = [k])$, have as their
      composite the family $p_{n_{\sigma i}}: x^{n_{\sigma i}} \proto x$,
      $1 \leq i \leq k$, also indexed by the span $\sigma^*$. But the same is
      true for the family of identity proarrows
      $\id_{x^{n_{\sigma i}}} \cong \id_x^{n_{\sigma i}}$, $i \leq 1 \leq k$,
      indexed by $\sigma^*$, composed with the proarrow family
      $p_{n_{\sigma i}}: x^{n_{\sigma i}} \proto x$, $1 \leq i \leq k$, indexed
      by the identity span $\id_{[k]}$. Since a finite-product double theory has
      \emph{strong} finite products, we obtain a canonical isomorphism
      \begin{equation} \label{eq:th-generalized-multicategory-canonical-iso}
        (p_{n_1} \times \cdots \times p_{n_\ell}) \odot x^\sigma_! \cong
          x^{\sigma \wr (n_1,\dots,n_\ell)}_! \odot
          (p_{n_{\sigma 1}} \times \cdots \times p_{n_{\sigma k}}),
      \end{equation}
      which, by composing with the identity $1_{p_k}$ on the right, yields the
      isomorphism used in the right-hand side.
  \end{itemize}
  A model of the theory of cartesian promonoids is equivalent to a cartesian
  multicategory. This follows by reasoning similar to, but simpler than, that in
  \citep[\mbox{Proposition 6.21}]{lambert2024}, now appealing to Lemma
  \ref{lem:laxator-unitor-product} to control the model's laxators for
  companions.

  The \define{theory of symmetric promonoids} is defined in the same way except
  that the skeleton of $\FinSet$ is replaced by its core, the skeleton of the
  groupoid of finite sets and bijections. Models of the theory of symmetric
  promonoids are equivalent to symmetric multicategories.
\end{example}

\begin{remark}[Comparison with restriction sketches]
  Possessing restrictions of products along structure arrows, finite-product
  double theories subsume some, though not all, of the intended applications of
  ``restriction sketches,'' a provisional notion introduced in
  \citep[\mbox{Definition 6.17}]{lambert2024}. While, unlike an equipment or a
  restriction sketch, finite-product double theories do not allow restriction
  along \emph{arbitrary} arrows, they have the advantage over equipments of not
  producing uncontrolled laxators, as the additional laxators here are uniquely
  determined by Lemma \ref{lem:laxator-unitor-product}. And they have the
  advantage over restriction sketches of not needing to manually impose extra
  equations that should follow from properties of products and restrictions;
  compare the \emph{derived} Equation
  \eqref{eq:th-generalized-multicategory-canonical-iso} above with the
  \emph{postulated} \citep[\mbox{Equations 6.3 \& 6.4}]{lambert2024}.
\end{remark}

\section{Higher morphisms and double categories of models}
\label{sec:models}

If models of finite-product double theories are product-preserving lax functors,
then morphisms between models, and cells between those, ought to be the higher
morphisms between lax functors that also preserve finite products. In this final
section, we determine the sense in which lax natural transformations and modules
between lax functors, and modulations between squares of those, should or
already do preserve products. The definitions and lemmas are all more or less
obvious, following the pattern of Lemma \ref{lem:laxator-unitor-product} on lax
functors and having counterparts for cartesian double categories in previous
work \citep{lambert2024}. Nevertheless, it is important to state and prove these
results, as they ensure that the morphisms and cells between models of
finite-product double theories behave as intended and recover the expected
notions in examples of interest. We will conclude by showing that these higher
morphisms assemble into a unital virtual double category, and in particular a
2-category, of models of a finite-product double theory.

Natural transformations between product-preserving lax double functors require
no extra assumptions, just as ordinary natural transformations between
product-preserving functors do not. However, \emph{lax} natural transformations
between double functors, introduced in \citep[\S{7}]{lambert2024} to generalize
lax natural transformations between 2-functors, need an extra axiom that make
products be respected \emph{strictly}.

\begin{definition}[Product-preserving lax transformation]
  Let $F,G: \dbl{D} \to \dbl{E}$ be lax functors between double categories with
  lax products. A lax natural transformation $\alpha: F \To G$ \define{preserves
    products} if it is \emph{strictly} natural with respect to projections,
  meaning that for every family of objects $(I,\vec{x})$ in $\dbl{D}$ and every
  $i \in I$, the square
  \begin{equation*}
    \begin{tikzcd}
      {F\Pi\vec{x}} & {Fx_i} \\
      {G\Pi\vec{x}} & {Gx_i}
      \arrow["{\alpha_{\Pi\vec{x}}}"', from=1-1, to=2-1]
      \arrow["{F\pi_i}", from=1-1, to=1-2]
      \arrow["{\alpha_{x_i}}", from=1-2, to=2-2]
      \arrow["{G\pi_i}"', from=2-1, to=2-2]
    \end{tikzcd}
  \end{equation*}
  of arrows in $\dbl{E}$ commutes, and the naturality comparison for the
  projection $\pi_i$ has the form
  \begin{equation*}
    \begin{tikzcd}
      {F\Pi\vec{x}} & {F\Pi\vec{x}} \\
      {G\Pi\vec{x}} & {Fx_i} \\
      {Gx_i} & {Gx_i}
      \arrow["{\alpha_{\Pi\vec{x}}}"', from=1-1, to=2-1]
      \arrow["{G\pi_i}"', from=2-1, to=3-1]
      \arrow["{\alpha_{x_i}}", from=2-2, to=3-2]
      \arrow["{F\pi_i}", from=1-2, to=2-2]
      \arrow[""{name=0, anchor=center, inner sep=0}, "{G\id_{x_i}}"', "\shortmid"{marking}, from=3-1, to=3-2]
      \arrow[""{name=1, anchor=center, inner sep=0}, "{\id_{F\Pi\vec{x}}}", "\shortmid"{marking}, from=1-1, to=1-2]
      \arrow["{\alpha_{\pi_i}}"{description}, draw=none, from=1, to=0]
    \end{tikzcd}
    \quad=\quad
    \begin{tikzcd}[row sep=scriptsize]
      {F\Pi\vec{x}} & {F\Pi\vec{x}} \\
      {G\Pi\vec{x}} & {Fx_i} \\
      {Gx_i} & {Gx_i} \\
      {Gx_i} & {Gx_i}
      \arrow["{\alpha_{\Pi\vec{x}}}"', from=1-1, to=2-1]
      \arrow["{G\pi_i}"', from=2-1, to=3-1]
      \arrow["{\alpha_{x_i}}", from=2-2, to=3-2]
      \arrow["{F\pi_i}", from=1-2, to=2-2]
      \arrow[""{name=0, anchor=center, inner sep=0}, "{\id_{Gx_i}}", "\shortmid"{marking}, from=3-1, to=3-2]
      \arrow[""{name=1, anchor=center, inner sep=0}, "{\id_{F\Pi\vec{x}}}", "\shortmid"{marking}, from=1-1, to=1-2]
      \arrow[Rightarrow, no head, from=3-1, to=4-1]
      \arrow[Rightarrow, no head, from=3-2, to=4-2]
      \arrow[""{name=2, anchor=center, inner sep=0}, "{G\id_{x_i}}"', "\shortmid"{marking}, from=4-1, to=4-2]
      \arrow["\id"{description}, draw=none, from=1, to=0]
      \arrow["{G_{x_i}}"{description}, draw=none, from=0, to=2]
    \end{tikzcd}.
  \end{equation*}
\end{definition}

Lax natural transformation that restrict to strict natural transformations on
product projections have appeared previously in the literature on
two-dimensional categorical logic \citep[\mbox{Example 9.1}]{bourke2021}. The
$\mathcal{F}$-natural transformations from $\mathcal{F}$-category theory
\citep[\S{4.1}]{lack2012} are a general tool to handle such situations.

The next lemma generalizes \citep[\mbox{Lemma 7.11}]{lambert2024} on the
naturality comparisons of a cartesian lax natural transformation.

\begin{lemma}[Naturality comparisons for products]
  Let $F,G: \dbl{D} \to \dbl{E}$ be lax functors between double categories with
  lax products, and let $\alpha: F \To G$ be a product-preserving lax
  transformation. Then for every arrow $(f_0,f): (I,\vec{x}) \to (J,\vec{y})$ in
  $\DblFamOp(\dbl{D})$, we have
  \begin{equation} \label{eq:naturality-comparison-product}
    \begin{tikzcd}[column sep=large,row sep=scriptsize]
      {F\Pi\vec{x}} & {F\Pi\vec{x}} \\
      {G\Pi\vec{x}} & {F\Pi\vec{y}} \\
      {G\Pi\vec{y}} & {G\Pi\vec{y}} \\
      {G\Pi\vec{y}} & {G\Pi\vec{y}} \\
      {\Pi G\vec{y}} & {\Pi G\vec{y}}
      \arrow[""{name=0, anchor=center, inner sep=0}, "{\id_{F\Pi\vec{x}}}", "\shortmid"{marking}, from=1-1, to=1-2]
      \arrow["{\alpha_{\Pi\vec{x}}}"', from=1-1, to=2-1]
      \arrow["{G\Pi(f_0,f)}"', from=2-1, to=3-1]
      \arrow["{F\Pi(f_0,f)}", from=1-2, to=2-2]
      \arrow["{\alpha_{\Pi\vec{y}}}", from=2-2, to=3-2]
      \arrow[""{name=1, anchor=center, inner sep=0}, "{G\id_{\Pi\vec{y}}}", "\shortmid"{marking}, from=3-1, to=3-2]
      \arrow[Rightarrow, no head, from=3-1, to=4-1]
      \arrow[Rightarrow, no head, from=3-2, to=4-2]
      \arrow[""{name=2, anchor=center, inner sep=0}, "{G\Pi\id_{\vec{y}}}", "\shortmid"{marking}, from=4-1, to=4-2]
      \arrow["{\Psi_{\vec{y}}}"', from=4-1, to=5-1]
      \arrow["{\Psi_{\vec{y}}}", from=4-2, to=5-2]
      \arrow[""{name=3, anchor=center, inner sep=0}, "{\Pi G\id_{\vec{y}}}"', "\shortmid"{marking}, from=5-1, to=5-2]
      \arrow["{\alpha_{\Pi(f_0,f)}}"{description}, draw=none, from=0, to=1]
      \arrow["{G\Pi_{\vec{y}}}"{description, pos=0.4}, draw=none, from=1, to=2]
      \arrow["{\Psi_{\id_{\vec{y}}}}"{description}, draw=none, from=2, to=3]
    \end{tikzcd}
    =
    \begin{tikzcd}[column sep=large,row sep=scriptsize]
      {F\Pi\vec{x}} & {F\Pi\vec{x}} \\
      {\Pi F\vec{x}} & {\Pi F\vec{x}} \\
      {\Pi F\vec{x}} & {\Pi F\vec{x}} \\
      {\Pi G\vec{x}} & {\Pi F\vec{y}} \\
      {\Pi G\vec{y}} & {\Pi G\vec{y}}
      \arrow[""{name=0, anchor=center, inner sep=0}, "{\id_{F\Pi\vec{x}}}", "\shortmid"{marking}, from=1-1, to=1-2]
      \arrow["{\Phi_{\vec{x}}}"', from=1-1, to=2-1]
      \arrow["{\Phi_{\vec{x}}}", from=1-2, to=2-2]
      \arrow[""{name=1, anchor=center, inner sep=0}, "{\id_{\Pi F\vec{x}}}"', "\shortmid"{marking}, from=2-1, to=2-2]
      \arrow[""{name=2, anchor=center, inner sep=0}, "{\Pi \id_{F\vec{x}}}"', "\shortmid"{marking}, from=3-1, to=3-2]
      \arrow[Rightarrow, no head, from=2-1, to=3-1]
      \arrow[Rightarrow, no head, from=2-2, to=3-2]
      \arrow["{\Pi\alpha_{\vec{x}}}"', from=3-1, to=4-1]
      \arrow["{\Pi(f_0,Gf)}"', from=4-1, to=5-1]
      \arrow["{\Pi(f_0,Ff)}", from=3-2, to=4-2]
      \arrow["{\Pi\alpha_{\vec{y}}}", from=4-2, to=5-2]
      \arrow[""{name=3, anchor=center, inner sep=0}, "{\Pi G\id_{\vec{y}}}"', "\shortmid"{marking}, from=5-1, to=5-2]
      \arrow["{\Pi(f_0,\alpha_f)}"{description}, draw=none, from=2, to=3]
      \arrow["{\id_{\Phi_{\vec{x}}}}"{description}, draw=none, from=0, to=1]
      \arrow["{\Pi_{F\vec{x}}}"{description, pos=0.6}, draw=none, from=1, to=2]
    \end{tikzcd},
  \end{equation}
  where $\alpha_f$ denotes the family of naturality comparisons $\alpha_{f_j}$
  for the arrows $f_j: x_{f_0(j)} \to y_j$, $j \in J$.

  In particular, when $\dbl{D}$ has normal lax products and $G$ preserves
  products of identities, the naturality comparison for the product arrow
  $\Pi(f_0,f): \Pi\vec{x} \to \Pi\vec{y}$ is uniquely determined by the product
  $\Pi(f_0,\alpha_f)$ of naturality comparisons.
\end{lemma}
\begin{proof}
  Fixing an element $j \in J$, the post-composite of the left-hand side of
  Equation \eqref{eq:naturality-comparison-product} with the projection $\pi_j$ in
  $\dbl{E}$ is
  \begin{equation*}
    \begin{tikzcd}
      {F\Pi\vec{x}} & {F\Pi\vec{x}} \\
      {\Pi G\vec{y}} & {\Pi G\vec{y}} \\
      {Gy_j} & {Gy_j}
      \arrow[""{name=0, anchor=center, inner sep=0}, "{\id_{F\Pi\vec{x}}}", "\shortmid"{marking}, from=1-1, to=1-2]
      \arrow[""{name=1, anchor=center, inner sep=0}, "{\Pi G\id_{\vec{y}}}"', "\shortmid"{marking}, from=2-1, to=2-2]
      \arrow[from=1-1, to=2-1]
      \arrow[from=1-2, to=2-2]
      \arrow[""{name=2, anchor=center, inner sep=0}, "{G\id_{y_j}}"', "\shortmid"{marking}, from=3-1, to=3-2]
      \arrow["{\pi_j}", from=2-2, to=3-2]
      \arrow["{\pi_j}"', from=2-1, to=3-1]
      \arrow["{\text{LHS}}"{description}, draw=none, from=0, to=1]
      \arrow["{\pi_j}"{description, pos=0.6}, draw=none, from=1, to=2]
    \end{tikzcd}
    \quad=\quad
    \begin{tikzcd}
      {F\Pi\vec{x}} & {F\Pi\vec{x}} \\
      {G\Pi\vec{y}} & {G\Pi\vec{y}} \\
      {G\Pi\vec{y}} & {G\Pi\vec{y}} \\
      {G y_j} & {G y_j}
      \arrow[""{name=0, anchor=center, inner sep=0}, "{\id_{F\Pi\vec{x}}}", "\shortmid"{marking}, from=1-1, to=1-2]
      \arrow[""{name=1, anchor=center, inner sep=0}, "{G\id_{\Pi\vec{y}}}", "\shortmid"{marking}, from=2-1, to=2-2]
      \arrow[Rightarrow, no head, from=2-1, to=3-1]
      \arrow[Rightarrow, no head, from=2-2, to=3-2]
      \arrow[""{name=2, anchor=center, inner sep=0}, "{G\Pi\id_{\vec{y}}}", "\shortmid"{marking}, from=3-1, to=3-2]
      \arrow[from=1-1, to=2-1]
      \arrow[from=1-2, to=2-2]
      \arrow["{G\pi_j}"', from=3-1, to=4-1]
      \arrow["{G\pi_j}", from=3-2, to=4-2]
      \arrow[""{name=3, anchor=center, inner sep=0}, "{G\id_{y_j}}"', "\shortmid"{marking}, from=4-1, to=4-2]
      \arrow["{\alpha_{\Pi(f_0,f)}}"{description, pos=0.4}, draw=none, from=0, to=1]
      \arrow["{G\Pi_{\vec{y}}}"{description, pos=0.4}, draw=none, from=1, to=2]
      \arrow["{G\pi_j}"{description}, draw=none, from=2, to=3]
    \end{tikzcd}
    \quad=\quad
    \begin{tikzcd}[column sep=large]
      {F\Pi\vec{x}} & {F\Pi\vec{x}} \\
      {G\Pi\vec{x}} & {F\Pi\vec{y}} \\
      {G\Pi\vec{y}} & {G\Pi\vec{y}} \\
      {Gy_j} & {Gy_j}
      \arrow[""{name=0, anchor=center, inner sep=0}, "{\id_{F\Pi\vec{x}}}", "\shortmid"{marking}, from=1-1, to=1-2]
      \arrow["{\alpha_{\Pi\vec{x}}}"', from=1-1, to=2-1]
      \arrow["{G\Pi(f_0,f)}"', from=2-1, to=3-1]
      \arrow["{F\Pi(f_0,f)}", from=1-2, to=2-2]
      \arrow["{\alpha_{\Pi\vec{y}}}", from=2-2, to=3-2]
      \arrow[""{name=1, anchor=center, inner sep=0}, "{G\id_{\Pi\vec{y}}}", "\shortmid"{marking}, from=3-1, to=3-2]
      \arrow["{G\pi_j}"', from=3-1, to=4-1]
      \arrow["{G\pi_j}", from=3-2, to=4-2]
      \arrow[""{name=2, anchor=center, inner sep=0}, "{G\id_{y_j}}"', "\shortmid"{marking}, from=4-1, to=4-2]
      \arrow["{\alpha_{\Pi(f_0,f)}}"{description}, draw=none, from=0, to=1]
      \arrow["{G\id_{\pi_j}}"{description}, draw=none, from=1, to=2]
    \end{tikzcd}.
  \end{equation*}
  Setting $i \coloneqq f_0(j) \in I$ and using the equations
  \begin{equation} \label{eq:product-arrow}
    \begin{tikzcd}
      {\Pi\vec{x}} & {\Pi\vec{y}} \\
      {x_{f_0(j)}} & {y_j}
      \arrow["{\Pi(f_0,f)}", from=1-1, to=1-2]
      \arrow["{\pi_j}", from=1-2, to=2-2]
      \arrow["{f_j}"', from=2-1, to=2-2]
      \arrow["{\pi_i = \pi_{f_0(j)}}"', from=1-1, to=2-1]
    \end{tikzcd}
    \qquad\text{and}\qquad
    \begin{tikzcd}
      {\Pi\id_{F\vec{x}}} & {\Pi G\id_{\vec{y}}} \\
      {\id_{Fx_i}} & {G \id_{y_j}}
      \arrow["{\Pi(f_0,\alpha_f)}", from=1-1, to=1-2]
      \arrow["{\pi_j}", from=1-2, to=2-2]
      \arrow["{\alpha_{f_j}}"', from=2-1, to=2-2]
      \arrow["{\pi_i = \pi_{f_0(j)}}"', from=1-1, to=2-1]
    \end{tikzcd}
  \end{equation}
  in $\dbl{E}_0$ and $\dbl{E}_1$, the post-composite of the right-hand side of
  Equation \eqref{eq:naturality-comparison-product} with $\pi_j$ is
  \begin{equation*}
    \begin{tikzcd}
      {F\Pi\vec{x}} & {F\Pi\vec{x}} \\
      {\Pi G\vec{y}} & {\Pi G\vec{y}} \\
      {Gy_j} & {Gy_j}
      \arrow[""{name=0, anchor=center, inner sep=0}, "{\id_{F\Pi\vec{x}}}", "\shortmid"{marking}, from=1-1, to=1-2]
      \arrow[""{name=1, anchor=center, inner sep=0}, "{\Pi G\id_{\vec{y}}}"', "\shortmid"{marking}, from=2-1, to=2-2]
      \arrow[from=1-1, to=2-1]
      \arrow[from=1-2, to=2-2]
      \arrow[""{name=2, anchor=center, inner sep=0}, "{G\id_{y_j}}"', "\shortmid"{marking}, from=3-1, to=3-2]
      \arrow["{\pi_j}", from=2-2, to=3-2]
      \arrow["{\pi_j}"', from=2-1, to=3-1]
      \arrow["{\text{RHS}}"{description}, draw=none, from=0, to=1]
      \arrow["{\pi_j}"{description, pos=0.6}, draw=none, from=1, to=2]
    \end{tikzcd}
    =
    \begin{tikzcd}[row sep=scriptsize]
      {F\Pi\vec{x}} & {F\Pi\vec{x}} \\
      {\Pi F\vec{x}} & {\Pi F\vec{x}} \\
      {\Pi F\vec{x}} & {\Pi F\vec{x}} \\
      {Fx_i} & {Fx_i} \\
      {Gy_j} & {Gy_j}
      \arrow[""{name=0, anchor=center, inner sep=0}, "{\id_{F\Pi\vec{x}}}", "\shortmid"{marking}, from=1-1, to=1-2]
      \arrow["{\Phi_{\vec{x}}}"', from=1-1, to=2-1]
      \arrow["{\Phi_{\vec{x}}}", from=1-2, to=2-2]
      \arrow[""{name=1, anchor=center, inner sep=0}, "{\id_{\Pi F\vec{x}}}"', "\shortmid"{marking}, from=2-1, to=2-2]
      \arrow[""{name=2, anchor=center, inner sep=0}, "{\Pi \id_{F\vec{x}}}"', "\shortmid"{marking}, from=3-1, to=3-2]
      \arrow[Rightarrow, no head, from=2-1, to=3-1]
      \arrow[Rightarrow, no head, from=2-2, to=3-2]
      \arrow[from=4-1, to=5-1]
      \arrow[from=4-2, to=5-2]
      \arrow[""{name=3, anchor=center, inner sep=0}, "{\id_{Fx_i}}"', "\shortmid"{marking}, from=4-1, to=4-2]
      \arrow[""{name=4, anchor=center, inner sep=0}, "{G\id_{y_j}}"', "\shortmid"{marking}, from=5-1, to=5-2]
      \arrow["{\pi_i}"', from=3-1, to=4-1]
      \arrow["{\pi_i}", from=3-2, to=4-2]
      \arrow["{\id_{\Phi_{\vec{x}}}}"{description}, draw=none, from=0, to=1]
      \arrow["{\Pi_{F\vec{x}}}"{description, pos=0.6}, draw=none, from=1, to=2]
      \arrow["{\pi_i}"{description, pos=0.6}, draw=none, from=2, to=3]
      \arrow["{\alpha_{f_j}}"{description, pos=0.6}, draw=none, from=3, to=4]
    \end{tikzcd}
    =
    \begin{tikzcd}
      {F\Pi\vec{x}} & {F\Pi\vec{x}} \\
      {\Pi F\vec{x}} & {\Pi F\vec{x}} \\
      {Fx_i} & {Fx_i} \\
      {Gy_j} & {Gy_j}
      \arrow[""{name=0, anchor=center, inner sep=0}, "{\id_{F\Pi\vec{x}}}", "\shortmid"{marking}, from=1-1, to=1-2]
      \arrow["{\Phi_{\vec{x}}}"', from=1-1, to=2-1]
      \arrow["{\Phi_{\vec{x}}}", from=1-2, to=2-2]
      \arrow[""{name=1, anchor=center, inner sep=0}, "{\id_{\Pi F\vec{x}}}"', "\shortmid"{marking}, from=2-1, to=2-2]
      \arrow[from=3-1, to=4-1]
      \arrow[from=3-2, to=4-2]
      \arrow[""{name=2, anchor=center, inner sep=0}, "{\id_{Fx_i}}"', "\shortmid"{marking}, from=3-1, to=3-2]
      \arrow[""{name=3, anchor=center, inner sep=0}, "{G\id_{y_j}}"', "\shortmid"{marking}, from=4-1, to=4-2]
      \arrow["{\pi_i}"', from=2-1, to=3-1]
      \arrow["{\pi_i}", from=2-2, to=3-2]
      \arrow["{\id_{\Phi_{\vec{x}}}}"{description}, draw=none, from=0, to=1]
      \arrow["{\alpha_{f_j}}"{description, pos=0.6}, draw=none, from=2, to=3]
      \arrow["{\id_{\pi_i}}"{description, pos=0.6}, draw=none, from=1, to=2]
    \end{tikzcd}
    =
    \begin{tikzcd}
      {F\Pi\vec{x}} & {F\Pi\vec{x}} \\
      {Fx_i} & {Fx_i} \\
      {Gx_i} & {Fy_j} \\
      {Gy_j} & {Gy_j}
      \arrow[""{name=0, anchor=center, inner sep=0}, "{\id_{F\Pi\vec{x}}}", "\shortmid"{marking}, from=1-1, to=1-2]
      \arrow[""{name=1, anchor=center, inner sep=0}, "{\id_{Fx_i}}"', "\shortmid"{marking}, from=2-1, to=2-2]
      \arrow[""{name=2, anchor=center, inner sep=0}, "{G\id_{y_j}}"', "\shortmid"{marking}, from=4-1, to=4-2]
      \arrow["{F\pi_i}"', from=1-1, to=2-1]
      \arrow["{F\pi_i}", from=1-2, to=2-2]
      \arrow["{\alpha_{x_i}}"', from=2-1, to=3-1]
      \arrow["{Gf_j}"', from=3-1, to=4-1]
      \arrow["{Ff_j}", from=2-2, to=3-2]
      \arrow["{\alpha_{y_j}}", from=3-2, to=4-2]
      \arrow["{\id_{F\pi_i}}"{description}, draw=none, from=0, to=1]
      \arrow["{\alpha_{f_j}}"{description}, draw=none, from=1, to=2]
    \end{tikzcd}.
  \end{equation*}
  But, applying the functorality of the naturality comparisons to the first
  square in Equation \eqref{eq:product-arrow} and using the assumption that $\alpha$ is
  strictly natural with respect to projections, we also have
  \begin{equation*}
    \begin{tikzcd}[column sep=large]
      {F\Pi\vec{x}} & {F\Pi\vec{x}} \\
      {G\Pi\vec{x}} & {F\Pi\vec{y}} \\
      {G\Pi\vec{y}} & {G\Pi\vec{y}} \\
      {Gy_j} & {Gy_j}
      \arrow[""{name=0, anchor=center, inner sep=0}, "{\id_{F\Pi\vec{x}}}", "\shortmid"{marking}, from=1-1, to=1-2]
      \arrow["{\alpha_{\Pi\vec{x}}}"', from=1-1, to=2-1]
      \arrow["{G\Pi(f_0,f)}"', from=2-1, to=3-1]
      \arrow["{F\Pi(f_0,f)}", from=1-2, to=2-2]
      \arrow["{\alpha_{\Pi\vec{y}}}", from=2-2, to=3-2]
      \arrow[""{name=1, anchor=center, inner sep=0}, "{G\id_{\Pi\vec{y}}}", "\shortmid"{marking}, from=3-1, to=3-2]
      \arrow["{G\pi_j}"', from=3-1, to=4-1]
      \arrow["{G\pi_j}", from=3-2, to=4-2]
      \arrow[""{name=2, anchor=center, inner sep=0}, "{G\id_{y_j}}"', "\shortmid"{marking}, from=4-1, to=4-2]
      \arrow["{\alpha_{\Pi(f_0,f)}}"{description}, draw=none, from=0, to=1]
      \arrow["{G\id_{\pi_j}}"{description}, draw=none, from=1, to=2]
    \end{tikzcd}
    \quad=\quad
    \alpha_{\pi_j \circ \Pi(f_0,f)}
    \;=\;
    \alpha_{f_j \circ \pi_i}
    \quad=\quad
    \begin{tikzcd}
      {F\Pi\vec{x}} & {F\Pi\vec{x}} \\
      {Fx_i} & {Fx_i} \\
      {Gx_i} & {Fy_j} \\
      {Gy_j} & {Gy_j}
      \arrow[""{name=0, anchor=center, inner sep=0}, "{\id_{F\Pi\vec{x}}}", "\shortmid"{marking}, from=1-1, to=1-2]
      \arrow[""{name=1, anchor=center, inner sep=0}, "{\id_{Fx_i}}"', "\shortmid"{marking}, from=2-1, to=2-2]
      \arrow[""{name=2, anchor=center, inner sep=0}, "{G\id_{y_j}}"', "\shortmid"{marking}, from=4-1, to=4-2]
      \arrow["{F\pi_i}"', from=1-1, to=2-1]
      \arrow["{F\pi_i}", from=1-2, to=2-2]
      \arrow["{\alpha_{x_i}}"', from=2-1, to=3-1]
      \arrow["{Gf_j}"', from=3-1, to=4-1]
      \arrow["{Ff_j}", from=2-2, to=3-2]
      \arrow["{\alpha_{y_j}}", from=3-2, to=4-2]
      \arrow["{\id_{F\pi_i}}"{description}, draw=none, from=0, to=1]
      \arrow["{\alpha_{f_j}}"{description}, draw=none, from=1, to=2]
    \end{tikzcd}.
  \end{equation*}
  We have proved that the post-composite of
  Equation \eqref{eq:naturality-comparison-product} with any projection $\pi_j$ is true,
  hence the equation itself is true by the universal property of the product.
\end{proof}

We now turn to modules between lax double functors, introduced by Paré while
developing the Yoneda theory of double categories \citep{pare2011}. A
\define{module} between lax functors $F,G: \dbl{D} \to \dbl{E}$ can be
succinctly defined as a lax functor $M: \dbl{D} \times \dbl{I} \to \dbl{E}$,
where $\dbl{I} \coloneqq \{0 \proto 1\}$ is the walking proarrow, such that
$M(-,0) = F$ and $M(-,1) = G$. The definition is fully articulated in
\citep[\S{3.1}]{pare2011} or \citep[\mbox{Definition 9.1}]{lambert2024}. Since a
module is a kind of lax functor, a product-preserving module can be defined very
similarly to a product-preserving lax functor
(Definition \ref{def:preserve-dbl-products}), extending our previous definition of a
cartesian module \citep[\mbox{Definition 9.4}]{lambert2024}.

\begin{definition}[Product-preserving module]
  Let $F,G: \dbl{D} \to \dbl{E}$ be lax functors between double categories with
  lax products. A module $M: F \proTo G$ \define{preserves products} if, for
  every family of proarrows $(A,\vec{m}): (I,\vec{x}) \proto (J,\vec{y})$ in
  $\dbl{D}$, the canonical comparison cell
  \begin{equation*}
    \begin{tikzcd}
      {F\Pi\vec{x}} & {G\Pi\vec{y}} \\
      {\Pi F\vec{x}} & {\Pi G\vec{y}}
      \arrow[""{name=0, anchor=center, inner sep=0}, "{M\Pi\vec{m}}", "\shortmid"{marking}, from=1-1, to=1-2]
      \arrow["{\Phi_{\vec{x}}}"', from=1-1, to=2-1]
      \arrow["{\Psi_{\vec{y}}}", from=1-2, to=2-2]
      \arrow[""{name=1, anchor=center, inner sep=0}, "{\Pi M\vec{m}}"', "\shortmid"{marking}, from=2-1, to=2-2]
      \arrow["{\Mu_{\vec{m}}}"{description}, draw=none, from=0, to=1]
    \end{tikzcd},
  \end{equation*}
  defined by the equations
  \begin{equation} \label{eq:module-product-comparison}
    \begin{tikzcd}
      {F\Pi\vec{x}} & {G\Pi\vec{y}} \\
      {\Pi F\vec{x}} & {\Pi G\vec{y}} \\
      {Fx_i} & {Gy_j}
      \arrow[""{name=0, anchor=center, inner sep=0}, "{M\Pi\vec{m}}", "\shortmid"{marking}, from=1-1, to=1-2]
      \arrow["{\Phi_{\vec{x}}}"', from=1-1, to=2-1]
      \arrow["{\Psi_{\vec{y}}}", from=1-2, to=2-2]
      \arrow[""{name=1, anchor=center, inner sep=0}, "{\Pi M\vec{m}}"', "\shortmid"{marking}, from=2-1, to=2-2]
      \arrow["{\pi_i}"', from=2-1, to=3-1]
      \arrow["{\pi_j}", from=2-2, to=3-2]
      \arrow[""{name=2, anchor=center, inner sep=0}, "{Mm_a}"', "\shortmid"{marking}, from=3-1, to=3-2]
      \arrow["{\Mu_{\vec{m}}}"{description}, draw=none, from=0, to=1]
      \arrow["{\pi_a}"{description, pos=0.6}, draw=none, from=1, to=2]
    \end{tikzcd}
    \quad=\quad
    \begin{tikzcd}
      {F\Pi\vec{x}} & {G\Pi\vec{y}} \\
      {Fx_i} & {Gy_j}
      \arrow[""{name=0, anchor=center, inner sep=0}, "{M\Pi\vec{m}}", "\shortmid"{marking}, from=1-1, to=1-2]
      \arrow[""{name=1, anchor=center, inner sep=0}, "{Mm_a}"', "\shortmid"{marking}, from=2-1, to=2-2]
      \arrow["{F\pi_i}"', from=1-1, to=2-1]
      \arrow["{G\pi_j}", from=1-2, to=2-2]
      \arrow["{M\pi_a}"{description}, draw=none, from=0, to=1]
    \end{tikzcd},
    \qquad (i \xproto{a} j) : (I \xproto{A} J),
  \end{equation}
  is an isomorphism in $\dbl{E}_1$.
\end{definition}

For a module $M: F \proTo G$ between lax double functors to preserve products,
it is clearly necessary that the underlying functors
$F_0,G_0: \dbl{D}_0 \to \dbl{E}_0$ preserve products. One typically assumes that
the double functors $F,G: \dbl{D} \to \dbl{E}$ in their entirety preserve
products.

A square of compatible transformations and modules are bound together by a
\emph{modulation}, defined for strict transformations in
\citep[\S{3.3}]{pare2011} and for lax transformations in \citep[\mbox{Definition
  9.7}]{lambert2024}. As a top-dimensional cell, modulations automatically
preserve products.

\begin{lemma}[Modulation components for products]
  \label{lem:modulation-component-product}
  Let $\dbl{D}$ and $\dbl{E}$ be double categories with lax products, let
  $F,G,H,K: \dbl{D} \to \dbl{E}$ be lax functors, let $\alpha: F \To H$ and
  $\beta: G \To K$ be product-preserving lax transformations, let
  $M: F \proTo G$ and $N: H \proTo K$ be modules, and finally let
  \begin{equation*}
    \begin{tikzcd}
      F & G \\
      H & K
      \arrow[""{name=0, anchor=center, inner sep=0}, "M", "\shortmid"{marking}, from=1-1, to=1-2]
      \arrow["\alpha"', from=1-1, to=2-1]
      \arrow["\beta", from=1-2, to=2-2]
      \arrow[""{name=1, anchor=center, inner sep=0}, "N"', "\shortmid"{marking}, from=2-1, to=2-2]
      \arrow["\mu"{description}, draw=none, from=0, to=1]
    \end{tikzcd}
  \end{equation*}
  be a modulation. Then, for any family of proarrows
  $(A,\vec{m}): (I,\vec{x}) \proto (J,\vec{y})$ in $\dbl{D}$, we have
  \begin{equation} \label{eq:modulation-component-product}
    \begin{tikzcd}
      {F\Pi\vec{x}} & {G\Pi\vec{y}} \\
      {H\Pi\vec{x}} & {K\Pi\vec{y}} \\
      {\Pi H\vec{x}} & {\Pi K\vec{y}}
      \arrow["{\alpha_{\Pi\vec{x}}}"', from=1-1, to=2-1]
      \arrow[""{name=0, anchor=center, inner sep=0}, "{M\Pi\vec{m}}", "\shortmid"{marking}, from=1-1, to=1-2]
      \arrow["{\beta_{\Pi\vec{y}}}", from=1-2, to=2-2]
      \arrow[""{name=1, anchor=center, inner sep=0}, "{N\Pi\vec{m}}", "\shortmid"{marking}, from=2-1, to=2-2]
      \arrow["{\Eta_{\vec{x}}}"', from=2-1, to=3-1]
      \arrow[""{name=2, anchor=center, inner sep=0}, "{\Pi N\vec{m}}"', "\shortmid"{marking}, from=3-1, to=3-2]
      \arrow["{\Kappa_{\vec{y}}}", from=2-2, to=3-2]
      \arrow["{\mu_{\Pi\vec{m}}}"{description, pos=0.4}, draw=none, from=0, to=1]
      \arrow["{\Nu_{\vec{m}}}"{description}, draw=none, from=1, to=2]
    \end{tikzcd}
    \quad=\quad
    \begin{tikzcd}
      {F\Pi\vec{x}} & {G\Pi\vec{y}} \\
      {\Pi F\vec{x}} & {\Pi G\vec{y}} \\
      {\Pi H\vec{x}} & {\Pi K\vec{y}}
      \arrow[""{name=0, anchor=center, inner sep=0}, "{M\Pi\vec{m}}", "\shortmid"{marking}, from=1-1, to=1-2]
      \arrow["{\Phi_{\vec{x}}}"', from=1-1, to=2-1]
      \arrow["{\Psi_{\vec{y}}}", from=1-2, to=2-2]
      \arrow[""{name=1, anchor=center, inner sep=0}, "{\Pi M\vec{m}}", "\shortmid"{marking}, from=2-1, to=2-2]
      \arrow["{\Pi\alpha_{\vec{x}}}"', from=2-1, to=3-1]
      \arrow["{\Pi\beta_{\vec{y}}}", from=2-2, to=3-2]
      \arrow[""{name=2, anchor=center, inner sep=0}, "{\Pi N\vec{m}}"', "\shortmid"{marking}, from=3-1, to=3-2]
      \arrow["{\Mu_{\vec{m}}}"{description, pos=0.4}, draw=none, from=0, to=1]
      \arrow["{\Pi\mu_{\vec{m}}}"{description}, draw=none, from=1, to=2]
    \end{tikzcd},
  \end{equation}
  where $(1_A,\mu_{\vec{m}})$ is the cell in $\DblFamOp(\dbl{E})$ with
  components
  \begin{equation*}
    \begin{tikzcd}
      {Fx_i} & {Gy_j} \\
      {Hx_i} & {Ky_j}
      \arrow[""{name=0, anchor=center, inner sep=0}, "{Mm_a}", "\shortmid"{marking}, from=1-1, to=1-2]
      \arrow["{\alpha_{x_i}}"', from=1-1, to=2-1]
      \arrow["{\beta_{y_j}}", from=1-2, to=2-2]
      \arrow[""{name=1, anchor=center, inner sep=0}, "{Nm_a}"', "\shortmid"{marking}, from=2-1, to=2-2]
      \arrow["{\mu_{m_a}}"{description}, draw=none, from=0, to=1]
    \end{tikzcd},
    \qquad (i \xproto{a} j) : (I \xproto{A} J).
  \end{equation*}
  In particular, when the module $N: H \proTo K$ preserves products, the
  component of the modulation $\mu$ at the product $\Pi\vec{m}$ is uniquely
  determined by the product of the components at $m_a$ for each $a \in A$.
\end{lemma}
\begin{proof}
  Fixing an element $a: i \proto j$ of the indexing span, the naturality of the
  modulation $\mu$ applied to the projection cell $\pi_a: \Pi\vec{m} \to m_a$ in
  $\dbl{D}$ is the equation
  \begin{equation*}
    \begin{tikzcd}
      {F\Pi\vec{x}} & {G\Pi\vec{y}} \\
      {H\Pi\vec{x}} & {K\Pi\vec{y}} \\
      {Hx_i} & {K y_j}
      \arrow["{\alpha_{\Pi\vec{x}}}"', from=1-1, to=2-1]
      \arrow[""{name=0, anchor=center, inner sep=0}, "{M\Pi\vec{m}}", "\shortmid"{marking}, from=1-1, to=1-2]
      \arrow["{\beta_{\Pi\vec{y}}}", from=1-2, to=2-2]
      \arrow[""{name=1, anchor=center, inner sep=0}, "{N\Pi\vec{m}}", "\shortmid"{marking}, from=2-1, to=2-2]
      \arrow["{H\pi_i}"', from=2-1, to=3-1]
      \arrow["{K\pi_j}", from=2-2, to=3-2]
      \arrow[""{name=2, anchor=center, inner sep=0}, "{N m_a}"', "\shortmid"{marking}, from=3-1, to=3-2]
      \arrow["{\mu_{\Pi\vec{m}}}"{description, pos=0.4}, draw=none, from=0, to=1]
      \arrow["{N\pi_a}"{description}, draw=none, from=1, to=2]
    \end{tikzcd}
    \quad=\quad
    \begin{tikzcd}
      {F\Pi\vec{x}} & {G\Pi\vec{y}} \\
      {Fx_i} & {Gy_j} \\
      {Hx_i} & {Ky_j}
      \arrow[""{name=0, anchor=center, inner sep=0}, "{M\Pi\vec{m}}", "\shortmid"{marking}, from=1-1, to=1-2]
      \arrow["{F\pi_i}"', from=1-1, to=2-1]
      \arrow["{\alpha_{x_i}}"', from=2-1, to=3-1]
      \arrow["{\beta_{y_j}}", from=2-2, to=3-2]
      \arrow["{G\pi_j}", from=1-2, to=2-2]
      \arrow[""{name=1, anchor=center, inner sep=0}, "{Mm_a}", "\shortmid"{marking}, from=2-1, to=2-2]
      \arrow[""{name=2, anchor=center, inner sep=0}, "{Nm_a}"', "\shortmid"{marking}, from=3-1, to=3-2]
      \arrow["{M\pi_a}"{description, pos=0.4}, draw=none, from=0, to=1]
      \arrow["{\mu_{m_a}}"{description}, draw=none, from=1, to=2]
    \end{tikzcd},
  \end{equation*}
  where we have also used the assumption that the lax transformations $\alpha$
  and $\beta$ preserve products. But, by Equation \eqref{eq:module-product-comparison},
  the equation above is also the post-composite of
  Equation \eqref{eq:modulation-component-product} with the projection cell
  $\pi_a: \Pi N\vec{m} \to N m_a$ in $\dbl{E}$. So, by the universal property of
  the product $\Pi N\vec{m}$, the claimed Equation \eqref{eq:modulation-component-product}
  holds.
\end{proof}

Since lax double functors are, in general, the objects of a merely virtual
double category, we must also consider the multi-input generalization of a
modulation. \emph{Multimodulations} were defined for strict transformations in
\citep[\S{4.1}]{pare2011} and for lax transformations in \citep[\mbox{Definition
  10.1}]{lambert2024}. The following lemma on multimodulations is a
straightforward generalization of Lemma \ref{lem:modulation-component-product} on
modulations. We include it not just for completeness but to emphasize the role
by played by strong products in the domain, irrelevant in the unary case.

\begin{lemma}[Multimodulation components for products]
  Let $\dbl{D}$ and $\dbl{E}$ be double categories with lax products and, for
  any $k \geq 0$, let $F_0,\dots,F_k$, $G$, and $H$ be lax functors from
  $\dbl{D}$ to $\dbl{E}$. Suppose given a multimodulation
  \begin{equation*}
    \begin{tikzcd}
      {F_0} & \cdots & {F_k} \\
      G && H
      \arrow["{M_1}", "\shortmid"{marking}, from=1-1, to=1-2]
      \arrow["{M_k}", "\shortmid"{marking}, from=1-2, to=1-3]
      \arrow["\alpha"', from=1-1, to=2-1]
      \arrow[""{name=0, anchor=center, inner sep=0}, "N"', "\shortmid"{marking}, from=2-1, to=2-3]
      \arrow["\beta", from=1-3, to=2-3]
      \arrow["\mu"{description, pos=0.4}, draw=none, from=1-2, to=0]
    \end{tikzcd},
  \end{equation*}
  where $\alpha$ and $\beta$ are product-preserving lax natural transformations.
  Then, for any length-$k$ sequence of proarrow families
  $(I_0,\vec{x}_0) \xproto{(A_1,\vec{m}_1)} \cdots \xproto{(A_k,\vec{m}_k)} (I_k,\vec{x}_k)$
  in $\dbl{D}$, we have
  \begin{equation} \label{eq:multimodulation-component-product}
    \begin{tikzcd}
      {F_0 \Pi \vec{x}_0} & \cdots & {F_k \Pi\vec{x}_k} \\
      {G\Pi\vec{x}_0} && {H\Pi\vec{x}_k} \\
      {G\Pi\vec{x}_0} && {H\Pi\vec{x}_k} \\
      {\Pi G\vec{x}_0} && {\Pi H\vec{x}_k}
      \arrow["{M_1 \Pi\vec{m}_1}", "\shortmid"{marking}, from=1-1, to=1-2]
      \arrow["{\alpha_{\Pi\vec{x}_0}}"', from=1-1, to=2-1]
      \arrow["{M_k \Pi\vec{m}_k}", "\shortmid"{marking}, from=1-2, to=1-3]
      \arrow["{\beta_{\Pi\vec{x}_k}}", from=1-3, to=2-3]
      \arrow[""{name=0, anchor=center, inner sep=0}, "{N(\Pi\vec{m}_1 \odot\cdots\odot \Pi\vec{m}_k)}"', "\shortmid"{marking}, from=2-1, to=2-3]
      \arrow[""{name=1, anchor=center, inner sep=0}, "{\Pi N(\vec{m}_1 \odot\cdots\odot \vec{m}_k)}"', "\shortmid"{marking}, from=4-1, to=4-3]
      \arrow[Rightarrow, no head, from=2-1, to=3-1]
      \arrow[Rightarrow, no head, from=2-3, to=3-3]
      \arrow[""{name=2, anchor=center, inner sep=0}, "{N\Pi(\vec{m}_1 \odot\cdots\odot \vec{m}_k)}"', "\shortmid"{marking}, from=3-1, to=3-3]
      \arrow["{\Psi_{\vec{x}_0}}"', from=3-1, to=4-1]
      \arrow["{\Eta_{\vec{x}_k}}", from=3-3, to=4-3]
      \arrow["{\mu_{\Pi\vec{m}_1,\dots,\Pi\vec{m}_k}}"{description, pos=0.4}, draw=none, from=1-2, to=0]
      \arrow["{N \Pi_{\vec{m}_1,\dots,\vec{m}_k}}"{description, pos=0.6}, draw=none, from=0, to=2]
      \arrow["{\Nu_{\vec{m}_1 \odot\cdots\odot \vec{m}_k}}"{description, pos=0.6}, draw=none, from=2, to=1]
    \end{tikzcd}
    \quad=\quad
    \begin{tikzcd}
      {F_0 \Pi \vec{x}_0} & \cdots & {F_k \Pi\vec{x}_k} \\
      {\Pi F_0 \vec{x}_0} & \cdots & {\Pi F_k \vec{x}_k} \\
      {\Pi F_0 \vec{x}_0} && {\Pi F_k \vec{x}_k} \\
      {\Pi G\vec{x}_0} && {\Pi H\vec{x}_0}
      \arrow[""{name=0, anchor=center, inner sep=0}, "{M_1 \Pi\vec{m}_1}"{inner sep=.8ex}, "\shortmid"{marking}, from=1-1, to=1-2]
      \arrow["{(\Phi_0)_{\vec{x}_0}}"', from=1-1, to=2-1]
      \arrow[""{name=1, anchor=center, inner sep=0}, "{M_k \Pi\vec{m}_k}"{inner sep=.8ex}, "\shortmid"{marking}, from=1-2, to=1-3]
      \arrow["{(\Phi_k)_{\vec{x}_k}}", from=1-3, to=2-3]
      \arrow[""{name=2, anchor=center, inner sep=0}, "{\Pi M_1 \vec{m}_1}"'{inner sep=.8ex}, "\shortmid"{marking}, from=2-1, to=2-2]
      \arrow[equals, from=2-1, to=3-1]
      \arrow[""{name=3, anchor=center, inner sep=0}, "{\Pi M_k \vec{m}_k}"'{inner sep=.8ex}, "\shortmid"{marking}, from=2-2, to=2-3]
      \arrow[equals, from=2-3, to=3-3]
      \arrow[""{name=4, anchor=center, inner sep=0}, "{\Pi(M_1 \vec{m}_1 \odot\cdots\odot M_k \vec{m}_k)}"'{inner sep=.8ex}, "\shortmid"{marking}, from=3-1, to=3-3]
      \arrow["{\Pi\alpha_{\vec{x}_0}}"', from=3-1, to=4-1]
      \arrow["{\Pi\beta_{\vec{x}_k}}", from=3-3, to=4-3]
      \arrow[""{name=5, anchor=center, inner sep=0}, "{\Pi N(\vec{m}_1 \odot\cdots\odot \vec{m}_k)}"'{inner sep=.8ex}, "\shortmid"{marking}, from=4-1, to=4-3]
      \arrow["{(\Mu_1)_{\vec{m}_1}}"{description}, draw=none, from=0, to=2]
      \arrow["{(\Mu_k)_{\vec{m}_k}}"{description}, draw=none, from=1, to=3]
      \arrow["{\Pi_{M_1 \vec{m}_1, \dots, M_k \vec{m}_k}}"{description}, draw=none, from=2-2, to=4]
      \arrow["{\Pi\mu_{\vec{m}_1, \dots, \vec{m}_k}}"{description, pos=0.6}, draw=none, from=4, to=5]
    \end{tikzcd}.
  \end{equation}
  In particular, when the domain double category $\dbl{D}$ has strong products
  and the module $N$ preserves products, the component of the multimodulation
  $\mu$ at the products $\Pi\vec{m}_1,\dots,\Pi\vec{m}_k$ is uniquely determined
  by the product of the components at $m_{1,a_1}, \dots, m_{k,a_k}$ indexed by
  $i_0 \xproto{a_0} \cdots \xproto{a_k} i_k$.
\end{lemma}
\begin{proof}
  Fixing elements $i_0 \xproto{a_0} \cdots \xproto{a_k} i_k$ of the indexing
  spans $I_0 \xproto{A_0} \cdots \xproto{A_k} I_k$, the naturality of the
  multimodulation $\mu$ at the projection cells
  $\pi_{a_1}: \Pi\vec{m}_1 \to m_{1,a_1}$ through
  $\pi_{a_k}: \Pi\vec{m}_k \to m_{k,a_k}$ in $\dbl{D}$ is the equation
  \begin{equation*}
    \begin{tikzcd}
      {F_0 \Pi\vec{x}_0} & \cdots & {F_k \Pi\vec{x}_k} \\
      {G\Pi\vec{x}_0} && {H\Pi\vec{x}_k} \\
      {Gx_{0,i_0}} && {Hx_{k,i_k}}
      \arrow["{M_1 \Pi\vec{m}_1}", "\shortmid"{marking}, from=1-1, to=1-2]
      \arrow["{M_k \Pi\vec{m}_k}", "\shortmid"{marking}, from=1-2, to=1-3]
      \arrow[""{name=0, anchor=center, inner sep=0}, "{N(m_{1,a_1} \odot\cdots\odot m_{k,a_k})}"', "\shortmid"{marking}, from=3-1, to=3-3]
      \arrow[""{name=1, anchor=center, inner sep=0}, "{N(\Pi\vec{m}_1 \odot\cdots\odot \Pi\vec{m}_k)}"', "\shortmid"{marking}, from=2-1, to=2-3]
      \arrow["{\alpha_{\Pi\vec{x}_0}}"', from=1-1, to=2-1]
      \arrow["{\beta_{\Pi\vec{x}_k}}", from=1-3, to=2-3]
      \arrow["{G\pi_{i_0}}"', from=2-1, to=3-1]
      \arrow["{H\pi_{i_k}}", from=2-3, to=3-3]
      \arrow["{\mu_{\Pi\vec{m}_1,\dots,\Pi\vec{m}_k}}"{description, pos=0.4}, draw=none, from=1-2, to=1]
      \arrow["{N(\pi_{a_1} \odot\cdots\odot \pi_{a_k})}"{description, pos=0.6}, draw=none, from=1, to=0]
    \end{tikzcd}
    \quad=\quad
    \begin{tikzcd}
      {F_0 \Pi\vec{x}_0} & \cdots & {F_k \Pi\vec{x}_k} \\
      {F_0 x_{0,i_0}} & \cdots & {F_k x_{k,i_k}} \\
      {Gx_{0,i_0}} && {Hx_{k,i_k}}
      \arrow[""{name=0, anchor=center, inner sep=0}, "{M_1 m_{1,a_1}}"', "\shortmid"{marking}, from=2-1, to=2-2]
      \arrow[""{name=1, anchor=center, inner sep=0}, "{M_k m_{k,a_k}}"', "\shortmid"{marking}, from=2-2, to=2-3]
      \arrow["{F_0 \pi_{i_0}}"', from=1-1, to=2-1]
      \arrow[""{name=2, anchor=center, inner sep=0}, "{M_1 \Pi\vec{m}_1}", "\shortmid"{marking}, from=1-1, to=1-2]
      \arrow[""{name=3, anchor=center, inner sep=0}, "{M_k \Pi\vec{m}_k}", "\shortmid"{marking}, from=1-2, to=1-3]
      \arrow["{F_k \pi_{i_k}}", from=1-3, to=2-3]
      \arrow["{\alpha_{x_{0,i_0}}}"', from=2-1, to=3-1]
      \arrow[""{name=4, anchor=center, inner sep=0}, "{N(m_{1,a_1} \odot\cdots\odot m_{k,a_k})}"', "\shortmid"{marking}, from=3-1, to=3-3]
      \arrow["{\beta_{x_{k,a_k}}}", from=2-3, to=3-3]
      \arrow["{M_1 \pi_{a_1}}"{description}, draw=none, from=2, to=0]
      \arrow["{M_k \pi_{a_k}}"{description}, draw=none, from=3, to=1]
      \arrow["{\mu_{m_{1,a_1},\dots,m_{k,a_k}}}"{description}, draw=none, from=2-2, to=4]
    \end{tikzcd}.
  \end{equation*}
  But this equation is the post-composite of
  Equation \eqref{eq:multimodulation-component-product} with the projection cell
  \begin{equation*}
    \pi_{(a_1,\dots,a_k)}: \Pi N (\vec{m}_1 \odot\cdots\odot \vec{m}_k)
    \to N(m_{1,a_1} \odot\cdots\odot m_{k,a_k})
  \end{equation*}
  in $\dbl{E}$. By the universal property of the product,
  Equation \eqref{eq:multimodulation-component-product} holds.
\end{proof}

As anticipated, for two fixed double categories with lax products, the
product-preserving lax functors and higher morphisms between them assemble into
a unital virtual double category. The preceding lemmas are not needed to prove
this, only to ensure that the construction is well behaved. For the definitions
of virtual double categories and units in them, see \citep{cruttwell2010}.

\begin{theorem}[Virtual double category of product-preserving lax functors]
  Let $\dbl{D}$ and $\dbl{E}$ be double categories with lax products. There is a
  virtual double category with units whose
  \begin{itemize}[nosep]
    \item objects are product-preserving lax functors from $\dbl{D}$ to
      $\dbl{E}$,
    \item arrows are either (strict) natural transformations or
      product-preserving pseudo, lax, or oplax natural transformations,
    \item proarrows are product-preserving modules, and
    \item multicells are multimodulations.
  \end{itemize}
\end{theorem}
\begin{proof}
  For fixed double categories $\dbl{D}$ and $\dbl{E}$, there is a unital virtual
  double category of lax functors $\dbl{D} \to \dbl{E}$, natural
  transformations, modules, and multimodulations, for any choice of laxness in
  the natural transformations. This result was stated for strict transformations
  in \citep[\mbox{Theorem 4.3}]{pare2011} and proved for pseudo or (op)lax
  transformations in \citep[\mbox{Theorem 10.3}]{lambert2024}. So we need only
  check that the product-preserving arrows and proarrows are closed under the
  operations of a unital virtual double category. Product-preserving lax
  transformations are closed under composition, by pasting the naturality
  squares for projections, and the identity lax transformation arises from the
  identity strict transformation, so certainly preserves products. Finally, the
  identity module $\id_F$ on a lax functor $F$ preserves products if and only if
  $F$ does.
\end{proof}

\begin{corollary}[Virtual double category of models]
  Every finite-product double theory has a unital virtual double category of
  models, for any choice of strict, pseudo, or (op)lax maps of models.
\end{corollary}

For instance, the theory of local commutative monoids
(Example \ref{ex:cmon-enrichment}) gives rise to the unital virtual double category of
categories, functors, profunctors, and natural transformations, all enriched in
commutative monoids. In this case, the strict and lax maps of models coincide
since the only arrows in the theory are the structure arrows between products.
The verification of these claims uses the lemmas proved in this section. More
generally, for any category $\catV$ of models of a commutative algebraic theory,
Construction \ref{ex:algebraic-category-enrichment} gives the unital virtual
double category of $\catV$-categories, $\catV$-functors, $\catV$-profunctors,
and $\catV$-natural transformations.

\end{document}